\documentclass{amsart}

\usepackage[letterpaper,margin=1.1in]{geometry}

\usepackage{amsmath,amssymb,amsthm}
\usepackage{enumerate}
\usepackage{hyperref}
\usepackage{mathtools,pifont}
\usepackage{tikz-cd}
\usepackage{tikz}
\usepackage{mathrsfs}
\usepackage{xcolor}

\newtheorem{theorem}{Theorem}[section]
\newtheorem{observation}[theorem]{Observation}
\newtheorem{lemma}[theorem]{Lemma}
\newtheorem{proposition}[theorem]{Proposition}
\newtheorem{corollary}[theorem]{Corollary}
\newtheorem{question}[theorem]{Question}

\newtheorem*{theorem*}{Theorem}

\numberwithin{equation}{section}

\theoremstyle{definition}
\newtheorem{definition}[theorem]{Definition}
\newtheorem{notation}[theorem]{Notation}

\theoremstyle{remark}
\newtheorem{remark}[theorem]{Remark}
\newtheorem{example}[theorem]{Example}

\newcommand\R{\mathbb{R}}
\newcommand\C{\mathbb{C}}

\newcommand\N{\mathbb{N}}

\newcommand{\cA}{\mathcal{A}}
\newcommand{\cB}{\mathcal{B}}
\newcommand{\cC}{\mathcal{C}}
\newcommand{\cD}{\mathcal{D}}

\newcommand{\cM}{\mathcal{M}}

\newcommand{\cU}{\mathcal{U}}

\DeclareMathOperator{\id}{id}

\DeclareMathOperator{\re}{Re}
\DeclareMathOperator{\im}{Im}
\DeclareMathOperator{\Tr}{Tr}
\DeclareMathOperator{\tr}{tr}

\DeclareMathOperator{\sa}{sa}

\DeclareMathOperator{\app}{app}

\DeclareMathOperator{\dom}{dom}
\DeclareMathOperator{\Aff}{Aff}

\DeclareMathOperator{\opt}{opt}

\DeclareMathOperator{\Spec}{Spec}
\DeclareMathOperator{\fin}{fin}
\DeclareMathOperator{\bim}{bim}
\DeclareMathOperator{\sgn}{sgn}

\DeclarePairedDelimiter{\norm}{\lVert}{\rVert}
\DeclarePairedDelimiter{\ip}{\langle}{\rangle}

\begin{document}
	
	\title{Duality for optimal couplings in free probability}
	
	\author{Wilfrid Gangbo}
	\email{wgangbo@math.ucla.edu}
	\address{Department of Mathematics, University of California, Los Angeles, Los Angeles, CA 90095}
	
	\author{David Jekel}
	\email{djekel@ucsd.edu}
	\address{Department of Mathematics, University of California, San Diego, La Jolla, CA 92093}
	
	\author{Kyeongsik Nam}
	\email{ksnam@g.ucla.edu}
	\address{Department of Mathematics, University of California, Los Angeles, Los Angeles, CA 90095}
	
	\author{Dimitri Shlyakhtenko}
	\email{shlyakht@math.ucla.edu}
	\address{Department of Mathematics, University of California, Los Angeles, Los Angeles, CA 90095}
	
	\thanks{W.G.\ was supported by NSF grant DMS-1700202 and Air Force grant FA9550-18-1-0502.  D.J.\ was supported by NSF postdoc grant DMS-2002826.  D.S.\ was supported by NSF grant DMS-1762360.  We thank IPAM for the stimulating environment in the online long program on High-dimensional Hamilton-Jacobi PDEs in Spring 2020 where some conversations around this work began.  D.S. thanks Alice Guionnet and Yoann Dabrowski for multiple discussions about free transport; D.J. thanks Ben Hayes and Srivatsav Kunnawalkam Elayavalli for discussions about non-commutative laws and model theory, and in particular inspiration for Corollary \ref{cor:genericity}.  We thank the referees for many useful comments for improving the clarity of the paper.}
	
	\subjclass{Primary: 46L53; Secondary: 49Q22, 46L52, 81P45}
	
	\begin{abstract}
		We study the free probabilistic analog of optimal couplings for the quadratic cost, where classical probability spaces are replaced by tracial von Neumann algebras, and probability measures on $\R^m$ are replaced by non-commutative laws of $m$-tuples.  We prove an analog of the Monge-Kantorovich duality which characterizes optimal couplings of non-commutative laws with respect to Biane and Voiculescu's non-commutative $L^2$-Wasserstein distance using a new type of convex functions.  As a consequence, we show that if $(X,Y)$ is a pair of optimally coupled $m$-tuples of non-commutative random variables in a tracial $\mathrm{W}^*$-algebra $\cA$, then $\mathrm{W}^*((1 - t)X + tY) = \mathrm{W}^*(X,Y)$ for all $t \in (0,1)$.  Finally, we illustrate the subtleties of non-commutative optimal couplings through connections with results in quantum information theory and operator algebras.  For instance, two non-commutative laws that can be realized in finite-dimensional algebras may still require an infinite-dimensional algebra to optimally couple.  Moreover, the space of non-commutative laws of $m$-tuples is not separable with respect to the Wasserstein distance for $m > 1$.
	\end{abstract}
	
	\maketitle
	
	\section{Introduction}
	
	\subsection{Context and motivation}
	
	Tracial von Neumann algebras have long been viewed as a non-commutative analog of probability spaces, where the elements of the von Neumann algebra play the role of non-commuting random variables, but it was Voiculescu who pointed out that free products of operator algebras provide an analog of probabilistic independence with its own central limit theorem \cite{Voiculescu1985,Voiculescu1986}, initiating the discipline of free probability theory.  Free probability has since had many applications both to random matrix theory e.g.\ \cite{Voiculescu1991} and to von Neumann algebras e.g.\ \cite{VoiculescuFE3}.  Many developments in free probability theory have been motivated by information geometry (here by ``information geometry'' we mean the study of the space $\mathcal{P}(M)$ of probability measures on a manifold $M$, both as a metric space with the Wasserstein distance and as a formal Riemannian manifold, as well as the study of entropy and Fisher's information as functions on $\mathcal{P}(M)$; see \cite{JKO1998,Lafferty1988,Otto2001,OV2000}).  For instance, Voiculescu introduced free entropy and Fisher information \cite{VoiculescuFE1,VoiculescuFE2,VoiculescuFE5} and Biane and Voiculescu \cite{BV2001} defined an analog of the $L^p$ Wasserstein distance for non-commutative laws (the analog of probability distributions for $m$-tuples of non-commuting random variables), which was then used in free Talagrand inequalities \cite{BV2001,HPU2004,HU2006,Dabrowski2010}.
	
	Information-geometric ideas have also been used in quantum information theory, another non-commutative analog of probability theory that is distinct from free probability theory, even though it uses similar concepts and terminology.  For a survey of quantum information theory, see \cite{Wilde2013,Witten2020}.  To prevent any confusion, in free probability, operators in a tracial von Neumann algebra are viewed as non-commutative random variables (and there is no known analog of multivariable densities), while in quantum information theory, a positive operator with trace $1$ in a von Neumann algebra with a (not necessarily bounded) trace is viewed as a density.\footnote{More precisely, a positive operator $\rho$ defines a non-tracial state on the von Neumann algebra, and $\rho$ is the density of this state with respect to the trace.  However, von Neumann algebras with a semi-finite trace are difficult to classify, and indeed even those with a finite trace are difficult to classify, which makes it difficult to classify non-commutative laws in free probability.}  Hence, for example, a random matrix is typically studied in free probability theory, while a matrix-valued density is typically studied in quantum information theory.  Our paper is focused on the free probabilistic framework; however, in \S \ref{sec:qinfo}, we will draw a connection between free probabilistic optimal couplings and certain aspects of quantum information theory, specifically quantum channels or unital completely positive trace-preserving maps.
	
	In classical information geometry, both the Wasserstein distance and the entropy are intimately related to transport equations (differential equations describing functions which push forward some given probability distribution to another given probability distribution).  In the free setting, there has been some success in constructing non-commutative transport of measure for a special type of non-commutative law known as a \emph{free Gibbs law from a convex potential $V$} in \cite{DGS2016,GS2014,JekelExpectation,JekelThesis,JLS2021}; these ideas have even been generalized beyond the setting of tracial von Neumann algebras \cite{Nelson2015a,Nelson2015b,Shlyakhtenko2003}.  Unfortunately, the transport maps constructed in \cite{DGS2016,JekelExpectation,JekelThesis} were not optimal.  The transport in \cite{GS2014} was shown to be the gradient of a convex function, hence one would expect it to be optimal in light of the classical Monge-Kantorovich duality, but it was not clear yet how to prove this because there was no known non-commutative Monge-Kantorovich duality.  The optimality of these couplings was later verified in \cite[Remark 9.11]{JLS2021} by studying a Legendre transform for (sufficiently regular, uniformly convex) non-commutative functions \cite[Lemma 9.10]{JLS2021}.  This idea was one of the starting points for our current investigation into non-commutative optimal couplings, Legendre transforms, and Monge-Kantorovich duality with minimal regularity assumptions.
	
	One of the challenges in even formulating a Monge-Kantorovich duality for the free setting is to decide what type of convex functions to use.  Operator algebras are often thought of as non-commutative analogs of algebras of functions on a topological space or a measure space, but without a clear analog for points of the underlying space.  Our approach is to consider functions that can be evaluated on random variables rather than on points, or more precisely, to study functions $f: L^2(\cA)_{\sa}^m \to \R$ where $\cA$ is a tracial von Neumann algebra, $L^2(\cA)$ is the non-commutative $L^2$ space, and the subscript $\sa$ indicates the real subspace of self-adjoint elements.  The classical analog would be a function $L^2(\Omega,P;\R^m) \to \R$ where $(\Omega,P)$ is a probability space, rather than a function $\R^m \to \R$.  As we discuss in \S \ref{subsec:classicalcoupling}, such functions on the space of classical random variables have already found applications to Hamilton-Jacobi equations on the Wasserstein space \cite{GaTu2018,GMS2021} as well as the master equation on $\R^m \times \mathcal{P}(\R^m)$ in mean field games \cite{CDLL2019,GMMZ2021,GM2021}.
	
	As in \cite{GS2014} and \cite{JekelExpectation}, we remark that the complexity of classifying von Neumann algebras presents serious obstructions to non-commutative transport theory that simply do not exist in the classical setting.  It is a widely used fact in classical probability theory that any two standard Borel probability spaces with no atoms are measurably isomorphic; hence one can always arrange that their random variables are on some canonical probability space.  By contrast, McDuff \cite{McDuff1969} showed that there are uncountably many non-isomorphic tracial von Neumann algebras that are diffuse with trivial center (that is, $\mathrm{II}_1$ factors).  This provides a real obstruction to non-commutative transport of measure, because if $X = (X_1,\dots,X_m)$ and $Y = (Y_1,\dots,Y_m)$ are $m$-tuples of self-adjoint non-commutative random variables such that $X$ is expressed as a ``function'' of $Y$ and vice versa (for some reasonable notion of non-commutative functions), then $X$ and $Y$ generate the same von Neumann algebra.  Hence, non-commutative laws which produce non-isomorphic von Neumann algebras simply cannot be transported to each other in an invertible way.  Another result of Ozawa \cite{Ozawa2004} (based on group-theoretic results of Gromov \cite{Gromov1987} and Olshanskii \cite{Olshanskii1993}) shows there is no separable $\mathrm{II}_1$ factor that contains an isomorphic copy of every separable $\mathrm{II}_1$-factor.  Hence, we cannot even expect that there is some non-commutative law $\mu$ such that all other non-commutative laws can be expressed as push-forwards of $\mu$.
	
	These obstructions must inform how we go about defining the convex functions for the Monge-Kantorovich duality, as well as the level of regularity that we expect from an optimal coupling.  In fact, in \S \ref{sec:qinfo} we make a more explicit connection between optimal couplings and this result of Gromov, Olshanshkii, and Ozawa as well as exploring other pathological properties of the non-commutative Wasserstein distance through connections with quantum information theory.
	
	\subsection{Main results}
	
	Before stating the non-commutative Monge-Kantorovich duality, we establish following notational conventions; see \S \ref{sec:background} for background.  By \emph{tracial $\mathrm{W}^*$-algebra} we mean a pair $\cA = (A,\tau)$ where $A$ is a $\mathrm{W}^*$-algebra (or von Neumann algebra) and $\tau: A \to \C$ is a faithful normal tracial state.  In analogy with classical probability, we will denote the underlying algebra $A$ by $L^\infty(\cA)$ and the trace by $\tau_{\cA}$ when it is convenient to avoid naming $A$ and $\tau$ explicitly.  We denote by $L^2(\cA)$ the Hilbert space obtained from the GNS construction of $A$ and $\tau$.
	
	We denote by $L^\infty(\cA)_{\sa}^m$ the set of $m$-tuples of self-adjoint elements of $L^\infty(\cA)$ and for $X = (X_1,\dots,X_m) \in L^\infty(\cA)_{\sa}^m$, we write $\norm{X}_{L^\infty(\cA)_{\sa}^m} = \max_{j=1,\dots,m} \norm{X}_{L^\infty(\cA)}$.  If $X \in L^\infty(\cA)_{\sa}^m$, then $\mathrm{W}^*(X)$ denotes the $\mathrm{W}^*$-algebra generated by $X$ equipped with the appropriate trace.
	
	For each $X = (X_1,\dots,X_m) \in L^\infty(\cA)_{\sa}^m$, the \emph{non-commutative law} $\lambda_X$ is the linear map from the non-commutative polynomial algebra $\C\ip{x_1,\dots,x_m}$ to $\C$ given by $\lambda_X(p) = \tau_{\cA}(p(X))$.  The space of non-commutative laws (of self-adjoint $m$-tuples from any tracial $\mathrm{W}^*$-algebra) is denoted $\Sigma_m$.  Furthermore, $\Sigma_{m,R}$ denotes the subspace of those laws $\lambda_X$ where $\norm{X}_{L^\infty(\cA)_{\sa}^m} \leq R$ (where $\cA$ is a tracial $\mathrm{W}^*$-algebra and $X \in L^\infty(\cA)_{\sa}^m$).  The \emph{weak-$*$ topology on $\Sigma_{m,R}$} refers to the topology of pointwise convergence on $\C\ip{x_1,\dots,x_m}$.
	
	Following \cite{BV2001}, a \emph{coupling} of $\mu$, $\nu \in \Sigma_m$ is a triple $(\cA,X,Y)$ where $\cA$ is a tracial $\mathrm{W}^*$-algebra and $X, Y \in L^\infty(\cA)_{\sa}^m$ such that $\lambda_X = \mu$ and $\lambda_Y = \nu$.  The \emph{Wasserstein distance} $d_W^{(2)}(\mu,\nu)$ is the infimum of $\norm{X - Y}_{L^2(\cA)_{\sa}^m}$ over all couplings $(\cA,X,Y)$.  We denote by $C(\mu,\nu)$ the supremum of $\ip{X,Y}_{L^2(\cA)_{\sa}^m}$ over all couplings $(\cA,X,Y)$, where $\ip{X,Y}_{L^2(\cA)_{\sa}^m} = \sum_{j=1}^m \ip{X_j,Y_j}_{L^2(\cA)_{\sa}}$.  We say that a coupling is \emph{optimal} if it achieves the infimum of $\norm{X - Y}_{L^2(\cA)_{\sa}^m}$ or equivalently if it achieves the supremum of $\ip{X,Y}_{L^2(\cA)_{\sa}^m}$.  The existence of optimal couplings was observed in \cite{BV2001}.  That paper also showed that the non-commutative Wasserstein distance agrees with the classical one in the situation that $X_1$, \dots, $X_m$ commute and $Y_1$, \dots, $Y_m$ commute \cite[Theorem 1.5]{BV2001}.
	
	As mentioned before, the functions used in the non-commutative Monge-Kantorovich duality are functions on $L^2(\cA)_{\sa}^m$ for tracial $\mathrm{W}^*$-algebra $\cA$ with separable predual.  However, because of Ozawa's result \cite{Ozawa2004}, it is not sufficient to fix a single such tracial $\mathrm{W}^*$-algebra, but rather we must consider functions that are defined on $L^2(\cA)_{\sa}^m$ for every such $\cA$.  We give more precise versions of the definitions in \S \ref{sec:duality}.
	
	\begin{definition}
		A \emph{tracial $\mathrm{W}^*$-function with values in $(-\infty,\infty]$} is a collection of functions $f^{\cA}: L^2(\cA)_{\sa}^m \to (-\infty,+\infty]$, such that whenever $\iota: \cA \to \cB$ is an inclusion map of tracial $\mathrm{W}^*$-algebras, $f^{\cA} = f^{\cB} \circ \iota$ (here $\iota$ is extended to a map $L^2(\cA)_{\sa}^m \to L^2(\cB)_{\sa}^m$).  If $\mu \in \Sigma_m$ and $f$ is a tracial $\mathrm{W}^*$-function, then $\mu(f)$ is defined as $f^{\cA}(X)$ whenever $\cA$ is a tracial $\mathrm{W}^*$-algebra with separable predual and $X \in L^\infty(\cA)_{\sa}^m$ with $\lambda_X = \mu$; this is well-defined because $\mathrm{W}^*(X)$ is determined up to isomorphism by $\lambda_X = \mu$.
	\end{definition}
	
	One example of a tracial $\mathrm{W}^*$-function would be 
	\[
	f^{\cA}(X) = \begin{cases} \tau_{\cA}(p(X)), & \norm{X}_\infty \leq R \\ \infty, & \text{otherwise,} \end{cases}
	\]
	where $p$ is a non-commutative polynomial.  Tracial $\mathrm{W}^*$-functions also include scalar-valued tracial non-commutative smooth functions as in \cite{JekelThesis} and \cite{JLS2021} in the following sense.  If $\phi$ is such a tracial non-commutative smooth function, then $\phi^{\cA}(X)$ is only a priori defined when $X \in L^\infty(\cA)_{\sa}^m$; however, in many cases $\phi$ is Lipschitz with respect to $\norm{\cdot}_{L^2(\cA)_{\sa}^m}$ and hence can be extended to a function on $L^2(\cA)_{\sa}^m$ which will be a tracial $\mathrm{W}^*$-function.  However, tracial $\mathrm{W}^*$-functions are much more general because they are not assumed to be continuous in any sense.
	
	\begin{definition} \label{def:Econvexsummary}
		We say that $f$ is \emph{$E$-convex} if $f^{\cA}$ is convex and lower semi-continuous on $L^2(\cA)_{\sa}^m$ for each $\cA$, and if for every inclusion $\iota: \cA \to \cB$, letting $E: \cB \to \cA$ be the corresponding trace-preserving conditional expectation, we have $f^{\cA}(E[X]) \leq f^{\cB}(X)$ for $X \in L^2(\cB)_{\sa}^m$.  Here we use the notation $E[X] = (E[X_1],\dots,E[X_m])$ when $X = (X_1,\dots,X_m)$.
	\end{definition}
	
	Motivation for the definition of $E$-convexity will be given in Lemmas \ref{lem:scalarconvex} and \ref{lem:Econvexitymotivation}.
	
	%As one would expect, the quadratic function $f^{\cA}(X) = \norm{X}_{L^2(\cA)_{\sa}^m}^2$ is $E$-convex.
	
	\begin{proposition} \label{prop:dualitysummary}
		$C(\mu,\nu)$ is equal to the infimum of $\mu(f) + \nu(g)$ over pairs $(f,g)$ of $E$-convex $\mathrm{W}^*$-functions that satisfy $f^{\cA}(X) + g^{\cA}(Y) \geq \ip{X,Y}_{L^2(\cA)}$ for every tracial $\mathrm{W}^*$-algebra with separable predual and $X$, $Y \in L^2(\cA)_{\sa}^m$.  There exists an admissible pair of $E$-convex functions that achieves the infimum.  See Definition \ref{def:admissible} and Propositions \ref{prop:MKduality1} and \ref{prop:MKduality2}.
	\end{proposition}
	
	Another consequence of the classification-related obstructions to non-commutative transport is that we cannot expect too much regularity in general for the $E$-convex functions associated to an optimal coupling.  For instance, suppose two non-commutative laws $\mu$ and $\nu$ generate tracial von Neumann algebras that cannot embed into each other.  This implies that if $(X,Y)$ is an optimal coupling of these two laws on a tracial $\mathrm{W}^*$-algebra $\cA$, then neither of $\mathrm{W}^*(X)$ and $\mathrm{W}^*(Y)$ is contained in the other.  Thus, even though the non-commutative laws may be diffuse, the situation is similar to when coupling the classical measures $(1/2)(\delta_{-1} + \delta_1)$ and $(1/3)(\delta_{-1} + \delta_0 + \delta_1)$; in the optimal coupling, neither random variable can be expressed as a function of the other.  However, if a pair of $E$-convex functions associated to an optimal coupling were differentiable, that would imply that $X$ is in the von Neumann algebra generated by $Y$ and vice versa as a consequence of Lemma \ref{lem:Econvex}.
	
	It is natural to ask how close an arbitrary non-commutative optimal coupling is to a coupling where $X$ and $Y$ generate the same von Neumann algebra.  As a first application of duality, we show that every optimal coupling can be decomposed into an optimal coupling where the two variables generate the same $\mathrm{W}^*$-algebra and some additional orthogonal pieces.
	
	\begin{theorem} \label{thm:decompositionsummary}
		Suppose that $(\cA,X,Y)$ is an optimal coupling of $\mu, \nu \in \Sigma_m$.  Then there exists a $\mathrm{W}^*$-subalgebra $\cB$ such that the following hold.  Let $E_{\cB}: \cA \to \cB$ be the trace-preserving conditional expectation, and let $X' = E_{\cB}[X]$ and $Y' = E_{\cB}[Y]$.
		\begin{enumerate}[(1)]
			\item $X'$ and $Y'$ each generate $\cB$.
			\item $(\cB,X',Y')$ is an optimal coupling of $\lambda_{X'}$ and $\lambda_{Y'}$.
			\item $X' - Y'$, $X - X'$, $Y - Y'$ are mutually orthogonal.
		\end{enumerate}
		See Theorem \ref{thm:decomposition}.
	\end{theorem}
	
	Our main results in \S \ref{sec:displacement} concern the displacement interpolation.  If $(\cA,X,Y)$ is an optimal coupling of $\mu$ and $\nu$, then the \emph{displacement interpolation} refers to the family of random variables $X_t = (1 - t)X + tY$ for $t \in [0,1]$.  The associated laws $\mu_t = \lambda_{X_t}$ form a metric geodesic in $\Sigma_m$ with respect to the Wasserstein distance (see Proposition \ref{prop:geodesic}).  With the help of non-commutative Legendre transforms and Hopf-Lax semigroups, we will see that the $E$-convex functions associated to the couplings $(\cA,X_s,X_t)$ for $s, t \in (0,1)$ have more regularity than the $E$-convex functions associated to the original coupling $(\cA,X,Y)$ (see Proposition \ref{prop:functioninterpolation}).  As a consequence, we obtain the following non-commutative transport result.
	
	\begin{theorem} \label{thm:displacementW*}
		Let $(\cA,X,Y)$ be an optimal coupling of $\mu, \nu \in \Sigma_m$.  Then $\mathrm{W}^*((1 - t)X + tY) = \mathrm{W}^*(X,Y)$ for all $t \in (0,1)$.  For proof, see \S \ref{subsec:W*displacement}.
	\end{theorem}
	
	For instance, this theorem entails that for classical optimal couplings, the $\sigma$-algebra generated by $X_t$ is the same for all $t \in (0,1)$, which could be deduced directly from classical optimal transport theory by a similar proof.  The reader is encouraged to work out the classical example of $(1/2)(\delta_{-1} + \delta_1)$ and $(1/3)(\delta_{-1} + \delta_0 + \delta_1)$ as motivation.
	
	The results of \S \ref{sec:qinfo} highlight additional ways in which non-commutative optimal transport theory is significantly more complicated than its classical counterpart; specifically, the negative solution of the Connes embedding problem \cite{JNVWY2020} has a natural interpretation in terms of optimal couplings.  We observe that optimization over couplings involved in the definition of the Wasserstein distance can be replaced by optimization over what are called factorizable quantum channels in quantum information theory (see Observation \ref{obs:factorizablecoupling}). The results of \cite{HaMu2011,JNVWY2020,MuRo2020a} imply that there exist quantum channels between finite-dimensional matrix algebras which are factorizable whose factorization requires an infinite-dimensional non-Connes embeddable von Neumann algebra (see \S \ref{subsec:CEP}) for definitions).  We then show through Lemma \ref{lem:vectorduality} that channels with this property must occur as optimizers in the definition of Wasserstein distance.  From the optimal transportation point of view, this means that the optimal distance between certain tuples of finite-dimensional matrices cannot be even approximately realized inside a finite-dimensional coupling.
	
	\begin{proposition} \label{prop:counterexamples}
		Thanks to \cite{HaMu2011} and \cite{JNVWY2020}, for certain $n \in \N$, there exist non-commutative laws $\mu$ and $\nu$ associated to $n^2$-tuples in $M_n(\C)$ for which an optimal coupling requires a non-Connes embeddable tracial $\mathrm{W}^*$-algebra; see Corollary \ref{cor:nonConnes}.  Furthermore, thanks to \cite{MuRo2020a}, for every $n \geq 11$ and $d \in \N$, there exist $n^2$-tuples in $M_n(\C)$ such that if $(\cA,X,Y)$ is a coupling that is optimal among couplings on Connes-embeddable tracial $\mathrm{W}^*$-algebras, then $\cA$ must have dimension at least $d$; see Corollary \ref{cor:largedimcoupling} and Remark \ref{rem:largedimcoupling}.
	\end{proposition}
	
	In contrast to classical probability theory, we show that the $L^2$-Wasserstein metric does not generate the weak-$*$ topology on $\Sigma_{m,R}$.  We call the topology on $\Sigma_{m,R}$ generated by the Wasserstein distance the \emph{Wasserstein topology}.  We characterize when the two topologies agree at some $\mu$ in terms of the associated tracial $\mathrm{W}^*$-algebra (Proposition \ref{prop:twotopologies}) and hence obtain the following results (relying on the work of Connes \cite{Connes1976}).
	
	\begin{proposition} \label{prop:twotopologiessummary}
		The Wasserstein topology on $\Sigma_{m,R}$ is strictly stronger than the weak-$*$ topology; see \cite{BV2001} and Corollary \ref{cor:twotopologies}.  Furthermore, let $\Sigma_{m,R}^{\fin}$ denote the set of non-commutative laws $\lambda_X$ where $X$ comes from $L^2(\cA)_{\sa}^m$ with $\cA$ finite-dimensional.  Let $\mu$ be a non-commutative law and let $\cA$ be a tracial $\mathrm{W}^*$-algebra with a generating $m$-tuple $X$ such that $\lambda_X = \mu$ and $\norm{X}_{L^\infty(\cA)_{\sa}^m} \leq R$.  Then $\mu$ is in the weak-$*$ closure of $\Sigma_{m,R}^{\fin}$ if and only if $\cA$ is Connes-embeddable; see Lemma \ref{lem:ultraproductlaws2}.  Moreover, in this case, the weak-$*$ and Wasserstein topologies on $\Sigma_{m,R}$ agree at $\mu$ if and only if $\mu$ is in the Wasserstein closure of $\Sigma_{m,R}^{\fin}$, which is equivalent to $\cA$ being approximately finite-dimensional; see Proposition \ref{prop:twotopologies2}.
	\end{proposition}
	
	Approximate finite-dimensionality (see \S \ref{subsec:twotopologies} for definition) is the strongest way that a $\mathrm{W}^*$-algebra can be approximated by finite-dimensional algebras (besides being finite-dimensional itself), and thus the latter condition is quite restrictive when $m > 1$.  For instance, there is up to isomorphism only one AFD $\mathrm{II}_1$ factor \cite[\S XIV.2]{TakesakiIII}.  In \S \ref{subsec:randommatrix}, we explain how Propositions \ref{prop:counterexamples} and \ref{prop:twotopologiessummary} pose challenges to studying the large-$N$ convergence of Wasserstein distance for random matrix models.
	
	The results of Propositions \ref{prop:counterexamples} and \ref{prop:twotopologiessummary} constrast strongly with the classical situation.  Some treatments of optimal transport (e.g.\ \cite[p.\ 75]{Villani2008}) take for granted the fact that finitely supported probability measures are weak-$*$ dense in the space of probability measures on a compact set.  Such approximation arguments do not work in the non-commutative case for several reasons.  Due to the negative resolution of the Connes embedding problem \cite{JNVWY2020}, the non-commutative laws that can be realized in finite-dimensional algebras are not weak-$*$ dense.  Furthermore, by Proposition \ref{prop:twotopologiessummary}, the weak-$*$ closure of $\Sigma_{m,R}^{\fin}$ is much larger than its Wasserstein closure (assuming $m > 1$).  Finally, by Proposition \ref{prop:counterexamples}, even if two laws $\mu$ and $\nu$ can be realized in finite-dimensional algebras, an optimal coupling need not be weak-$*$ approximable by couplings in finite-dimensional algebras.
	
	Because the weak-$*$ and Wasserstein topologies are different for $m > 1$, one can deduce that $\Sigma_{m,R}$ with the Wasserstein distance is not compact (Corollary \ref{cor:noncompact}).  The following even more startling result is a consequence of Gromov, Olshanskii, and Ozawa's work \cite[Theorem 1]{Ozawa2004}.
	
	\begin{theorem} \label{thm:notseparable}
		For $m > 1$ and $R > 0$, the space $\Sigma_{m,R}$ is not separable with respect to $d_W^{(2)}$.
	\end{theorem}
	
	\subsection{Organization}
	
	The paper is organized as follows:
	\begin{itemize}
		\item In \S \ref{subsec:classicalcoupling} and \S \ref{subsec:matrixcoupling}, we motivate the definition of $E$-convex functions and the associated duality result in terms of two toy examples, classical probability spaces and $M_n(\C)$.
		\item In \S \ref{sec:background}, we recall standard background on tracial $\mathrm{W}^*$-algebras and their interpretation as non-commutative probability spaces for the sake of readers who are not specialists in that topic.
		\item In \S \ref{sec:duality}, we describe the properties of $E$-convex functions and the associated Legendre transform; we prove the non-commutative Monge-Kantorovich duality (Proposition \ref{prop:dualitysummary}) and the decomposition theorem for optimal couplings (Theorem \ref{thm:decompositionsummary}).
		\item In \S \ref{sec:displacement}, we study the non-commutative analog of inf-convolution and the regularity properties of $E$-convex and semi-concave functions; we prove Theorem \ref{thm:displacementW*} and give further detail about the functions associated to the displacement interpolation in Proposition \ref{prop:functioninterpolation}.
		\item In \S \ref{sec:qinfo}, we connect non-commutative optimal couplings with quantum information theory and prove Proposition \ref{prop:counterexamples}.  Then we study the differences between the weak-$*$ and the Wasserstein topology using a certain stability property (Proposition \ref{prop:twotopologies}) and hence prove Proposition \ref{prop:twotopologiessummary}.  Finally, we show non-separability of the Wasserstein space in \S \ref{subsec:nonseparability}.
		\item In \S \ref{subsec:randommatrix}, we explain how \S \ref{sec:qinfo} illustrates the difficulty of studying random matrix optimal transport in the large-$N$ limit.  Then \S \ref{subsec:bimodule} sketches a different but analogous theory of non-commutative optimal couplings that uses bimodules and $\operatorname{UCPT}$-maps of tracial $\mathrm{W}^*$-algebras.
		\item In the appendix \S \ref{sec:Lp}, we define non-commutative laws and optimal couplings for elements of non-commutative $L^p$ spaces, and show the existence of $L^p$ optimal couplings and Wasserstein geodesics.
	\end{itemize}

	\subsection{Motivation from classical probability} \label{subsec:classicalcoupling}
	
	First, we recall the classical Monge-Kantorovich duality.  Fix a standard Borel probability space $(\Omega,P)$ with no atoms. For $\mu$ and $\nu$ compactly supported probability measures on $\R^m$, a \emph{coupling} of $\mu$ and $\nu$ is a pair $(X,Y)$ of random variables on $\Omega$ with $X \sim \mu$ and $Y \sim \nu$.  The classical Wasserstein distance is the infimum of $\norm{X - Y}_{L^2(\Omega,P;\R^m)}$ over all such couplings, and a coupling is said to be \emph{optimal} if it achieves this infimum.
	
	\begin{theorem}[{See \cite[Theorem 5.10, Particular Case 5.17]{Villani2008}}]
		Let $(X,Y)$ be a coupling of two compactly supported measures $\mu$ and $\nu$ on $\R^m$.  Then $(X,Y)$ is optimal if and only if there exists a pair of convex functions $f, g: \R^m \to \R$ satisfying $f(x) + g(y) \geq \ip{x,y}$ for $x, y \in \R^m$ and $\mathbb{E}[f(X)] + \mathbb{E}[g(Y)] = \mathbb{E} \ip{X,Y}$.  Furthermore, $\mathbb{E}[f(X)] + \mathbb{E}[g(Y)] = \mathbb{E} \ip{X,Y}$ implies that $Y$ is almost surely in the subdifferential of $f$ at $X$ and $X$ is almost surely in the subdifferential of $g$ at $Y$.
	\end{theorem}
	
	As explained above, $E$-convex functions will be an analog of functions on $L^2(\Omega,P;\R^m)$ rather than $\R^m$.  Every convex function on $\R^m$ defines a convex function on $L^2(\Omega,P;\R^m)$ as follows.
	
	\begin{lemma} \label{lem:scalarconvex}
		Let $f: \R^m \to (-\infty,\infty]$ be convex and lower semi-continuous.  Let $(\Omega,P)$ be a non-atomic standard Borel probability space with underlying $\sigma$-algebra $\mathcal{F}$.  Define
		\[
		\tilde{f}: L^2(\Omega,P;\R^m) \to \R, X \mapsto \mathbb{E}[f(X)],
		\]
		which is well-defined in $(-\infty,\infty]$ thanks to Jensen's inequality.  Then
		\begin{enumerate}[(1)]
			\item $\tilde{f}(X)$ only depends on the law (probability distribution) of $X$.
			\item $\tilde{f}$ is convex and lower semi-continuous.
			\item Suppose that $\tilde{f}(X) < \infty$.  Then $Y$ is in the subdifferential of $\tilde{f}$ at $X$ if and only if $Y$ is in the subdifferential of $f$ at $X$ almost surely.
			\item $\tilde{f}$ is monotone under conditional expectations: If $\mathcal{G}$ is a sub-$\sigma$-algebra of $\mathcal{F}$, then
			\[
			\tilde{f}(E[X|\mathcal{G}]) \leq \tilde{f}(X).
			\]
		\end{enumerate}
	\end{lemma}
	
	\begin{proof}[Sketch of proof]
		(1) This is immediate.
		
		(2) Convexity of $\tilde{f}$ is immediate from convexity of $f$.  To show lower semi-continuity of $\tilde{f}$, not that $f(x) + |x|^2/2$ is bounded from below by some constant $C$ and thus $g(x) := f(x) + |x|^2 / 2 - C$ is a nonnegative convex function.  If $X_n \to X$ in $L^2(\Omega,P;\R^m)$, then $X_n \to X$ in probability, and hence $\liminf_{n \to \infty} g(X_n) \geq g(X)$ in probability.  Thus, by Fatou's lemma for convergence in probability $\liminf_{n \to \infty} \tilde{g}(X_n) \geq \tilde{g}(X)$, which implies that $\tilde{f}$ is also lower semicontinuous.
		
		(3) If $Y$ is in the subdifferential of $f$ at $X$ almost surely and $Z \in L^2(\Omega,P)$, then $f(Z) \geq f(X) + \ip{Z-X,Y}_{\R^m}$ almost surely, and thus by taking expectations $\tilde{f}(Z) \geq \tilde{f}(X) + \ip{Z-X,Y}_{L^2(\Omega,P;\R^m)}$.  For the converse, let $S = \{x \in \R^m: f(x) < \infty\}$ and fix a countable dense subset $\Xi$ of $S$.  For each $n > 0$ and $\xi \in \Xi$, let $E_{n,\xi}$ be the event
		\[
		E_{n,\xi} = \{f(\xi) \leq f(X) + \ip{\xi - X,Y}_{\R^m} - 1/n\}.
		\]
		Because $Y$ is in the subdifferential of $\tilde{f}$ at $X$, we have
		\[
		\tilde{f}(1_{E_{n,\xi}^c}X + 1_{E_{n,\xi}} \xi) \geq \tilde{f}(X) + \ip{1_{E_{n,\xi}}(\xi - X),Y}_{L^2(\Omega,P;\R^m)}.
		\]
		On the other hand, by definition of $E_{n,\xi}$, we have
		\[
		\tilde{f}(1_{E_{n,\xi}^c}X + 1_{E_{n,\xi}} \xi) \leq \tilde{f}(X) + \ip{1_{E_{n,\xi}}(\xi - X),Y}_{L^2(\Omega,P;\R^m)} + \frac{1}{n} P(E_{n,\xi}).
		\]
		Therefore, $P(E_{n,\xi}) = 0$.  Since this holds for all $n \in \N$, we have $f(\xi) \geq f(X) + \ip{\xi - X,Y}_{\R^m}$ almost surely for each $\xi$.  Since $\Xi$ is countable, we have this condition every $\xi \in \Xi$ at once almost surely.  On this event, if $x \in \R^m$ with $f(x) < \infty$, then $f$ is continuous at $x$, and therefore by taking sequence of $\xi \in \Xi$ that converges to $x$ we obtain $f(x) \geq f(X) + \ip{x-X,Y}_{\R^m}$.
		
		(4) This follows from Jensen's inequality and the existence of regular conditional distributions for standard Borel probability spaces.
	\end{proof}
	
	\begin{remark}
		Similar reasoning shows that if $g$ is the Legendre transform of $f$ on $\R^m$, then $\tilde{g}$ is the Legendre transform of $\tilde{f}$ on $L^2(\Omega,P;\R^m)$.
	\end{remark}
	
	Let us call a function $F: L^2(\Omega,P;\R^m) \to (-\infty,\infty]$ \emph{classically $E$-convex} if
	\begin{enumerate}[(1)]
		\item $F(X)$ depends only on the law of $X$.
		\item $F$ is convex and lower semi-continuous.
		\item We have $F(E[X|\mathcal{G}]) \leq F(X)$ for every sub-$\sigma$-algebra $\mathcal{G}$ and every $X \in L^2(\Omega,P;\R^m)$.
	\end{enumerate}
	Then we have the following version of Monge-Kantorovich duality using classically $E$-convex functions on $L^2(\Omega,P;\R^m)$.
	
	\begin{corollary} \label{cor:classicalMK}
		Let $(X,Y)$ be a coupling on $(\Omega,P)$ of two compactly supported measures $\mu$ and $\nu$ on $\R^m$.  Then $(X,Y)$ is optimal if and only if there exists a pair of classically $E$-convex functions $F$, $G: L^2(\Omega,P;\R^m) \to (-\infty,\infty]$ such that
		\[
		F(X') + G(Y') \geq \ip{X',Y'}_{L^2(\Omega,P;\R^m)} \text{ for all } X', Y' \in L^2(\Omega,P;\R^m),
		\]
		and
		\[
		F(X) + G(Y) = \ip{X,Y}_{L^2(\Omega,P;\R^m)}.
		\]
	\end{corollary}
	
	\begin{proof}
		($\implies$) By the classical Monge-Kantorovich duality, there are convex functions $f, g: \R^m \to (-\infty,\infty]$ with $f(x) + g(y) \geq \ip{x,y}_{\R^m}$ and $\mathbb{E}f(X) + \mathbb{E} f(Y) = \ip{X,Y}_{L^2(\Omega,P;\R^m)}$.  Let $F = \tilde{f}$ and $G = \tilde{g}$.  By Lemma \ref{lem:scalarconvex}, $F$ and $G$ are classical $E$-convex and clearly $F(X) + G(Y) = \ip{X,Y}_{L^2(\Omega,P;\R^m)}$.  Also, $F(X') + G(Y') \geq \ip{X',Y'}_{L^2(\Omega,P;\R^m)}$ since $f(x) + f(y) \geq \ip{x,y}_{\R^m}$.
		
		($\impliedby$) Suppose that $(X',Y')$ is another coupling of $\mu$ and $\nu$ on $(\Omega,P)$.  Then
		\[
		\ip{X',Y'}_{L^2(\Omega,P;\R^m)} \leq F(X') + G(Y') = F(X) + G(Y) = \ip{X,Y}_{L^2(\Omega,P;\R^m)},
		\]
		where in the middle equality we have used that $F(X) = F(X')$ and $G(Y) = G(Y')$ since $X \sim X'$ and $Y \sim Y'$ in law.  Therefore, the coupling $(X,Y)$ is optimal.
	\end{proof}
	
	Corollary \ref{cor:classicalMK} is the statement that we will generalize to the non-commutative setting.  We remark that although classically $E$-convex functions are much less concrete than convex functions on $\R^m$, Corollary \ref{cor:classicalMK} still has the power to prove the classical analogs of Theorems \ref{thm:decompositionsummary} and \ref{thm:displacementW*} by exactly the same arguments that we will use in the non-commutative case.
	
	In fact, convex functions on a space of classical random variables have also been used in the theory of mean field games \cite{GM2021}.  Mean field games involves the study of the \emph{master equation} \cite{CDLL2019,GMMZ2021}, a differential equation for a function $u(t,x,\mu)$ depending on a time variable $t$, a space variable $x$ (representing the position of an individual agent), and a measure $\mu$ (representing the distribution of the positions of a continuum of other agents).  We can define a function $\widehat{u}$ on $[0,\infty) \times \R^m \times L^2(\Omega,P;\R^m)$ by $\widehat{u}(t,x,X) = u(t,x,\mu_X)$, where $\mu_X$ is the law of $X$.  The first-order regularity conditions needed to solve the master equation are more easily stated in terms of the function $\widehat{u}$ on the Hilbert space $\R^m \times L^2(\Omega,P;\R^m)$.  Moreover, the proof of existence and uniqueness of solutions to Hamilton-Jacobi equations on Wasserstein space $\mathcal{P}_2(\R^m)$ \cite{GMS2021,GaTu2018} relies on the theory of viscosity solutions to Hamilton-Jacobi equations on Hilbert spaces \cite{CrLi1985,CrLi1986a,CrLi1986b,LL1986}.
	
	The inf-convolution techniques that we use in \S \ref{sec:displacement} are an important special case of this theory of Hamilton-Jacobi equations on Hilbert spaces.  In fact, part of our motivation was to understand the non-commutative version of Hamilton-Jacobi equations for functions of a random variable.  Recent work has connected random matrix theory to viscosity solutions of Hamilton-Jacobi equations \cite{BDLL2021} and mean field games \cite{CCP2020}.  However, these connections are restricted to the setting of a single random matrix because they rely heavily on the description of self-adjoint random matrices in terms of their eigenvalues.  It would be of great interest to have a theory of viscosity solutions to partial differential equations in several non-commuting variables as is suggested by the study of heat equations in \cite{DGS2016,Jekel2018,JLS2021} and the Hamilton-Jacobi-Bellman equation in \cite{Dabrowski2017,Jekel2018}.
	
	\subsection{Motivation from matrix tuples} \label{subsec:matrixcoupling}
	
	In order to motivate some of the ideas of our paper, we explain a toy model of couplings between tuples of $n \times n$ matrices.  Let $M_n(\C)$ denote the space of complex $n \times n$ matrices.  Let $\tr_n = (1/n) \Tr_n$ be the normalized trace on $M_n(\C)$. We define an inner product on $M_n(\C)$ by
	\[
	\ip{S,T}_{\tr_n} = \tr_n(S^*T).
	\]
	Let $M_n(\C)_{\sa}$ denote the real subspace of self-adjoint matrices.  Then $\ip{X,Y}_{\tr_n} \in \R$ for all $X, Y \in M_n(\C)_{\sa}$.  Every element of $M_n(\C)$ can be uniquely written as $S + iT$ with $S, T \in M_n(\C)_{\sa}$, and hence there is a natural identification of the complex inner product space $M_n(\C)$ with the complexification of the real inner product space $M_n(\C)_{\sa}$.
	
	From a non-commutative probability viewpoint, we can view $M_n(\C)$ as an algebra of ``random variables'' and the normalized trace $\tr_n: M_n(\C) \to \C$ as the ``expectation.''  To motivate this, suppose $X \in M_n(\C)_{\sa}$.  The \emph{empirical spectral distribution} of $X$ is the measure $\mu = \frac{1}{n} \sum_{j=1}^n \delta_{\lambda_j}$ where $\lambda_1$, \dots, $\lambda_n$ are the eigenvalues of $X$ listed with multiplicity.  We then have for every polynomial $p$ that
	\[
	\tr_n(p(X)) = \int p\,d\mu.
	\]
	Thus, $\mu$ is analogous to the distribution of a random variable.
	
	If $X = (X_1,\dots,X_m) \in M_n(\C)_{\sa}^m$, the ``joint distribution'' of $X_1$, \dots, $X_m$ is not described by a measure on $\R^m$, since $X_1$, \dots, $X_m$ do not commute.  Rather we consider the \emph{non-commutative law} $\lambda_X$, which is the linear functional on the algebra of $m$-variable non-commutative polynomials given by
	\[
	p \mapsto \tr_n(p(X_1,\dots,X_m)).
	\]
	It turns out that two tuples $X$ and $Y \in M_n(\C)_{\sa}^m$ have the same non-commutative law if and only if they are unitarily conjugate.
	
	\begin{lemma}[] \label{lem:matricesunitarilyconjugate}
		Let $X, Y \in M_n(\C)_{\sa}^m$.  Then the following are equivalent
		\begin{enumerate}[(1)]
			\item $\tr_n(p(X)) = \tr_n(p(Y))$ whenever $p$ is a non-commutative polynomial in $m$ variables.
			\item There exists a unitary $U$ in $M_n(\C)$ such that $Y_j = UX_j U^*$ for $j = 1$, \dots, $m$.
		\end{enumerate}
	\end{lemma}
	
	This lemma follows from the multivariate version of Specht's theorem \cite{Specht1940} observed by Wiegmann \cite{Wiegmann1961} and verified in \cite[Theorem 2.2]{Jing2015}.  This result is closely connected to the invariant theory of matrices \cite{Procesi1976}, and related results have been rediscovered many times as the survey \cite{Shapiro1991} explains.  Moreover, many in the operator algebras community are aware it can be deduced from Lemma \ref{lem:lawisomorphism} below, and the fact that any two trace-preserving embeddings of a finite-dimensional tracial $*$-algebra into $M_n(\C)$ are unitarily conjugate, which is a consequence of the Artin-Wedderburn-type classification of finite-dimensional $*$-algebras and their representations (see e.g.\ \cite[\S 2]{Effros1981}).
	
	We consider the toy problem of optimally coupling two matrix tuples inside $M_n(\C)$ (beware that because of Proposition \ref{prop:counterexamples} an optimal coupling inside $M_n(\C)$ is not necessarily optimal among all couplings in tracial $\mathrm{W}^*$-algebras).  Because of Lemma \ref{lem:matricesunitarilyconjugate}, the toy problem reduces to the following:  Given $X, Y \in M_n(\C)_{\sa}^m$, find a unitary $U$ so that $\norm{UXU^* - Y}_{\tr_n}$ is as small as possible, where $UXU^* = (UX_1U^*, \dots, UX_mU^*)$, and where $\norm{\cdot}_{\tr_n}$ is the normalized Hilbert-Schmidt norm
	\[
	\norm{T}_{\tr_n} = \left(\sum_{j=1}^m \tr_n(T_j^*T_j) \right)^{1/2}.
	\]
	This motivates the following definition: For $X, Y \in M_n(\C)_{\sa}^m$, we say that $(X,Y)$ are an \emph{optimal coupling in $M_n(\C)$} if $\norm{UXU^* - Y}_{\tr_n} \geq \norm{X - Y}_{\tr_n}$ for every unitary $U$.  The next lemma guarantees existence of optimal couplings.
	
	\begin{lemma} \label{lem:sumofcommutators}
		Let $X$, $Y \in M_n(\C)_{\sa}^m$.  Then there exists an $n \times n$ unitary $U$ that minimizes $\norm{UXU^* - Y}_{\tr_n}$.  Moreover, every such unitary must satisfy
		\[
		\sum_{j=1}^m [UX_jU^*,Y_j] = 0,
		\]
		where $[S,T] = ST - TS$ is the commutator.
	\end{lemma}
	
	\begin{proof}
		Existence of a minimizer follows from the fact that the unitary group is compact and $U \mapsto \norm{UXU^* - Y}_{\tr_n}$ is continuous.  Now suppose that $U$ is a minimizer and let $Z = UXU^*$.  Let $A$ be a self-adjoint matrix, and consider the unitary $e^{itA}$ for $t \in \R$.  By minimality, we have $\norm{e^{itA}Ze^{-itA} - Y}_{\tr_n}^2 \geq \norm{Z - Y}_{\tr_n}^2$.  Since $\norm{e^{itA}Ze^{-itA}}_{\tr_n}^2 = \norm{Z}_{\tr_n}^2$, it follows that $\ip{e^{itA}Ze^{-itA},Y}_{\tr_n}$ is minimized at $t = 0$.  Differentiating at $t = 0$ yields
		\[
		\sum_{j=1}^m \tr_n((iAZ_j - iZ_jA) Y_j) = \sum_{j=1}^m \tr_n(A i(Z_jY_j - Y_jZ_j)) = \tr_n \left( A \sum_{j=1}^m i[Z_j,Y_j] \right).
		\]
		Since this holds for all $A \in M_n(\C)_{\sa}$, it follows that $\sum_{j=1}^m [Z_j,Y_j] = 0$ as desired.
	\end{proof}
	
	\begin{remark}
		In the case $m = 1$, this lemma actually provides an alternative proof the spectral theorem as follows.  Let $X \in M_n(\C)_{\sa}$.  Let $Y$ be a fixed diagonal matrix with distinct diagonal entries $y_1$ ,\dots, $y_n$.  Let $U$ be a unitary minimizing $\norm{UXU^* - Y}_{\tr_n}$.  Then $[UXU^*,Y] = 0$.  Any matrix $A$ that commutes with $Y$ must satisfy $a_{i,j} y_j = y_i a_{i,j}$, and hence $A$ must be diagonal.  Therefore, $UXU^*$ is diagonal.\footnote{One might object that the preceding lemma seems to assume the spectral theorem already because it uses functional calculus to define $e^{itA}$.  However, this only requires analytic functional calculus, not continuous functional calculus.  One can use power series to define $e^{itA}$, show that $e^{i(s+t)A} = e^{isA} e^{itA}$ for $s, t \in \R$, show that $(e^{itA})^* = e^{-itA^*}$, and hence conclude that $e^{itA}$ is unitary when $A$ is self-adjoint.}
	\end{remark}
	
	Next, we describe an analog of the Monge-Kantorovich duality for the setting of matrix tuples.
	
	\begin{lemma}
		Let $X, Y \in M_n(\C)_{\sa}^m$.  Then $(X,Y)$ is an optimal coupling in $M_n(\C)$ if and only if there exist functions $f, g: M_n(\C)_{\sa}^m \to \R$ satisfying the following properties:
		\begin{enumerate}[(1)]
			\item $f$ and $g$ are convex.
			\item $f$ and $g$ are unitarily invariant, that is, $f(UX'U^*) = f(X')$ and $g(UY'U^*) = g(Y')$ for $U$ unitary and $X'$, $Y' \in M_n(\C)_{\sa}^m$.
			\item $f(X') + g(Y') \geq \ip{X',Y'}_{\tr_n}$ for all $X'$, $Y' \in M_n(\C)_{\sa}^m$.
			\item $f(X) + g(Y) = \ip{X,Y}_{\tr_n}$.
		\end{enumerate}
	\end{lemma}
	
	\begin{proof}
		($\implies$).  Let $\cU(M_n(\C))$ be the unitary group.  Let
		\[
		f(X') = \sup_{U \in \cU(M_n(\C))} \ip{X',UYU^*}_{\tr_n}.
		\]
		Note that $f$ is convex because it is the supremum of a family of affine functions.  Moreover, $f$ is unitarily invariant because we took the supremum over all unitaries  $U$.
		
		Let $g$ be the Legendre transform of $f$, that is,
		\[
		g(Y') = \sup_{X' \in M_n(\C)_{\sa}^m} \left( \ip{Y',X'}_{\tr_n} - f(X') \right).
		\]
		It is immediate that $g$ is convex, $g$ is unitarily invariant because $f$ is unitarily invariant and the inner product is unitarily invariant, and $f(X') + g(Y') \geq \ip{X',Y'}_{\tr_n}$ for all $X'$, $Y' \in M_n(\C)_{\sa}^m$.  In particular, $f(X) + g(Y) \geq \ip{X,Y}_{\tr_n}$.
		
		On the other hand, note that the supremum defining $f(X)$ is achieved when $U = 1$ because we assumed that $(X,Y)$ is optimal coupling, hence $\ip{X,UYU^*}$ is maximized when $U = 1$.  Hence, $f(X) = \ip{X,Y}$.  Moreover,
		\[
		f(X') \geq \ip{X',Y}_{\tr_n},
		\]
		hence
		\[
		g(Y) \leq \sup_{X' \in M_n(\C)_{\sa}^n} \left( \ip{X',Y}_{\tr_n} - \ip{X',Y}_{\tr_n} \right) = 0.
		\]
		Thus, $f(X) + g(Y) \leq \ip{X,Y}_{\tr_n}$.  Hence, $f(X) + g(Y) = \ip{X,Y}$ as desired.
		
		($\impliedby$) Suppose that $f$ and $g$ satisfy (1)--(4).  Let $U$ be a unitary.  Then
		\[
		\ip{UXU^*,Y}_{\tr_n} \leq f(UXU^*) + g(Y) = f(X) + g(Y) = \ip{X,Y}_{\tr_n}.
		\]
		Therefore, $(X,Y)$ is optimal.
	\end{proof}
	
	Unitarily invariant convex functions on $M_n(\C)_{\sa}^m$ satisfy a monotonicity property with respect to the non-commutative condition expectation from $M_n(\C)$ onto a $*$-subalgebra $A$, which is one motivation for our notion of $E$-convexity in the tracial $\mathrm{W}^*$-setting.
	
	\begin{lemma} \label{lem:Econvexitymotivation}
		Let $A$ be a $*$-subalgebra of $M_n(\C)$, and let $E: M_n(\C) \to A \subseteq M_n(\C)$ be the orthogonal projection with respect to the inner product $\ip{S,T}_{\tr_n} = \tr_n(S^*T)$.  Then $E[ST] = S E[T]$ and $E[TS] = E[T]S$ and $E[T^*] = E[T]^*$ for $T \in M_n(\C)$ and $S \in A$.  Moreover, if $f: M_n(\C)_{\sa}^m \to \R$ is a convex function that is invariant under unitary conjugation, then for $X = (X_1,\dots,X_m) \in M_n(\C)_{\sa}^m$, we have
		\[
		f(E[X]) \leq f(X).
		\]
		Here $E[X] = (E[X_1],\dots,E[X_m])$.
	\end{lemma}
	
	\begin{proof}
		For a subalgebra $A \subseteq M_n(\C)$, we denote by $\cU(A)$ the group of unitary matrices that are contained in $A$.  We define the \emph{commutant}
		\[
		A' = \{S \in M_n(\C) | [S,T] = 0 \text{ for all } T \in A \}.
		\]
		We recall that $A'' = A$ by von Neumann's bicommutant theorem \cite[Theorem II.3.9]{TakesakiI}.
		
		Let $\mu$ be the Haar measure on $\cU(A')$, and define $F:M_n(\C) \to M_n(\C)$ by
		\[
		F(X) = \int_{\cU(A')} UXU^* \,d\mu(U).
		\]
		We claim that $F(X) = E[X]$.  First, to show that $F(X) \in A$, note that for $V \in \cU(\cA')$, we have $VF(X)V^* = F(X)$, hence $[F(X),V] = 0$ by invariance of the Haar measure.  Since $A'$ is a $*$-algebra, it is linearly spanned by its unitaries, and therefore, $[F(X),S] = 0$ for all $S \in A'$.  So $F(X) \in A'' = A$.  Furthermore, for all $T \in A$, we have
		\[
		\tr_n(T^* F(X)) = \int_{\cU(A')} \tr_n(T^* UXU^*) \,d\mu(U) = \int_{\cU(A')} \tr_n(UT^*XU^*) \,d\mu(U) = \tr_n(T^*X).
		\]
		Thus, $F(X)$ is the orthogonal projection of $X$ onto $A$, or $F(X) = E[X]$, as desired.
		
		Similar computations from definition of $F$ show that $F$ is an $A$-$A$-bimodule map and $F(X^*) = F(X)^*$, and hence these properties also hold for $E$.
		
		Since $\mu$ is a probability measure, Jensen's inequality and the unitary invariance of $f$ imply that
		\[
		f(E[X]) \leq \int_{\cU(A')} f(UXU^*)\,d\mu(U) = f(X). \qedhere
		\]
	\end{proof}

	\section{Background on tracial $\mathrm{W}^*$-algebras} \label{sec:background}
	
	For the sake of readers who are less familiar tracial $\mathrm{W}^*$-algebras, we explain the prerequisites needed for the paper:\ the definition of a tracial $\mathrm{W}^*$-algebra, its interpretation as a non-commutative generalization of probability spaces, inclusions and trace-preserving conditional expectations of tracial $\mathrm{W}^*$-algebras, free products with amalgamation, and non-commutative laws.
	
	\subsection{Tracial $\mathrm{W}^*$-algebras}
	
	Historically, von Neumann algebras and $\mathrm{W}^*$-algebras were defined differently, but it turns out that these two definitions give the same objects thanks to work of Sakai; see e.g.\ \cite[Theorem 1.16.7]{Sakai1971}.  Here we follow Sakai's approach that starts with the definition of $\mathrm{W}^*$-algebras as $\mathrm{C}^*$-algebras which are dual Banach spaces \cite{Sakai1971}.  Other background references on von Neumann algebras include \cite{ADP,TakesakiI,TakesakiII,TakesakiIII}.
	
	\begin{definition}
		A \emph{unital $*$-algebra} is a (unital) algebra $A$ over $\C$ together with a skew-linear involution $a \mapsto a^*$ such that $(ab)^* = b^*a^*$.  If $A$ and $B$ are $*$-algebras, then a map $\rho: A \to B$ is said to be a \emph{$*$-homomorphism} if it is linear and respects multiplication and the $*$-operation.
	\end{definition}
	
	\begin{definition}
		A \emph{unital $\mathrm{C}^*$-algebra} is a $*$-algebra $A$ equipped with a norm $\norm{\cdot}$ such that
		\begin{itemize}
			\item $A$ is a Banach space with respect to $\norm{\cdot}$;
			\item $\norm{ab} \leq \norm{a} \norm{b}$ for $a, b \in A$;
			\item $\norm{a^*a} = \norm{a}^2$ for $a \in A$.
		\end{itemize}
	\end{definition}
	
	\begin{definition}
		A \emph{$\mathrm{W}^*$-algebra} is a $\mathrm{C}^*$-algebra $A$ together with a topology $\mathscr{T}$, such that $A$ as a Banach space is the dual of some Banach space $A_*$ and $\mathscr{T}$ is the weak-$*$ topology on $A$.
	\end{definition}
	
	We remark that $A_*$ can be uniquely recovered from $(A,\mathscr{T})$ as the subspace of $A^{**}$ consisting of linear functionals that are continuous with respect to $\mathscr{T}$.  In fact, it turns out that the predual of $A_*$ of a $\mathrm{W}^*$-algebra $A$ is uniquely determined by $A$ alone without reference to its weak-$*$ topology \cite[Corollary 1.13.3]{Sakai1971}.
	
	\begin{definition}
		If $A$ is a $\mathrm{W}^*$-algebra and $A_*$ is a predual of $A$, then a \emph{faithful normal trace on $A$} is an element $\tau \in A_*$ satisfying the following properties:
		\begin{itemize}
			\item $\tau(1) = 1$;
			\item $\tau(a^*a) \geq 0$ for $a \in A$;
			\item $\tau(a^*a) = 0$ if and only if $a = 0$;
			\item $\tau(ab) = \tau(ba)$ for $a, b \in A$.
		\end{itemize}
	\end{definition}
	
	We remark that in general von Neumann algebra theory, the word ``trace'' is often used to refer to the semi-finite trace on a semi-finite von Neumann algebra, but in this paper ``trace'' always means ``tracial state.''
	
	\begin{definition}
		A \emph{tracial $\mathrm{W}^*$-algebra} is a pair $\cA = (A,\tau)$, where $A$ is a $\mathrm{W}^*$-algebra and $\tau$ is a faithful normal trace.
	\end{definition}
	
	\begin{example}
		Let $(\Omega,P)$ be a probability space.  We take $A = L^\infty(\Omega,P)$, with the pointwise addition and multiplication operations.  The $*$-operation is pointwise complex conjugation.  The norm is the standard one for $L^\infty(\Omega,P)$, and note that $\norm{fg} \leq \norm{f} \norm{g}$ and $\norm{f^*f} = \norm{f}^2$.  By the Riesz representation theorem, $L^\infty(\Omega,P) = L^1(\Omega,P)^*$, and therefore, we can take $A_* = L^1(\Omega,P)$, and then equip $L^\infty(\Omega,P)$ with the corresponding weak-$*$ topology.  We define $\tau$ using the element $1 \in L^1(\Omega,P)$, so that $\tau(f) = \int_{\Omega} f\,dP$.  Since $L^\infty(\Omega,P)$ is commutative, it is immediate that $\tau(fg) = \tau(gf)$.  The other properties of $\tau$ are straightforward to check from well-known facts in measure theory.  	Conversely, it turns out that every commutative tracial $\mathrm{W}^*$-algebra is isomorphic to $L^\infty$ of some probability space \cite[\S 1.18]{Sakai1971}, \cite[Theorem 1.18]{TakesakiI}.
	\end{example}
	
	\begin{example}
		Let $H$ be an infinite-dimensional Hilbert space, and let $A = B(H)$ be the algebra of bounded operators on $H$ equipped with the operator norm.  Let $A_*$ be the space of trace class operators.  Then $A$ can be canonically identified with the dual of $A_*$ by the pairing $(a,T) = \Tr(aT)$ for $a \in A$ and $T \in A_*$.  The weak-$*$ topology on $B(H)$ is also known as the \emph{$\sigma$-weak operator topology}.  Thus, $B(H)$ is a $\mathrm{W}^*$-algebra.  However, it is not a tracial $\mathrm{W}^*$-algebra because $\Tr$ is not well-defined on all of $B(H)$ and $\Tr(1) = \infty$.  See for instance \cite[Theorem 1.15.3]{Sakai1971}.
	\end{example}
	
	\begin{theorem}[GNS construction for tracial $\mathrm{W}^*$-algebras] \label{thm:W*GNSrep}
		Let $\cA = (A,\tau)$ be a tracial $\mathrm{W}^*$-algebra.  Note that $\ip{a,b}_{\cA} := \tau(a^*b)$ defines an inner product on $A$ (which is non-degenerate because $\tau$ is faithful).  This can be completed to a Hilbert space, which we denote by $L^2(\cA)$.  Let us denote the map $A \to L^2(\cA)$ by $a \mapsto \widehat{a}$.  Then for each $a \in A$, there is are unique operators $\pi_\ell(a), \pi_r(a) \in B(L^2(\cA))$ such that $\pi_\ell(a) \widehat{b} = \widehat{ab}$ and $\pi_r(a) \widehat{b} = \widehat{ba}$ for $b \in A$.  Moreover, $\pi_\ell$ defines a $*$-homomorphism $A \to B(L^2(\cA))$ which is continuous with respect to the weak-$*$ topologies on $A$ and $B(L^2(\cA))$.  Similarly, $\pi_r$ is a $*$-anti-homomorphism (it preserves $+$ and $*$ but reverses the order of multiplication) that is weak-$*$ continuous.  Furthermore, since $\norm{a^*}_{L^2(\cA)} = \norm{a}_{L^2(\cA)}$, there is a unique skew-linear isometry $J: L^2(\cA) \to L^2(\cA)$ such that $J(\widehat{a}) = \widehat{a^*}$.  See \cite[\S IV]{MvN2} and \cite[\S 7]{ADP}.
	\end{theorem}
	
	\begin{example}
		Let $A = L^\infty(\Omega,P)$ and let $\tau$ be integration against $P$.  Then $\ip{f,g}_{L^2(\cA)} = \int_{\Omega} \overline{f}g\,dP$.  The completion $L^2(\cA)$ can be canonically identified with $L^2(\Omega,P)$.  The map $\widehat{~}$ is the standard inclusion $L^\infty(\Omega,P) \to L^2(\Omega,P)$.  The operator $\pi(f) \in B(L^2(\Omega,P))$ is the operator of multiplication by $f$.
	\end{example}
	
	\begin{remark}
		Our examples indicate that if $\cA = (A,\tau)$ is a tracial $\mathrm{W}^*$-algebra, then $A$ is an analog of $L^\infty(\Omega,P)$, $A_*$ is an analog of $L^1(\Omega,P)$ and $L^2(\cA)$ is an analog of $L^2(\Omega,P)$.  In fact, there is an even a non-commutative analog of measurable functions on $\Omega$ that are finite almost everywhere; this is known as the algebra $\Aff(\cA)$ of operators \emph{affiliated to $\cA$}, certain closed unbounded operators on the Hilbert space $L^2(\cA)$.  The space $L^2(\cA)$ can be canonically identified with a subspace of the affiliated operators.  Thus, the left and right multiplication operators $\pi_\ell(a)$ and $\pi_r(a)$ for $a \in A$ become instances of multiplying affiliated operators.  Moreover, there are subspaces $L^p(\cA) \subseteq \Aff(\cA)$ for $p \in [1,\infty)$ which share many properties of the classical $L^p$ spaces.  There is also a natural identification of $A_*$ with $L^1(\cA)$.  See \S \ref{subsec:affiliated} and the references therein for details.
	\end{remark}
	
	\subsection{$\mathrm{W}^*$-embeddings, trace-preserving conditional expectations, and $\mathrm{W}^*$-isomorphisms}
	
	\begin{notation}
		If $\cA = (A,\tau)$ is a tracial $\mathrm{W}^*$-algebra, we will use the notation $L^\infty(\cA)$ for $A$ and $\tau_{\cA}$ for $\tau$ when it is convenient to avoid naming $A$ and $\tau$ explicitly.  In particular, the norm on $A$ will be denoted $\norm{\cdot}_{L^\infty(\cA)}$.  Furthermore, we will treat $L^\infty(\cA)$ as a subspace of $L^2(\cA)$.  We will also write $ab$ rather than $\pi_\ell(a) \widehat{b}$ and $ba$ rather than $\pi_r(a) \widehat{b}$ for $a \in L^\infty(\cA)$ and $b \in L^2(\cA)$.  Finally, we write $a^*$ instead of $J(a)$ for $a \in L^2(\cA)$.  We denote by $L^2(\cA)_{\sa}$ the real subspace of $L^2(\cA)$ consisting of those elements fixed by $J$.
	\end{notation}
	
	\begin{definition}
		Let $\cA$ and $\cB$ be tracial $\mathrm{W}^*$-algebras.  A linear map $\phi: L^\infty(\cA) \to L^\infty(\cB)$ is said to be \emph{trace-preserving} if $\tau_{\cA} = \tau_{\cB} \circ \phi$.
	\end{definition}
	
	\begin{lemma}[{See \cite[Lemma 1.5.11]{BrownOzawa2008} and \cite[\S 9.1]{ADP}}] \label{lem:inclusionconditionalexpectation}
		Let $\cA$ and $\cB$ be tracial $\mathrm{W}^*$-algebras.  Let $\phi: L^\infty(\cA) \to L^\infty(\cB)$ be a trace-preserving unital $*$-homomorphism.  Then
		\begin{enumerate}[(1)]
			\item $\phi$ extends to an isometry $L^2(\cA) \to L^2(\cB)$, and in particular $\phi$ is injective on $A$.
			\item $\phi$ is a contraction $L^\infty(\cA) \to L^\infty(\cB)$.
			\item The adjoint map $E = \phi^*: L^2(\cB) \to L^2(\cA)$ restricts to a map $L^\infty(\cB) \to L^\infty(\cA)$ that is contractive with respect to the $L^\infty$ norm.
			\item We have $E[b^*] = E[b]^*$ for $b \in L^\infty(\cB)$, and in fact also for $b \in L^2(\cB)$.
			\item $E$ is a bimodule map over $L^\infty(\cA)$, that is, for $a \in L^\infty(\cA)$ and $b \in L^2(\cB)$, we have $E(\phi(a)b) = a E(b)$ and $E(b \phi(a)) = E(b)a$.
			\item $E$ is unital ($E(1) = 1$) and trace-preserving ($\tau_{\cA} \circ E = \tau_{\cB}$).
		\end{enumerate}
	\end{lemma}
	
	\begin{definition}
		In the situation of the previous lemma, we call $\phi$ a \emph{(tracial $\mathrm{W}^*$)-embedding $\cA \to \cB$} and $E$ the associated \emph{trace-preserving conditional expectation}.  (Note that both maps are unital by definition and the previous proposition.)
	\end{definition}
	
	\begin{remark}
		It turns out that a trace-preserving $*$-homomorphism $L^\infty(\cA) \to L^\infty(\cB)$ is automatically continuous with respect to the weak-$*$ topology, essentially because the weak-$*$ topology can be recovered from the action of $L^\infty(\cA)$ on $L^2(\cA)$ by Theorem \ref{thm:W*GNSrep}; see \cite{Dixmier1953} or \cite[Proposition 2.6.4]{ADP}.  For similar reasons, the trace-preserving conditional expectation is also weak-$*$ continuous.
	\end{remark}
	
	\begin{example}
		Suppose that $\cB = L^\infty(\Omega,\mathcal{F},P)$ for some probability space $(\Omega,\mathcal{F},P)$, where $\mathcal{F}$ is the $\sigma$-algebra associated to the measure.  Let $\mathcal{G}$ be a $\sigma$-subalgebra of $\mathcal{F}$.  Then there is an expectation-preserving inclusion $L^\infty(\Omega,\mathcal{G},P) \to L^\infty(\Omega,\mathcal{F},P)$.  This extends to a map on the $L^2$ spaces, and the adjoint of this map is the conditional expectation $E: L^2(\Omega,\mathcal{F},P) \to L^2(\Omega,\mathcal{G},P)$ sending $X$ to $E[X|\mathcal{G}]$.  The properties in Lemma \ref{lem:inclusionconditionalexpectation} then reduce to the well-known classical properties of conditional expectation.  For instance, (2) the conditional expectation is contractive on $L^\infty$, (3) The conditional expectation respects complex conjugation, (4) ff $X \in L^2(\Omega,\mathcal{F},P)$ and $Y \in L^\infty(\Omega,\mathcal{G},P)$, then $E[XY | \mathcal{G}] = E[X | \mathcal{G}] Y$, (5) the conditional expectation is expectation-preserving: $E[E[X|\mathcal{G}]] = E[X]$.
	\end{example}
	
	\begin{notation} \label{not:inclusionCE}
		If $\cA$ and $\cB$ are tracial $\mathrm{W}^*$-algebras, we say that $\cA \subseteq \cB$ if $L^\infty(\cA) \subseteq L^\infty(\cB)$, the addition, product, $*$-operation and weak-$*$ topology for $L^\infty(\cA)$ are the restrictions of those from $L^\infty(\cB)$, and $\tau_{\cA} = \tau_{\cB}|_{L^\infty(\cA)}$.  In this case, we denote the conditional expectation $\cB \to \cA$ by $E_{\cA}$.
	\end{notation}
	
	As the paper will often deal with $m$-tuples of self-adjoint elements of $L^2$, we introduce the following convention to simplify notation.
	
	\begin{notation} \label{not:maptuples}
		If $\cA$ and $\cB$ are tracial $\mathrm{W}^*$-algebras and $\phi: L^\infty(\cA) \to L^\infty(\cB)$ is a tracial $\mathrm{W}^*$-embedding or a trace-preserving conditional expectation, then we will use the same letter $\phi$ to denote the extension of the map to the $L^2$ spaces.  Furthermore, if $X = (X_1,\dots,X_m) \in L^2(\cA)_{\sa}^m$, then we will write $\phi(X) = (\phi(X_1),\dots,\phi(X_m))$.
	\end{notation}
	
	\begin{definition}
		A tracial $\mathrm{W}^*$-embedding $\phi: \cA \to \cB$ is said to be a \emph{tracial $\mathrm{W}^*$-isomorphism} if it is bijective and the inverse map is also a tracial $\mathrm{W}^*$-embedding.
	\end{definition}
	
	For reasons of mathematical logic, the class of tracial $\mathrm{W}^*$-algebras is not a set.  However, it will be convenient for us in \S \ref{subsec:Econvex} to have a \emph{set} of isomorphism class representatives of tracial $\mathrm{W}^*$-algebras with separable predual.
	
	\begin{lemma} \label{lem:setofrepresentatives}
		There exists a set $\mathbb{W}$ of tracial $\mathrm{W}^*$-algebras, such that
		\begin{enumerate}[(1)]
			\item the elements of $\mathbb{W}$ are pairwise non-isomorphic,
			\item for every tracial $\mathrm{W}^*$-algebra with separable predual, there is a tracial $\mathrm{W}^*$-isomorphism to some element of $\mathbb{W}$.
		\end{enumerate}
	\end{lemma}
	
	\begin{proof}[Sketch of proof]
		We saw earlier that if $\cA = (A,\tau)$ is a tracial $\mathrm{W}^*$-algebra with separable predual, then there is a $\mathrm{W}^*$-embedding $A \to B(H_{\cA})$ (here by $\mathrm{W}^*$-embedding, we mean an injective normal $*$-homomorphism in the theory of von Neumann algebras).  Also, it is well-known (see e.g.\ \cite{Sakai1971}) that if $A$ has separable predual, then $H_{\cA} \cong L^2(\cA)$ is separable and hence isomorphic as a Hilbert space to $\ell^2(\N)$.  Therefore, $A$ is isomorphic to some $\mathrm{W}^*$-subalgebra of $B(\ell^2(\N))$.  Let $S_1$ be the set of $\mathrm{W}^*$-subalgebras of $B(\ell^2(\N))$ (which is a subset of the power set of $B(\ell^2(\N))$).  Let $S_2$ be the set of pairs $\{(A,\tau): A \in S_1, \tau: A \to \C \text{ faithful normal trace}\}$.  If $(A,\tau) \in S_2$, then the adjoint of the inclusion map produces a map from the space $B(\ell^2(\N))_*$ of trace class operators to $A_* \cong L^1(A,\tau)$, and hence $A_*$ is separable.  Thus, $S_2$ is a set of tracial $\mathrm{W}^*$-algebras such that every tracial $\mathrm{W}^*$-algebra with separable predual is isomorphic to some element of $S_2$.  Finally, observe that tracial $\mathrm{W}^*$-isomorphism defines an equivalence relation on $S_2$, and let $S_3$ be the set of equivalence classes.
	\end{proof}
	
	\subsection{Amalgamated free products}
	
	Next, we explain the definition of free independence with amalgamation.  This is an analog of conditional independence in classical probability theory.  For background see for instance \cite{VDN1992} or \cite[\S 4.7]{BrownOzawa2008}.
	
	\begin{definition} \label{def:freeproduct}
		Let $\cA = (A,\tau)$ be a tracial $\mathrm{W}^*$-algebra.  Let $B$, $A_1$, \dots, $A_N$ be $\mathrm{W}^*$-subalgebras of $A$ with $B \subseteq A_j$ for every $j$.  Let $\cB = (B,\tau|_B)$ and let $E_{\cB}: \cA \to \cB$ be the trace-preserving conditional expectation.  We say that $A_1$, \dots, $A_N$ are \emph{freely independent with amalgamation over $\cB$} if the following condition holds:  Whenever $\ell \in \N$ and $i_1$, \dots, $i_\ell \in \{1,\dots,N\}$ with $i_1 \neq i_2$, $i_2 \neq i_3$, \dots, $i_{\ell-1} \neq i_\ell$ and $a_j \in A_{i_j}$ with $E_{\cB}[a_j] = 0$ for $j = 1$, \dots, $\ell$, then $E_{\cB}[a_1 \dots a_\ell] = 0$.
	\end{definition}
	
	\begin{proposition} \label{prop:amalgamatedfreeproduct}
		Let $\cB = (B,\sigma)$ be a tracial $\mathrm{W}^*$-algebra.  For $j = 1$, \dots, $N$, let $\cA_j = (A_j,\tau_j)$ be a tracial $\mathrm{W}^*$-algebra and let $\iota_j: \cB \to \cA_j$ be a tracial $\mathrm{W}^*$-embedding.  Then there exists a tracial $\mathrm{W}^*$-algebra $\cA = (A,\tau)$ and tracial $\mathrm{W}^*$-embeddings $\iota: \cB \to \cA$ and $\phi_j: \cA_j \to \cA$ such that $\iota = \phi_j \circ \iota_j$ for all $j$, and such that $\phi_1(A_1)$, \dots, $\phi_N(A_N)$ are freely independent in $\cA$ with amalgamation over $\iota(B)$.  Moreover, $(\cA,\tau,\iota,\phi_1,\dots,\phi_N)$ are unique up to a canonical isomorphism; in other words, if $(\tilde{\cA},\tilde{\tau},\tilde{\iota},\tilde{\phi}_1,\dots,\tilde{\phi}_N)$ is another such tuple, then there is a unique tracial $\mathrm{W}^*$-isomorphism $\pi: \cA \to \tilde{\cA}$ satisfying $\pi \circ \iota = \tilde{\iota}$ and $\pi \circ \phi_j = \tilde{\phi}_j$ for all $j$.
	\end{proposition}
	
	\begin{definition}
		If $\cB$, $\cA_1$, \dots, $\cA_N$, and $\cA$ are as above (with the specified maps $\iota$, $\phi_1$, \dots, $\phi_N$), then we say that $\cA$ is a \emph{free product of $\cA_1$, \dots, $\cA_N$ with amalgamation over $\iota_1(\cB)$, \dots, $\iota_N(\cB)$}.
	\end{definition}
	
	In the case where $\cB = \C$, we refer to these concepts simply as \emph{free independence} and \emph{free products}.
	
	\subsection{Non-commutative laws and generators}
	
	Next, we describe the space of non-commutative laws.  A non-commutative law is the analog of a linear functional $\C[x_1,\dots,x_m] \to \R$ given by $f \mapsto \int f\,d\mu$ for some compactly supported measure on $\R^m$.  Instead of $\C[x_1,\dots,x_m]$, we use the non-commutative polynomial algebra in $m$ variables.
	
	\begin{definition}[Non-commutative polynomial algebra] \label{def:NCpolynomial}
		We denote by $\C\ip{x_1,\dots,x_m}$ the universal unital algebra generated by variables $x_1$, \dots, $x_m$.  As a vector space, $\C\ip{x_1,\dots,x_m}$ has a basis consisting of all products $x_{i_1} \dots x_{i_\ell}$ for $\ell \geq 0$ and $i_1$, \dots, $i_\ell \in \{1,\dots,m\}$.  We equip $\C\ip{x_1,\dots,x_m}$ with the unique $*$-operation such that $x_j^* = x_j$; more explicitly, the $*$-operation is defined on monomials by $(x_{i_1} \dots x_{i_\ell})^* = x_{i_\ell}^* \dots x_{i_1}^*$.
	\end{definition}
	
	\begin{definition}[Non-commutative law] \label{def:NClaw}
		A linear functional $\lambda: \C\ip{x_1,\dots,x_m}$ is said to be \emph{exponentially bounded} if there exists $R > 0$ such that $|\lambda(x_{i_1} \dots x_{i_\ell})| \leq R^\ell$ for all $\ell \in \N_0$ and $i_1$, \dots, $i_\ell \in \{1,\dots,m\}$, and in this case we say $R$ is an \emph{exponential bound} for $\lambda$.  A \emph{non-commutative law} is a unital, positive, tracial, exponentially bounded linear functional $\lambda: \C\ip{x_1,\dots,x_m} \to \C$. We denote the space of non-commutative laws by $\Sigma_m$, and we equip it with the weak-$*$ topology (that is, the topology of pointwise convergence on $\C\ip{x_1,\dots,x_m}$).  We denote by $\Sigma_{m,R}$ the subset of $\Sigma_m$ comprised of non-commutative laws with exponential bound $R$.
	\end{definition}
	
	\begin{observation}
		The space $\Sigma_{m,R}$ is convex, compact, and metrizable.
	\end{observation}
	
	\begin{observation}
		Let $A$ be a $*$-algebra and $X = (X_1, \dots, X_m) \in A_{\sa}^m$.  Then there is a unique $*$-homomorphism $\pi_{X}: \C\ip{x_1,\dots,x_m} \to \cA$ such that $\pi_{X}(x_j) = X_j$ for $j = 1$, \dots, $m$.
	\end{observation}
	
	\begin{definition}[Non-commutative law of an $m$-tuple]
		Let $\cA$ be a tracial $\mathrm{W}^*$-algebra.  Let $X = (X_1,\dots,X_m) \in L^\infty(\cA)_{\sa}^m$.  Then we define $\lambda_{X}: \C\ip{x_1,\dots,x_m} \to \C$ by $\lambda_{X} = \tau \circ \pi_{X}$, where $\pi_X$ is the map defined in the previous observation.
	\end{definition}
	
	\begin{notation}
		If $\cA$ is a tracial $\mathrm{W}^*$-algebra and $X \in L^\infty(\cA)^m$, we write
		\[
		\norm{X}_{L^\infty(\cA)^m} := \max(\norm{X_j}_{L^\infty(\cA)}: j = 1,\dots,m).
		\]
	\end{notation}
	
	\begin{observation}
		If $\cA$ and $X$ are as above, then $\lambda_{X}$ is a non-commutative law with exponential bound $\norm{X}_\infty$.  Conversely, if $R$ is an exponential bound for $\lambda_{X}$, then
		\[
		\norm{X}_{L^\infty(\cA)_{\sa}^m} = \max_j \lim_{n \to \infty} \tau(X_j^{2n})^{1/2n} \leq R.
		\]
		Hence, $\norm{X}_\infty$ is the smallest exponential bound for $\lambda_{X}$ and in particular it is uniquely determined by $\lambda_{X}$.
	\end{observation}
	
	In the case of a single operator $X$, we can apply the spectral theorem to show that there is a unique probability measure $\mu_X$ on $\R$ satisfying
	\[
	\int_{\R} f\,d\mu_X = \tau(f(X)) \text{ for } f \in C_0(\R).
	\]
	Since $X$ is bounded, $\mu_X$ is compactly supported and thus makes sense to evaluate on polynomials.  If $p$ is a polynomial, then $\lambda_X[p] = \int_{\R} p\,d\mu_X$.  Thus, $\lambda_X$ is simply the linear functional on polynomials corresponding to the spectral distribution.
	
	We use the notation $\lambda_{X}$ in particular when $\cA = M_n(\C)$.  We denote by $\tr_n$ the normalized trace $(1/n) \Tr$ on $M_n(\C)$; recall that this is the unique (unital) trace on $M_n(\C)$.  Thus, for any $X \in M_n(\C)_{\sa}^m$, a non-commutative law $\lambda_{X}$ is unambiguously specified by the previous definition.  In the $m = 1$ case, the non-commutative law is given by the empirical spectral distribution.  Note that when $\mathrm{X}$ is a random $m$-tuple of matrices, we will use the notation $\lambda_{X}$ by default to refer to the empirical non-commutative law, that is, the (random) non-commutative law of $X$ with respect to $\tr_n$.
	
	The next proposition shows that any non-commutative law can be realized by a self-adjoint $m$-tuple in some tracial $\mathrm{W}^*$-algebra.  This is a version of the \emph{Gelfand-Naimark-Segal construction} (or \emph{GNS construction}).  A proof can be found in \cite[Proposition 5.2.14(d)]{AGZ2009}.
	
	\begin{proposition}[GNS construction for non-commutative laws] \label{prop:GNS}
		Let $\lambda \in \Sigma_{m,R}$.  Then we may define a semi-inner product on $\C\ip{x_1,\dots,x_m}$ by
		\[
		\ip{p,q}_\lambda = \lambda(p^*q).
		\]
		Let $H_\lambda$ be the separation-completion of $\C\ip{x_1,\dots,x_m}$ with respect to this inner product, that is, the completion of $\C\ip{x_1,\dots,x_m} / \{p: \lambda(p^*p) = 0\}$, and let $[p]$ denote the equivalence class of a polynomial $p$ in $H_\lambda$.
		
		There is a unique unital $*$-homomorphism $\pi: \C\ip{x_1,\dots,x_m} \to B(H_\lambda)$ satisfying $\pi(p)[q] = [pq]$ for $p$, $q \in \C\ip{x_1,\dots,x_m}$.  Moreover, $\norm{\pi(x_j)} \leq R$.
		
		Let $X_j = \pi(x_j)$, let $X = (X_1,\dots,X_m)$ and let $A$ be the $\mathrm{W}^*$-subalgebra of $B(H_\lambda)$ generated by $X_1$, \dots, $X_m$.  Define $\tau: A \to \C$ by $\tau(Y) = \ip{[1], Y[1]}_\lambda$.  Then $\tau$ is a faithful normal trace on $A$, and hence $\cA = (A,\tau)$ is a tracial $\mathrm{W}^*$-algebra.
	\end{proposition}
	
	\begin{definition}
		In the situation of the previous proposition, we call $(\cA,X)$ the \emph{GNS realization of $\lambda$}.
	\end{definition}
	
	The tracial $\mathrm{W}^*$-algebra associated to $\lambda$ is canonical in the sense that any other construction would yield an isomorphic tracial $\mathrm{W}^*$-algebra.  The following lemma can be deduced from the well-known properties of the GNS representation associated to a faithful trace $\tau$ on a $\mathrm{W}^*$-algebra $\cA$ (which gives the so-called standard form of a tracial $\mathrm{W}^*$-algebra).
	
	\begin{lemma} \label{lem:lawisomorphism}
		Let $\cA$ and $\cB$ be tracial $\mathrm{W}^*$-algebras.  Let $X \in L^\infty(\cA)_{\sa}^m$ and $Y \in L^\infty(\cB)_{\sa}^m$ such that $\lambda_{X} = \lambda_{Y}$.  Let $\mathrm{W}^*(X)$ and $\mathrm{W}^*(Y)$ be the $\mathrm{W}^*$-subalgebras of $L^\infty(\cA)$ and $L^\infty(\cB)$ generated by $X$ and $\mathrm{Y}$ respectively, equipped with the traces $\tau_{\cA}|_{\mathrm{W}^*(X)}$ and $\tau_{\cB}|_{\mathrm{W}^*(Y)}$.  Then there is a unique tracial  $\mathrm{W}^*$-isomorphism $\rho: \mathrm{W}^*(X) \to \mathrm{W}^*(Y)$ such that $\rho(X_j) = Y_j$.
	\end{lemma}
	
	Here is a related lemma about generating sets for a tracial $\mathrm{W}^*$-algebra, which relies on the Kaplansky density theorem \cite[Theorem II.4.8]{TakesakiI}.
	
	\begin{lemma} \label{lem:generators}
		Let $\cA$ be a tracial $\mathrm{W}^*$-algebra.  Let $S \subseteq L^\infty(\cA)$.  Let $\mathrm{W}^*(S)$ be the smallest $\mathrm{W}^*$-subalgebra of $\cA$ containing $S$, which is equal to  the weak-$*$ closure of the unital $*$-algebra generated by $S$.  Then every $Z \in \mathrm{W}^*(S)$ can be approximated in the $L^2(\cA)$ norm by a sequence $Z_n$ in the unital $*$-algebra generated by $S$ such that $\norm{Z_n}_{L^\infty(\cA)} \leq \norm{Z}_{L^\infty(\cA)}$.  Furthermore, if $\phi: \cA \to \cB$ is a $\mathrm{W}^*$-embedding, then $\phi|_{\mathrm{W}^*(S)}$ is uniquely determined by $\phi|_S$.
	\end{lemma}
	
	In fact, the notion of generators for a $\mathrm{W}^*$-algebra extends to elements of $L^2(\cA)$.  For instance, for a self-adjoint tuple $X \in L^2(\cA)_{\sa}^m$, using the theory of affiliated operators sketched in \S \ref{subsec:affiliated}, it is valid to apply a bounded Borel function $f$ to $X_j$ through functional calculus, and $f(X_j)$ will be an element of $\cA$.  Thus, we may define $\mathrm{W}^*(X)$ as (for instance) the $\mathrm{W}^*$-subalgebra generated by $\arctan(X_1)$, \dots, $\arctan(X_m)$, and then, as one would hope, $X$ turns out to be in $L^2(\mathrm{W}^*(X))_{\sa}^m$.  See also \cite[p.~482-483]{BrownOzawa2008}.  We can state a characterization of $\mathrm{W}^*(X)$ without reference to affiliated operators as follows.
	
	\begin{lemma} \label{lem:L2generated}
		Let $\cA$ be a tracial $\mathrm{W}^*$-algebra and $X \in L^2(\cA)_{\sa}^m$.  Then there exists a unique smallest $\mathrm{W}^*$-subalgebra $B \subseteq L^\infty(\cA)$ such that $X \in B_{\sa}^m$.  We use the notation $\mathrm{W}^*(X)$ for $B$ and for $(B,\tau_{\cA}|_B)$ as needed.
	\end{lemma}

	\section{Duality for $L^2$ optimal couplings} \label{sec:duality}
	
	Our goal is to prove a version of the Monge-Kantorovich duality for the non-commutative version of the $L^2$ Wasserstein distance defined by Biane and Voiculescu \cite{BV2001}.  In \S \ref{subsec:basiccouplings} we recall the definitions of optimal couplings that were stated more succinctly in the introduction.  We define $E$-convex functions in \S \ref{subsec:Econvex} and the corresponding Legendre transform in \S \ref{subsec:Legendre}.  Then we prove the non-commutative Monge-Kantorovich duality in \S \ref{subsec:MKduality}, and as an application we prove a decomposition result for optimal couplings in \S \ref{subsec:decomposition}.
	
	\subsection{Wasserstein distance and optimal couplings} \label{subsec:basiccouplings}
	
	\begin{definition}[{Biane-Voiculescu \cite[\S 1.1]{BV2001}}] \label{def:boundedcoupling}
		Let $\mu$, $\nu \in \Sigma_m$ be non-commutative laws.  A \emph{coupling} of $\mu$ and $\nu$ is a triple $(\cA,X,Y)$ where $\cA$ is a tracial $\mathrm{W}^*$-algebra and $X$, $Y \in L^\infty(\cA)_{\sa}^m$ such that $\lambda_X = \mu$ and $\lambda_Y = \nu$.  For $\mu$, $\nu \in \Sigma_m$, the \emph{(non-commutative $L^2$) Wasserstein distance} $d_W^{(2)}(\mu,\nu)$ is the infimum of $\norm{X - Y}_{L^2(\cA)_{\sa}^m}$ over all couplings $(\cA,X,Y)$ for $\cA \in \mathbb{W}$.
	\end{definition}
	
	It is shown in \cite[Theorem 1.3]{BV2001} that $d_W^{(2)}$ is a metric on the set $\Sigma_m$, and for each $R > 0$, $\Sigma_{m,R}$ is complete in this metric.  However, as shown in \S \ref{subsec:twotopologies}, the topology generated by $d_W^{(2)}$ is strictly stronger than the weak-$*$ topology on $\Sigma_m$.  The notion of optimal couplings corresponding to the Wasserstein distance is as follows.
	
	\begin{definition} \label{def:boundedoptimalcoupling}
		A coupling $(\cA,X,Y)$ of two non-commutative laws $\mu$ and $\nu$ is \emph{optimal} if $\norm{X - Y}_{L^2(\cA)_{\sa}^m} = d_W^{(2)}(\mu,\nu)$.
	\end{definition}
	
	\begin{remark}
		As remarked in \cite{BV2001}, for every $\mu$, $\nu \in \Sigma_m$, some optimal coupling exists.  To see this, suppose $R > 0$ is an exponential bound for $\mu$ and $\nu$.  Note that that if $(\cA,X,Y)$ is a coupling and $\gamma$ is the joint law of $(X,Y)$, then $\norm{X - Y}_{L^2(\cA)_{\sa}^m} = \left( \sum_{j=1}^m \gamma((x_j - x_{m+j})^2)\right)^{1/2}$.  The space of joint laws $\gamma \in \Sigma_{2m,R}$ with marginals $\mu$ and $\nu$ is closed in $\Sigma_{2m,R}$ and therefore compact, and $\gamma \mapsto \left( \sum_{j=1}^m \gamma((x_j - x_{m+j})^2)\right)^{1/2}$ is continuous.  Thus, it achieves a minimum at some $\gamma^*$, and we obtain an optimal coupling $(\cA,X,Y)$ from the GNS construction with $\gamma^*$ (Proposition \ref{prop:GNS}).
	\end{remark}
	
	Just as in classical optimal transport theory, it is convenient to frame $L^2$ optimal couplings in terms of inner products rather than $L^2$ norms in order to relate them with Legendre transforms.  If $(\cA,X,Y)$ is a coupling of $\mu$ and $\nu$, then
	\[
	\norm{X - Y}_{L^2(\cA)_{\sa}^m}^2 = \norm{X}_{L^2(\cA)_{\sa}^m}^2 - 2\ip{X,Y}_{L^2(\cA)_{\sa}^m} + \norm{Y}_{L^2(\cA)_{\sa}^m}^2.
	\]
	Since $\norm{X}_{L^2(\cA)_{\sa}^m}^2$ and $\norm{Y}_{L^2(\cA)_{\sa}^m}^2$ are uniquely determined by $\mu$ and $\nu$, a coupling minimizes $\norm{X - Y}_{L^2(\cA)_{\sa}^m}$ if and only if it maximizes the inner product $\ip{X,Y}_{L^2(\cA)_{\sa}^m}$.  This motivates the following definition.
	
	\begin{definition}
		For $\mu$, $\nu \in \Sigma_m$, we denote by $C(\mu,\nu)$ the maximal value of $\ip{X,Y}_{L^2(\cA)_{\sa}^m}$ over all couplings $(\cA,X,Y)$ of $\mu$ and $\nu$.
	\end{definition}
	
	The preceding paragraph shows that
	\[
	d_W^{(2)}(\mu,\nu)^2 = \sum_{j=1}^m \mu(x_j^2) + \sum_{j=1}^m \nu(x_j^2) - 2 C(\mu,\nu).
	\]
	The goal of the section is to establish a duality result that $C(\mu,\nu)$ is the infimum of $\mu(f) + \nu(g)$ over certain pairs $(f,g)$ of $E$-convex functions.
	
	\subsection{$E$-convex functions} \label{subsec:Econvex}
	
	Fix a set $\mathbb{W}$ of isomorphism class representatives for tracial $\mathrm{W}^*$-algebras with separable predual, as was given by Lemma \ref{lem:setofrepresentatives}. (Although we are only considering a set of isomorphism class representatives, we make no identifications between different tracial $\mathrm{W}^*$-embeddings from a given $\cA \in \mathbb{W}$ to a given $\cB \in \mathbb{W}$.)
	
	\begin{definition}
		Let $S$ be a set.  A \emph{tracial $\mathrm{W}^*$-function with values in $S$} is tuple $f = (f^{\cA})_{\cA \in \mathbb{W}}$, where $f^{\cA}: L^2(\cA)_{\sa}^m \to S$ if whenever $\iota: \cA \to \cB$ is a tracial $\mathrm{W}^*$-embedding, we have $f^{\cA} = f^{\cB} \circ \iota$.  (Here the inclusion $\iota$ is understood to extend to a map $L^2(\cA)_{\sa}^m \to L^2(\cB)_{\sa}^m$ per Notation \ref{not:maptuples}.)
	\end{definition}
	
	Thus, roughly speaking, being a $\mathrm{W}^*$-function means that the evaluation of $f$ on some $X \in L^2(\cA)_{\sa}^m$ is independent of the ambient algebra.  Hence, in particular, for bounded operators, $f^{\cA}(X)$ only depends on the non-commutative law of $X$.
	
	Although the definition of $f$ only specifies $f^{\cA}$ when $\cA$ is in the set $\mathbb{W}$, it will sometimes be convenient to use the notation $f^{\cA}$ for a general tracial $\mathrm{W}^*$-algebra $\cA$ with separable predual.  Indeed, by our choice of $\mathbb{W}$, there exists an isomorphism $\phi$ from $\cA$ to some $\cB \in \mathbb{W}$.  We can then set $f^{\cA} = f^{\cB} \circ \phi$.  This is well-defined, that is, independent of the particular choice of $\phi$, because $f^{\cB} \circ \psi = f^{\cB}$ for every automorphism $\psi$ of $\cB$; this in turn follows from the definition of $\mathrm{W}^*$-functions since an automorphism $\psi$ is in particular an inclusion from $\cB$ into $\cB$.
	
	\begin{definition}
		A tracial $\mathrm{W}^*$-function $f = (f^{\cA})_{\cA \in \mathbb{W}}$ with values in $[-\infty,+\infty]$ is said to be \emph{$E$-convex} if either it is identically equal to $-\infty$ or the following conditions hold:
		\begin{enumerate}[(1)]
			\item For each $\cA$, $f^{\cA}$ is a convex and lower semi-continuous function $L^2(\cA)_{\sa}^m \to (-\infty,+\infty]$.
			\item If $\iota: \cA \to \cB$ is a trace-preserving embedding, and if $E = \iota^*: \cB \to \cA$ is the corresponding trace-preserving conditional expectation, then
			\[
			f^{\cA}(E[X]) \leq f^{\cB}(X)
			\]
			for $X \in \cB_{\sa}^m$.  (Here $E$ is understood to extend to a map $L^2(\cB)_{\sa}^m \to L^2(\cA)_{\sa}^m$ per Notation \ref{not:maptuples}.)
		\end{enumerate}
	\end{definition}
	
	\begin{example}
		For $t \in (0,\infty)$, let $q_t^{\cA}(X) = (1/2t) \norm{X}_{L^2(\cA)_{\sa}^m}^2$.  Then $q_t$ is $E$-convex.  Indeed, it is convex because of the Cauchy-Schwarz and arithmetic-geometric mean inequalities.  It is clearly continuous.  Finally, it satisfies monotonicity under conditional expectation because conditional expectations are contractive in $\norm{\cdot}_{L^2(\cA)_{\sa}^m}$.
	\end{example}
	
	We next explain an equivalent characterization of $E$-convexity using subgradient vectors.
	
	\begin{definition}
		If $H$ is a real Hilbert space and $f: H \to (-\infty,\infty]$ is a function, we say that $y \in H$ is a \emph{subgradient} for $f$ at $x$ if
		\[
		f(x') \geq f(x) + \ip{y, x' - x} \text{ for all } x' \in H.
		\]
		We define the \emph{subdifferential} $\eth f(x)$ as the set of subgradient vectors at $x$.
	\end{definition}
	
	The following facts are well-known in convex analysis.
	
	\begin{lemma} \label{lem:Hilbertsubgradient}
		Let $H$ be a Hilbert space.  If $f: H \to [-\infty,\infty]$ is convex and lower semi-continuous and $f(x)$ is finite, then $\eth f(x)$ is nonempty, closed, and convex.  Conversely, $f: H \to (-\infty,\infty)$ and $\eth f$ is nonempty for every $x$, then $f$ is convex.
	\end{lemma}
	
	Analogously, we will show that $E$-convex $\mathrm{W}^*$-functions are characterized by the existence of a subgradient vector $Y$ to $f^{\cA}$ at $X$ such that $Y \in L^2(\mathrm{W}^*(X))_{\sa}^m$ (where $\mathrm{W}^*(X)$ is given by Lemma \ref{lem:L2generated}).  In addition, we handle the case where $f$ can take the value $+\infty$.
	
	\begin{lemma} \label{lem:Econvex}
		Let $f$ be a $\mathrm{W}^*$-function taking values in $(-\infty,\infty)$.  Then $f$ is $E$-convex if and only if for each $\cA \in \mathbb{W}$ and $X \in L^2(\cA)_{\sa}^m$, there exists $Y \in L^2(\mathrm{W}^*(X))_{\sa}^m$ which is a subgradient vector to $f^{\cA}$ at $X$.  Here $\mathrm{W}^*(X)$ is given by Lemma \ref{lem:L2generated}.
	\end{lemma}
	
	\begin{proof}
		First, suppose that $f$ is $E$-convex. Fix $X \in L^2(\cA)_{\sa}^m$.  By Lemma \ref{lem:Hilbertsubgradient}, there exists some subgradient vector $Z$ to $f^{\cA}(X)$.  Let $\cB = \mathrm{W}^*(X)$, let $E: \cA \to \cB$ be the trace-preserving conditional expectation, and let $Y = E[Z]$.  Then for $X' \in L^2(\cA)_{\sa}^m$, we have
		\begin{align*}
			f^{\cA}(X') &\geq f^{\cB}(E[X']) = f^{\cA}(E[X']) \\
			&\geq f^{\cA}(X) + \ip{Z, E[X'] - X}_{L^2(\cA)_{\sa}^m} \\
			&= f^{\cA}(X) + \ip{Z, E[X' - X]}_{L^2(\cA)_{\sa}^m} \\
			&= f^{\cA}(X) + \ip{Y, X' - X}_{L^2(\cA)_{\sa}^m}.
		\end{align*}
		Thus, the desired subgradient condition holds.
		
		Conversely, suppose this subgradient condition holds.  Lower semi-continuity of $f^{\cA}$ follows from the existence of subgradient vectors.  For $X_0$, $X_1 \in L^2(\cA)_{\sa}^m$ and $t \in (0,1)$, we have
		\[
		f^{\cA}((1 - t) X_0 + t X_1) \leq (1 - t) f^{\cA}(X_0) + t f^{\cA}(X_1).
		\]
		because of the existence of a subgradient vector at $(1 - t) X_0 + t X_1$.  To check the monotonicity under conditional expectation, consider an embedding $\iota: \cA \to \cB$ and let $E: \cB \to \cA$ be the corresponding conditional expectation.  Let $X \in L^2(\cB)_{\sa}^m$ and let $X' = E[X]$.  By (1), there is a subgradient vector $Y$ to $f^{\cB}$ at the point $X'$ that is in $L^2(\mathrm{W}^*(X'))_{\sa}^m$, and in particular $Y \in L^2(\cA)_{\sa}^m$.  But then
		\begin{align*}
			f^{\cB}(X) &\geq f^{\cB}(X') + \ip{Y, X - X'}_{L^2(\cB)_{\sa}^m} \\
			&= f^{\cB}(E[X]) + \ip{Y, X - E[X]}_{L^2(\cB)_{\sa}^m} \\
			&= f^{\cB}(E[X]).
		\end{align*}
	\end{proof}
	
	\begin{remark} \label{rem:Econvexitytest}
		The same argument shows that for a $\mathrm{W}^*$-function taking values in $(-\infty,+\infty]$, $E$-convexity is equivalent to the combination of the following three conditions:
		\begin{enumerate}[(1)]
			\item For each $\cA \in \mathbb{W}$ and $X \in L^2(\cA)_{\sa}^m$, if $f^{\cA}(X) < \infty$, then there exists $Y \in L^2(\mathrm{W}^*(X))$ which is a subgradient vector to $f^{\cA}$ at $X$.
			\item For each $\cA$, the set $(f^{\cA})^{-1}((-\infty,M])$ is closed and convex in $L^2(\cA)_{\sa}^m$.
			\item If $\iota: \cA \to \cB$ is a tracial $\mathrm{W}^*$-embedding and $E = \iota^*: \cB \to \cA$ is the corresponding conditional expectation, then $f^{\cB}(X) < +\infty$ implies $f^{\cA}(E[X]) < +\infty$.
		\end{enumerate}
	\end{remark}
	
	\begin{remark}
		If $f$ is a tracial $\mathrm{W}^*$-function, then $f^{\cA}(UXU^*) = f^{\cA}(X)$ for every unitary $U$ in $L^\infty(\cA)$ and $X \in L^2(\cA)_{\sa}^m$; this is because conjugation by $U$ defines an automorphism of $\cA$ (hence in particular a tracial $\mathrm{W}^*$-embedding $\cA \to \cA$), and $f$ respects tracial $\mathrm{W}^*$-embeddings.
		
		If $f$ is $E$-convex, then this unitary invariance gives rise to a ``sum of commutators'' condition on subgradient vectors related to Lemma \ref{lem:sumofcommutators}.  More precisely, suppose $f$ is $E$-convex, $Y \in \eth f^{\cA}(X)$ and $U$ is a unitary in $L^\infty(\cA)$.  Then
		\[
		f^{\cA}(X) = f^{\cA}(UXU^*) \geq f^{\cA}(X) + \ip{UXU^* - X, Y}_{L^2(\cA)_{\sa}^m}.
		\]
		As in Lemma \ref{lem:sumofcommutators}, by taking $U = e^{itA}$ for $A \in L^\infty(\cA)_{\sa}$ and differentiating at $t = 0$, we obtain $\sum_{j=1}^m [X_j,Y_j] = 0$.
	\end{remark}
	
	The next lemma describes how the subdifferential interacts with conditional expectations.
	
	\begin{lemma} \label{lem:CEgradient1} 
		Let $f$ be an $E$-convex $\mathrm{W}^*$-function.  Let $\cA \in \mathbb{W}$ and $X \in L^2(\cA)_{\sa}^m$.  Let $\cB$ be a tracial $\mathrm{W}^*$-subalgebra of $\cA$.
		\begin{enumerate}[(1)]
			\item If $f^{\cB}(E_{\cB}[X]) = f^{\cA}(X)$, then
			\[
			\eth f^{\cB}(E_{\cB}[X]) = L^2(\cB)_{\sa}^m \cap \eth f^{\cA}(X).
			\]
			\item If $Y \in \eth f^{\cA}(X)$, then $E_{\mathrm{W}^*(X)}[Y] \in \eth f^{\cA}(X)$.
		\end{enumerate}
	\end{lemma}
	
	\begin{proof}
		(1) First, we show that $\eth f^{\cB}(E_{\cB}[X]) \subseteq L^2(\cB)_{\sa}^m \cap \eth f^{\cA}(X)$.  If $Y \in \eth f^{\cB}(E_{\cB}[X])$, then clearly $Y \in L^2(\cB)_{\sa}^m$.  Moreover, for all $Z \in L^2(\cA)$, we have
		\begin{align*}
			f^{\cA}(Z) &\geq f^{\cB}(E_{\cB}[Z]) \\
			&\geq f^{\cB}(E_{\cB}[X]) + \ip{Y, E_{\cB}[Z] - E_{\cB}[X]}_{L^2(\cB)_{\sa}^m} \\
			&= f^{\cA}(X) + \ip{Y, Z - X}_{L^2(\cA)_{\sa}^m},
		\end{align*}
		where we have used the fact that $E_{\cB}$ is self-adjoint and $E_{\cB}[Y] = Y$.  Hence, $Y \in \eth f^{\cA}(X)$ as desired.
		
		Conversely, to show that $L^2(\cB)_{\sa}^m \cap \eth f^{\cA}(X) \subseteq \eth f^{\cB}(E_{\cB}[X])$, suppose that $Y \in L^2(\cB)_{\sa}^m \cap \eth f^{\cA}(X)$.  Then for $Z \in L^2(\cB)_{\sa}^m$,
		\begin{align*}
			f^{\cB}(Z) &= f^{\cA}(Z) \\
			&\geq f^{\cA}(X) + \ip{Y, Z - X}_{L^2(\cA)_{\sa}^m} \\
			&= f^{\cB}(E_{\cB}[X]) + \ip{Y, Z - E_{\cB}[X]}_{L^2(\cB)_{\sa}^m},
		\end{align*}
		because $E_{\cB}$ is a self-adjoint operator on $L^2(\cA)$ and $Y \in L^2(\cB)_{\sa}^m$.
		
		(2) Let $\cB = \mathrm{W}^*(X)$ (where the trace is given by the restriction of $\tau_{\cA}$). Let $Z \in L^2(\cB)_{\sa}^m$.  Then
		\begin{align*}
			f^{\cB}(Z)  &= f^{\cA}(Z) \\
			&\geq f^{\cA}(X) + \ip{Y,Z - X}_{L^2(\cA)_{\sa}^m} \\
			&= f^{\cB}(X) + \ip{E_{\cB}[Y], Z - X}_{L^2(\cB)_{\sa}^m}.
		\end{align*}
		Thus, $E_{\cB}[Y] \in \eth f^{\cB}(X)$, and so by (1), $E_{\cB}[Y] \in \eth f^{\cA}(X)$.
	\end{proof}
	
	\begin{lemma}
		Let $f$ be an $E$-convex $\mathrm{W}^*$-function.  Let $\cA \in \mathbb{W}$ and $X\in L^2(\cA)_{\sa}^m$.
		\begin{enumerate}[(1)]
			\item There exists a unique $Y \in \eth f^{\cA}(X)$ of minimal $L^2$-norm.
			\item The $Y$ from (1) satisfies $Y \in L^2(\mathrm{W}^*(X))_{\sa}^m$.
			\item Let $\cB = \mathrm{W}^*(Y)$ as described in Lemma \ref{lem:L2generated}, where $Y$ is as in (1).  Then $f^{\cB}(E_{\cB}[X]) = f^{\cA}(X)$ and $\cB = \mathrm{W}^*(E_{\cB}[X])$.
		\end{enumerate}
	\end{lemma}
	
	\begin{proof}
		(1) Because $\eth f^{\cA}(X)$ is a closed convex set, it has a unique element of minimal $L^2$-norm.
		
		(2) Let $\cC = \mathrm{W}^*(X)$. Let $Y' = E_{\cC}[Y]$.  We claim that $Y' \in \eth f^{\cC}(X)$.  Let $Z \in L^2(\cC)_{\sa}^m$.  Then
		\begin{align*}
			f^{\cC}(Z)  &= f^{\cA}(Z) \\
			&\geq f^{\cA}(X) + \ip{Y,Z - X}_{L^2(\cA)_{\sa}^m} \\
			&= f^{\cC}(X) + \ip{Y', Z - X}_{ L^2(\cC)_{\sa}^m}.
		\end{align*}
		Thus, $Y' \in \eth f^{\cC}(X)$.  By the previous lemma, $Y' \in \eth f^{\cA}(X)$.  But because $Y$ has minimal norm, we have $\norm{E_{\cC}[Y]}_{L^2(\cA)_{\sa}^m} = \norm{Y}_{L^2(\cA)_{\sa}^m}$, hence $E_{\cC}[Y] = Y$, so $Y \in L^2(\cC)_{\sa}^m$.
		
		(3) First, we show that $f^{\cB}(E_{\cB}[X]) = f^{\cA}(X)$.  By $E$-convexity, $f^{\cB}(E_{\cB}[X]) \leq f^{\cA}(X)$.  Conversely,
		\[
		f^{\cB}(E_{\cB}[X]) = f^{\cA}(E_{\cB}[X]) \geq f^{\cA}(X) + \ip{Y, E_{\cB}[X] - X}_{L^2(\cA)_{\sa}^m} = f^{\cA}(X).
		\]
		Thus, by Lemma \ref{lem:CEgradient1} (1), $Y \in \eth f^{\cB}(E_{\cB}[X])$.  Now letting $\cD = \mathrm{W}^*(E_{\cB}[X])$ with the trace $\tau|_{\cD}$, Lemma \ref{lem:CEgradient1} (2) implies that
		$E_{\cD}[Y] \in \eth f^{\cB}(E_{\cB}[X])$, hence also $E_{\cD}[Y] \in \eth f^{\cA}(X)$ by Lemma \ref{lem:CEgradient1} (1).  Because $Y$ was chosen to have minimal norm, we have $E_{\cD}[Y] = Y$, and thus, $\cD \supseteq \mathrm{W}^*(Y) = \cB$ by the characterization of $\mathrm{W}^*(Y)$ given in Lemma \ref{lem:L2generated}.  Hence, $\cB = \cD = \mathrm{W}^*(E_{\cB}[X])$.
	\end{proof}
	
	\subsection{Legendre transforms} \label{subsec:Legendre}
	
	\begin{definition}
		We define the \emph{Legendre transform} as the tuple $\mathcal{L}f = (\mathcal{L}f^{\cA})_{\cA \in \mathbb{W}}$ by
		\[
		\mathcal{L}f^{\cA}(X) = \sup \{ \ip{\iota(X),Y} - f^{\cB}(Y): \cB \in \mathbb{W}, \iota: \cA \to \cB \text{ a tracial $\mathrm{W}^*$-embedding}, Y \in L^2(\cB)_{\sa}^m\}
		\]
		for $X \in L^2(\cA)_{\sa}^m$.
	\end{definition}
	
	\begin{example}
		Consider again $q_t^{\cA}(X) = (1/2t) \norm{X}_{L^2(\cA)_{\sa}^m}^2$.  A standard computation with norms and inner products shows that $\mathcal{L} q_t = q_{1/t}$.
	\end{example}
	
	\begin{proposition} \label{prop:Legendre}
		Let $f$ be a tracial $\mathrm{W}^*$-function.
		\begin{enumerate}[(1)]
			\item The Legendre transform $\mathcal{L}f$ is an $E$-convex tracial $\mathrm{W}^*$-function.
			\item If $f \leq g$, then $\mathcal{L}f \geq \mathcal{L}g$.
			\item We have $\mathcal{L}^2 f \leq f$ with equality if and only if $f$ is $E$-convex.
			\item $\mathcal{L}^2 f$ is the maximal $E$-convex function that is less than or equal to $f$.
		\end{enumerate}
	\end{proposition}
	
	\begin{proof}
		(1) If $f$ is identically equal to $-\infty$ or $+\infty$, $\mathcal{L}f$ will be $+\infty$ or $-\infty$ respectively and there is nothing to prove.  Hence, assume that $f$ attains some finite value at some $Y \in L^2(\cB)_{\sa}^m$ for some $\cB \in \mathbb{W}$.
		
		For any $\cA \in \mathbb{W}$, the free product $\cA * \cB$ is isomorphic to some $\cC \in \mathbb{W}$.  Let $\iota_1: \cA \to \cC$ and $\iota_2: \cB \to \cC$ be the corresponding tracial $\mathrm{W}^*$-embeddings.  Then
		\[
		\mathcal{L}f^{\cA}(X) \geq \ip{\iota_1(X),\iota_2(Y)}_{L^2(\cC)_{\sa}^m} - f^{\cC}(\iota_2(Y)) > -\infty,
		\]
		since $f^{\cC}(\iota_2(Y)) = f^{\cB}(Y) < +\infty$.  Hence, $\mathcal{L} f$ is never equal to $-\infty$.
		
		For each $\cA$, the function $\mathcal{L} f^{\cA}$ is a supremum of affine functions, and therefore it is convex and lower semi-continuous.
		
		Let $\iota: \cA \to \cB$ be a tracial $\mathrm{W}^*$-embedding and let $E: \cB \to \cA$ be the corresponding trace-preserving conditional expectation.  Let $X \in L^2(\cB)_{\sa}^m$.  Let $\tilde{\iota}: \cA \to \tilde{\cB}$ be another inclusion.  Let $\cM \in \mathbb{W}$ be isomorphic to the amalgamated free product  of $\cB$ and $\tilde{\cB}$ over the subalgebra $\cA$ (or more precisely, over the images of $\iota(\cA) \subseteq \cB$ and $\tilde{\iota}(\cA) \subseteq \tilde{\cB}$) as in Proposition \ref{prop:amalgamatedfreeproduct}.  Let $\rho: \cB \to \cM$ and $\tilde{\rho}: \tilde{\cB} \to \cM$ be the inclusions.  Then for $Y \in L^2(\tilde{\cB})$,
		\begin{align*}
			\mathcal{L}f^{\cB}(X) &\geq \ip{\rho(X), \tilde{\rho}(Y)}_{L^2(\cM)_{\sa}^m}- f^{\cM}(\tilde{\rho}(Y)) \\
			&= \ip{\tilde{\iota} \circ E(X), Y}_{L^2(\tilde{\cB})_{\sa}^m} - f^{\tilde{\cB}}(Y),
		\end{align*}
		where we have used free independence with amalgamation to compute the inner product, and we have used the fact that $f$ is a tracial $\mathrm{W}^*$-function.  Because $\tilde{\iota}: \cA \to \tilde{\cB}$ and $Y$ were arbitrary, we have
		\[
		\mathcal{L}f^{\cB}(X) \geq \mathcal{L} f^{\cA}(E(X)),
		\]
		which establishes condition (2) in the definition of $E$-convexity.
		
		It only remains to show that $\mathcal{L}f$ is a tracial $\mathrm{W}^*$-function.  Suppose $\iota: \cA \to \cB$ is a tracial $\mathrm{W}^*$-inclusion.  If $\iota': \cB \to \cC$ is a tracial $\mathrm{W}^*$-inclusion, then so is $\iota' \circ \iota$, which implies that
		\[
		\mathcal{L} f^{\cA}(X) \geq \sup_{\substack{\iota': \cB \to \cC \\ Y \in L^2(\cC)_{\sa}^m}} \ip{\iota' \circ \iota(X),Y}_{L^2(\cC)_{\sa}^m} - f^{\cC}(Y) = \mathcal{L} f^{\cB}(\iota(X)).
		\]
		If $E: \cB \to \cA$ is the conditional expectation corresponding to $\iota$, then by the preceding argument
		\[
		\mathcal{L} f^{\cA}(X) = \mathcal{L} f^{\cA}(E \circ \iota(X)) \leq \mathcal{L} f^{\cB}(\iota(X)).
		\]
		Thus, $\mathcal{L} f^{\cA} = \mathcal{L} f^{\cB} \circ \iota$, so $\mathcal{L} f$ is a tracial $\mathrm{W}^*$-function.
		
		(2)  This is immediate from the definition and the properties of suprema and infima.
		
		(3) By definition of $\mathcal{L}f$, for every $\cA \in \mathbb{W}$ and $X$, $Y \in L^2(\cA)_{\sa}^m$, we have
		\[
		\mathcal{L}f^{\cA}(X) \geq \ip{X,Y}_{L^2(\cA)_{\sa}^m} - f^{\cA}(Y),
		\]
		hence
		\[
		\mathcal{L}f^{\cA}(X) + f^{\cA}(Y) \geq \ip{X,Y}_{L^2(\cA)_{\sa}^m}.
		\]
		Hence, given an inclusion $\iota$ of $\cA$ into $\cB$ and $Y \in L^2(\cA)_{\sa}^m$ and $X \in L^2(\cB)_{\sa}^m$, we have
		\[
		f^{\cA}(Y) = f^{\cB}(\iota(Y)) \geq \ip{\iota(Y),X}_{L^2(\cB)_{\sa}^m} - \mathcal{L}f^{\cB}(X).
		\]
		Taking the supremum on the right-hand side, $f^{\cA}(Y) \geq \mathcal{L}^2 f^{\cA}(Y)$.  Thus, $f \geq \mathcal{L}^2 f$.
		
		Now suppose that $f$ is $E$-convex, and we must show that $f = \mathcal{L}^2 f$.  If $f$ is identically $-\infty$ or $+\infty$, there is nothing to prove.  Otherwise, fix $\cA$.  Because $f^{\cA}$ is convex and lower semi-continuous, classical results about convex functions tell us that $f^{\cA}$ can be expressed as the supremum of a family of affine functions $(g_\alpha)_{\alpha \in I}$, where
		\[
		g_\alpha(X) = \ip{X,Z_\alpha}_{L^2(\cA)_{\sa}^m} + c_\alpha
		\]
		with $Z_\alpha \in L^2(\cA)_{\sa}^m$ and $c_\alpha \in \R$.  Let $\iota: \cA \to \cB$ be an inclusion and let $E: \cB \to \cA$ be the corresponding conditional expectation.  If $Y \in L^2(\cB)_{\sa}^m$, then by the $E$-convexity property
		\[
		\ip{\iota(Z_\alpha),Y}_{L^2(\cB)_{\sa}^m} - f^{\cB}(Y) \leq \ip{\iota(Z_\alpha),Y}_{L^2(\cB)_{\sa}^m} - f^{\cA}(E(Y)) \leq \ip{Z_\alpha,E(Y)}_{L^2(\cA)_{\sa}^m} - \ip{Z_\alpha,E(Y)}_{L^2(\cA)_{\sa}^m} - c_\alpha = -c_\alpha.
		\]
		Therefore, $\mathcal{L} f^{\cA}(Z_\alpha) \leq -c_\alpha$, which implies that
		\[
		\mathcal{L}^2 f^{\cA}(X) \geq \ip{X,Z_\alpha}_{L^2(\cA)_{\sa}^m} - \mathcal{L} f^{\cA}(Z_\alpha) \geq \ip{X,Z_\alpha}_{L^2(\cA)_{\sa}^m} + c_\alpha = g_\alpha(X).
		\]
		Therefore,
		\[
		\mathcal{L}^2 f^{\cA}(X) \geq \sup_{\alpha \in I} g_\alpha(X) = f^{\cA}(X).
		\]
		So $\mathcal{L}^2 f = f$ as desired.  Conversely, if $f = \mathcal{L}^2 f$, then $f$ is $E$-convex because it is the Legendre transform of some function.
		
		(4) We already showed that $\mathcal{L}^2 f$ is $E$-convex and $\mathcal{L}^2 f \leq f$.  Moreover, suppose $g$ is $E$-convex and $g \leq f$.  Then $\mathcal{L} g \geq \mathcal{L} f$ and hence $\mathcal{L}^2 g \leq \mathcal{L}^2 f$ by (2).  Meanwhile, $g = \mathcal{L}^2 g$ by (3), and therefore $g = \mathcal{L}^2 g \leq \mathcal{L}^2 f$.
	\end{proof}
	
	\begin{remark} \label{rem:EconvexLegendre}
		It follows from $E$-convexity that for every $\iota: \cA \to \cB$ and $X \in L^2(\cA)_{\sa}^m$ and $Y \in L^2(\cB)_{\sa}^m$, we have
		\[
		\ip{\iota(X),Y}_{L^2(\cB)_{\sa}^m} - f^{\cB}(Y) \leq \ip{X,E[Y]}_{L^2(\cA)_{\sa}^m} - f^{\cA}(E[Y]),
		\]
		where $E: \cB\to \cA$ is the conditional expectation corresponding to $\iota$.  Therefore,
		\[
		\mathcal{L} f^{\cA}(X) = \sup_{Y \in L^2(\cA)_{\sa}^m} \left( \ip{X,Y}_{L^2(\cA)_{\sa}^m} - f^{\cA}(Y) \right).
		\]
		Hence, if $f$ is $E$-convex, there is no need to consider a larger $\mathrm{W}^*$-algebra when computing the Legendre transform, and moreover $\mathcal{L} f^{\cA}$ agrees with the classical Legendre transform of $f^{\cA}$ as a function on the real Hilbert space $L^2(\cA)_{\sa}^m$.
	\end{remark}
	
	\begin{remark}
	In fact, the argument of Proposition \ref{prop:Legendre} can be used to prove slightly stronger statements:
	\begin{itemize}
		\item If $(f^{\cA})_{\cA\in \mathbb{W}}$ is any collection of functions, then $\mathcal{L}f$ is convex, and for any tracial $\mathrm{W}^*$-embedding $\iota: \cA \to \cB$, it satisfies $\mathcal{L}f^{\cA} \geq \mathcal{L}f^{\cB} \circ \iota$.
		\item If $(f^{\cA})_{\cA \in \mathbb{W}}$ satisfies $\mathcal{L}f^{\cA} \geq \mathcal{L}f^{\cB} \circ \iota$ for every embeddings $\iota: \cA \to \cB$, then $\mathcal{L}f$ is $E$-convex.
		\item In particular, if $(f^{\cA})_{\cA \in \mathbb{W}}$ is any collection of functions, then $\mathcal{L}^2 f$ is $E$-convex.
	\end{itemize}
	\end{remark}
	
	The next lemma states the relationship between Legendre transforms and subgradients, which is exactly analogous to the behavior of classical Legendre transforms.  We will use this lemma many times.
	
	\begin{lemma} \label{lem:Legendresubgradient}
		Let $f$ be an $E$-convex $\mathrm{W}^*$-function, let $\cA \in \mathbb{W}$ and $X, Y \in L^2(\cA)_{\sa}^m$.  Then the following are equivalent:
		\begin{enumerate}[(1)]
			\item $f^{\cA}(X) + \mathcal{L}f^{\cA}(Y) = \ip{X,Y}_{L^2(\cA)_{\sa}^m}$.
			\item $Y \in \eth f^{\cA}(X)$.
			\item $X \in \eth \mathcal{L}f^{\cA}(Y)$.
		\end{enumerate}
	\end{lemma}
	
	\begin{proof}
		(1) $\implies$ (2). Suppose that $f^{\cA}(X) + \mathcal{L}f^{\cA}(Y) = \ip{X,Y}_{L^2(\cA)_{\sa}^m}$.  By definition of $\mathcal{L}f$, we have for all $X' \in L^2(\cA)_{\sa}^m$ that
		\[
		\ip{X',Y}_{L^2(\cA)_{\sa}^m} - f^{\cA}(X') \leq \mathcal{L}f^{\cA}(Y) = \ip{X,Y}_{L^2(\cA)_{\sa}^m} - f^{\cA}(X),
		\]
		hence, $f^{\cA}(X') \geq f^{\cA}(X) + \ip{X'-X,Y}_{L^2(\cA)_{\sa}^m}$, so $Y \in \eth f^{\cA}(X)$.
		
		(2) $\implies$ (1).  Suppose $Y \in \eth f^{\cA}(X)$.  Let $\iota: \cA \to \cB$ be a $\mathrm{W}^*$-inclusion and $E: \cB\to \cA$ the corresponding conditional expectation.  Since $E[\iota(X)] = X$, Lemma \ref{lem:CEgradient1} (1) tells us that
		\[
		\iota(\eth f^{\cA}(X)) = \eth f^{\iota(\cA)}(\iota(X)) = L^2(\cA)_{\sa}^m \cap \eth f^{\cB}(\iota(X)),
		\]
		so in particular, $\iota(Y) \in \eth f^{\cB}(\iota(X))$. Hence, for any $Z \in L^2(\cB)_{\sa}^m$, we have
		\[
		\ip{Z,\iota(Y)}_{L^2(\cB)_{\sa}^m} - f^{\cB}(Z) \leq \ip{\iota(X),\iota(Y)}_{L^2(\cB)_{\sa}^m} - f^{\cB}(\iota(X)) = \ip{X,Y}_{L^2(\cA)_{\sa}^m} - f^{\cA}(X).
		\]
		Since $\iota$, $\cB$, and $Z$ were arbitrary, the supremum defining $\mathcal{L}f^{\cA}(Y)$ is attained at the point $X$, so that $f^{\cA}(X) + \mathcal{L}f^{\cA}(Y) = \ip{X,Y}_{L^2(\cA)_{\sa}^m}$.
		
		Therefore, we have proved that (1) $\iff$ (2).  Because $f$ is $E$-convex, we have $\mathcal{L}(\mathcal{L}f) = f$.  Therefore, (1) $\iff$ (3) follows from (1) $\iff$ (2) by switching the roles of $f$ and $\mathcal{L}f$ and the roles of $X$ and $Y$.
	\end{proof}

	\subsection{A non-commutative Monge-Kantorovich duality} \label{subsec:MKduality}
	
	\begin{definition}
		If $f$ is a tracial $\mathrm{W}^*$-function and $\mu \in \Sigma_m$, then we define $\mu(f) = f^{\cA}(X)$, where $\cA \in \mathbb{W}$ is (isomorphic to) the GNS representation of $\mu$ and $X$ is the canonical generating $m$-tuple.
	\end{definition}

	If $f$ is a tracial $\mathrm{W}^*$-function, for every $\cA$ and every $X \in \cA_{\sa}^m$ with $\lambda_{X} = \mu$, we have $\mu(f) = f^{\cA}(X)$.  This follows by the definition of tracial $\mathrm{W}^*$-function and the fact that $\mathrm{W}^*(X)$ is isomorphic to the GNS representation of $\mu$.
	
	\begin{definition} \label{def:admissible}
		Let us call a pair $(f,g)$ of tracial $\mathrm{W}^*$-functions \emph{admissible} if they take values in $(-\infty,\infty]$ and for every $\cA \in \mathbb{W}$,
		\[
		f^{\cA}(X) + g^{\cA}(Y) \geq \ip{X,Y}_{L^2(\cA)_{\sa}^m} \text{ for all } X, Y \in L^2(\cA)_{\sa}^m.
		\]
	\end{definition}
	
	\begin{proposition} \label{prop:MKduality1}
		Let $\mu$, $\nu \in \Sigma_m$.  The following quantities are equal:
		\begin{enumerate}[(1)]
			\item $C(\mu,\nu)$.
			\item $\inf \{\mu(f) + \nu(g): (f,g) \text{ admissible}\}$.
			\item $\inf \{\mu(f) + \nu(\mathcal{L}f): f \text{ a tracial } \mathrm{W}^*\text{-function not identically } \infty \}$.
			\item $\inf \{\mu(f) + \nu(g): (f,g) \text{ admissible and } E\text{-convex} \}$.
			\item $\inf \{\mu(f) + \nu(\mathcal{L}f): f ~E\text{-convex not identically } \infty \}$.
		\end{enumerate}
		Here all the functions under consideration take values in $(-\infty,\infty]$.
	\end{proposition}
	
	\begin{proof}
		(1) $\leq$ (2) Let $(\cA,X,Y)$ be a coupling of $\mu$ and $\nu$, and let $(f,g)$ be an admissible pair.  Then
		\[
		\ip{X,Y}_{L^2(\cA)_{\sa}^m} \leq f^{\cA}(X) + g^{\cA}(Y) = \mu(f) + \nu(g).
		\]
		Taking the supremum over couplings on the left-hand side and the infimum over admissible pairs $(f,g)$ on the right-hand side, we have (1) $\leq$ (2).
		
		(2) $\leq$ (3).  It is clear from the definition of $\mathcal{L}f$ that $f^{\cA}(X) + \mathcal{L}f^{\cA}(Y) \geq \ip{X,Y}_{L^2(\cA)_{\sa}^m}$. Therefore, $(f,\mathcal{L}f)$ is always an admissible pair, and hence (3) is the infimum over a smaller set than (2).
		
		(3) $\leq$ (1).  Define
		\[
		f^{\cA}(X) = \begin{cases} 0, & \text{if } X \in L^\infty(\cA)_{\sa}^m \text{ and } \lambda_{X} = \mu, \\ +\infty, & \text{otherwise.} \end{cases}
		\]
		Note that $f$ is a tracial $\mathrm{W}^*$-function.  Let $\cA$ be the GNS-representation of $\nu$ with the canonical generators $Y$.  Then $\mathcal{L}f^{\cA}(Y)$ is the supremum of $\ip{\iota(Y),X}_{L^2(\cB)_{\sa}^m}$ where $\iota: \cA \to \cB$ is an inclusion and $X \in L^\infty(\cB)_{\sa}^m$ satisfies $\lambda_{X} = \mu$.  In particular for a non-commutative law $\nu$, letting $(\cA,Y)$ be the GNS realization of $\nu$, we have $\nu(\mathcal{L}f) = \mathcal{L}f^{\cA}(Y) = C(\mu,\nu)$.  Moreover, $\mu(f) = 0$ and hence $C(\mu,\nu) = \mu(f) + \nu(\mathcal{L}f)$.
		
		(2) $\leq$ (4).  This is immediate since (4) is the infimum over a smaller set.
		
		(4) $\leq$ (5).  Suppose that $f$ is $E$-convex.  Then $(f,\mathcal{L}f)$ is admissible as noted above. Also, $\mathcal{L}f$ is always $E$-convex, so (5) is the infimum over a smaller set than (4).
		
		(5) $\leq$ (3).  Let $f$ be a tracial $\mathrm{W}^*$-function.  Then $\mathcal{L}^2 f \leq f$ and $(\mathcal{L}^2 f, \mathcal{L} f)$ is an $E$-convex admissible pair.  Therefore,
		\[
		\mu(f) + \nu(\mathcal{L}f) \geq \mu(\mathcal{L}^2f) + \nu(\mathcal{L}f).
		\]
		Of course, since $\mathcal{L}(\mathcal{L}^2f) = \mathcal{L}^2 (\mathcal{L}f) = \mathcal{L}f$, the term on the right-hand side participates in the infimum (5).  Since the $f$ on the left-hand side was chosen arbitrarily, (3) $\geq$ (5).
	\end{proof}
	
	\begin{proposition} \label{prop:MKduality2}
		Let $(\cA,X,Y)$ be a coupling of $\mu, \nu \in \Sigma_m$.  The following are equivalent:
		\begin{enumerate}[(1)]
			\item The coupling is optimal.
			\item There exists an admissible pair $(f,g)$ such that $\ip{X,Y}_{L^2(\cA)_{\sa}^m} = f^{\cA}(X) + g^{\cA}(Y)$.
			\item There exists a tracial $\mathrm{W}^*$-function $f$ such that $\ip{X,Y}_{L^2(\cA)_{\sa}^m} = f^{\cA}(X) + \mathcal{L}f^{\cA}(Y)$.
			\item There exists an admissible, $E$-convex pair $(f,g)$ such that $\ip{X,Y}_{L^2(\cA)_{\sa}^m} = f^{\cA}(X) + g^{\cA}(Y)$.
			\item There exists an $E$-convex $f$ such that $\ip{X,Y}_{L^2(\cA)_{\sa}^m} = f^{\cA}(X) + \mathcal{L}f^{\cA}(Y)$.
			\item There exists an $E$-convex $\mathrm{W}^*$-function $f$ such that $Y$ is a subgradient vector to $f^{\cA}$ at the point $X$.
		\end{enumerate}
	\end{proposition}
	
	\begin{proof}
		It is immediate from the previous proposition that each of the conditions (2) -- (5) implies (1).
		
		For the converse implication, assume the coupling is optimal.  Let
		\[
		f^{\cB}(Z) = \begin{cases} 0, & \text{if }  Z \in L^\infty(\cB)_{\sa}^m \text{ and } \lambda_{Z} = \mu, \\ +\infty, & \text{otherwise.} \end{cases}
		\]
		As in the proof of the previous proposition, we have $\mu(f) + \nu(\mathcal{L}f) = C(\mu,\nu)$, or equivalently $\ip{X,Y}_{L^2(\cA)_{\sa}^m} = f^{\cA}(X) + \mathcal{L}f^{\cA}(Y)$.  We also have $C(\mu,\nu) = \mu(\mathcal{L}^2f) + \nu(\mathcal{L}f) \leq \mu(f) + \nu(\mathcal{L}f) = C(\mu,\nu)$.  Thus, the pair $(\mathcal{L}^2f, \mathcal{L}f)$ fulfills all of the criteria of (2) -- (5).
		
		The equivalence of (5) and (6) follows from Lemma \ref{lem:Legendresubgradient}.
	\end{proof}
	
	\subsection{A decomposition result for optimal couplings} \label{subsec:decomposition}
	
	As an initial application of duality, we present the following result that expresses an optimal coupling $(X,Y)$ in terms of another optimal coupling $(X',Y')$ with $\cB = \mathrm{W}^*(X') = \mathrm{W}^*(Y')$.
	
	\begin{theorem} \label{thm:decomposition}
		Let $\mu$, $\nu \in \Sigma_{m,R}$, and let $(\cA,X,Y)$ be an optimal coupling of $\mu$ and $\nu$.  Then there exists a $\mathrm{W}^*$-subalgebra $\cB \subseteq \cA$ with the following properties, letting $X' = E_{\cB}[X]$ and $Y' = E_{\cB}[Y]$:
		\begin{enumerate}[(1)]
			\item $\cB = \mathrm{W}^*(X') = \mathrm{W}^*(Y')$.
			\item $X - X'$, $X' - Y'$, and $Y' - Y$ are orthogonal.
			\item $(\cA,X',Y')$ is an optimal coupling of $\lambda_{X'}$ and $\lambda_{Y'}$.  Similarly, $(\cA,X,Y')$ and $(\cA,X',Y)$ are optimal couplings of the respective laws.
		\end{enumerate}
		We may choose $\cB$ to be contained in $\mathrm{W}^*(X)$ (or symmetrically, we may choose it to be contained in $\mathrm{W}^*(Y)$).
		
		Furthermore, there exists some optimal coupling $(\cA,X,Y)$ and a $\cB$ satisfying (1) -- (3) with respect to this coupling such that $\mathrm{W}^*(X,\cB)$ and $\mathrm{W}^*(Y,\cB)$ are freely independent with amalgamation over $\cB$.
	\end{theorem}
	
	\begin{proof}
		Let
		\[
		\mathscr{B} = \{\mathrm{W}^*\text{-subalgebras } \cB \subseteq \cA: \ip{E_{\cB}[X], E_{\cB}[Y]}_{L^2(\cA)_{\sa}^m} = \ip{X,Y}_{L^2(\cA)_{\sa}^m} \},
		\]
		which is partially ordered by inclusion.  We claim that $\mathscr{B}$ has a minimal element, and we will prove this by a transfinite reverse martingale argument.  By Zorn's lemma, it suffices to show that every chain in $\mathscr{B}$ has a lower bound.  Consider a chain $\mathscr{C} \subseteq \mathscr{B}$, and let $\cC = \bigcap_{\cB \in \mathscr{C}} \cB$.  We claim that $\lim_{\cB \in \mathscr{C}} E_{\cB}[X] = E_{\cC}[X]$ in $L^2(\cA)_{\sa}^m$.  Let
		\[
		\delta = \inf_{\cB \in \mathscr{C}} \norm{E_{\cB}[X]}_{L^2(\cA)_{\sa}^m}^2.
		\]
		Given $\epsilon > 0$, there exists $\cB_0 \in \mathscr{C}$ such that $\norm{E_{\cB_0}[X]}_{L^2(\cA)_{\sa}^m}^2 < \delta^2 + \epsilon^2$.  Then for all $\cB \in \mathscr{C}$ with $\cB \subseteq \cB_0$, we have
		\[
		\norm{E_{\cB}[X] - E_{\cB_0}[X]}_{L^2(\cA)_{\sa}^m}^2 = \norm{E_{\cB_0}[X]}_{L^2(\cA)_{\sa}^m}^2 - \norm{E_{\cB}[X]}_{L^2(\cA)_{\sa}^m}^2 \leq \delta^2 + \epsilon^2 - \delta^2 = \epsilon^2.
		\]
		This implies that $Z = \lim_{\cB \in \mathscr{C}} E_{\cB}[X]$ exists in $L^2(\cA)_{\sa}^m$.  Moreover, $\norm{Z_j}_{L^\infty(\cA)} \leq \norm{X_j}_{L^\infty(\cA)}$.  Clearly $Z_j \in \bigcap_{\cB \in \mathscr{C}} \cB = \cC$, and $\ip{Z,W}_{L^2(\cA)_{\sa}^m} = \ip{X,W}_{L^2(\cA)_{\sa}^m}$ for all $W \in L^2(\cC)_{\sa}^m$.  Thus, $\lim_{\cB \in \mathscr{C}} E_{\cB}[X] = E_{\cC}[X]$ in $L^2(\cA)_{\sa}^m$. By the same token $\lim_{\cB \in \mathscr{C}} E_{\cB}[Y] = E_{\cC}[Y]$ in $L^2(\cA)_{\sa}^m$.  Therefore,
		\[
		\ip{E_{\cC}[X], E_{\cC}[Y]}_{L^2(\cA)_{\sa}^m} = \lim_{\cB \in \mathscr{C}} \ip{E_{\cB}[X], E_{\cB}[Y]}_{L^2(\cA)_{\sa}^m} = \ip{X,Y}_{L^2(\cA)_{\sa}^m}.
		\]
		Therefore, $\cC \in \mathscr{B}$ as desired.
		
		So by Zorn's lemma, $\mathscr{B}$ has some minimal element, which we will call $\cB$.  Let $X' = E_{\cB}[X]$ and $Y' = E_{\cB}[Y]$.  Now $\mathrm{W}^*(X') \subseteq \cB$ and we have
		\[
		\ip{X', E_{\mathrm{W}^*(X')}[Y']}_{L^2(\cA)_{\sa}^m} = \ip{X',Y'}_{L^2(\cA)_{\sa}^m}.
		\]
		By minimality of $\cB$, we have $\cB = \mathrm{W}^*(X')$, and similarly, $\cB = \mathrm{W}^*(Y')$.  Hence, (1) holds.
		
		To show that $\cB$ can be chosen inside $\mathrm{W}^*(X)$, note that
		\[
		\ip{E_{\mathrm{W}^*(X)}[X], E_{\mathrm{W}^*(X)}[Y]}_{L^2(\cA)_{\sa}^m} = \ip{X, E_{\mathrm{W}^*(X)}[Y]}_{L^2(\cA)_{\sa}^m} =  \ip{X,Y}_{L^2(\cA)_{\sa}^m}.
		\]
		Thus, we can apply the same argument with $\mathscr{B}$ replaced by the elements of $\mathscr{B}$ contained inside $\mathrm{W}^*(X)$.
		
		To prove (2), since $X - X' = X - E_{\cB}[X]$ is orthogonal to $\cB$, it is immediate that $X - X'$ and $X' - Y'$ are orthogonal.  Similarly, $Y' - Y$ and $X' - Y'$ are orthogonal.  Finally, to show that $X - X'$ and $Y' - Y$ are orthogonal, note that
		\begin{equation} \label{eq:ipexpansion}
		\ip{X - X', Y' - Y}_{L^2(\cA)_{\sa}^m} = \ip{X, Y'}_{L^2(\cA)_{\sa}^m} + \ip{X',Y}_{L^2(\cA)_{\sa}^m} - \ip{X,Y}_{L^2(\cA)_{\sa}^m} - \ip{X',Y'}_{L^2(\cA)_{\sa}^m}.
		\end{equation}
		Observe that
		\[
		\ip{X, Y'}_{L^2(\cA)_{\sa}^m} = \ip{X,E_{\cB}[Y]}_{L^2(\cA)_{\sa}^m} = \ip{E_{\cB}[X],E_{\cB}[Y]}_{L^2(\cA)_{\sa}^m} = \ip{X',Y'}_{L^2(\cA)_{\sa}^m}.
		\]
		Similarly, $\ip{X',Y}_{L^2(\cA)_{\sa}^m} = \ip{X',Y'}_{L^2(\cA)_{\sa}^m}$.  Moreover, $\ip{X,Y}_{L^2(\cA)_{\sa}^m} = \ip{X',Y'}_{L^2(\cA)_{\sa}^m}$ by our choice of $\cB$.  Thus, all the terms in \eqref{eq:ipexpansion} cancel, and $X - X'$ and $Y' - Y$ are orthogonal.
		
		To prove (3), by Proposition \ref{prop:MKduality2}, there exists an admissible pair of $E$-convex $\mathrm{W}^*$-functions $f$ and $g$ such that $f^{\cA}(X) + g^{\cA}(Y) = \ip{X,Y}_{L^2(\cA)_{\sa}^m}$.  By construction of $\cB$ and by $E$-convexity,
		\begin{align*}
			f^{\cA}(X') + g^{\cA}(Y') &\geq \ip{X',Y'}_{L^2(\cA)_{\sa}^m} \\
			&= \ip{X,Y}_{L^2(\cA)_{\sa}^m} \\
			&= f^{\cA}(X) + g^{\cA}(Y) \\
			&\geq f^{\cA}(X') + g^{\cA}(Y').
		\end{align*}
		This implies that $(X',Y')$ is an optimal coupling.  By similar reasoning, since $\ip{X,Y'}_{L^2(\cA)_{\sa}^m} = \ip{X',Y'}_{L^2(\cA)_{\sa}^m}$ and $f^{\cA}(X') \leq f^{\cA}(X)$, we see that $(X',Y)$ is an optimal coupling, and symmetrically $(X,Y')$ is an optimal coupling.
		
		Let $\cA_1$ be a copy of $\mathrm{W}^*(X,\cB)$ and let $\cA_2$ be a copy of $\mathrm{W}^*(Y,\cB)$.  Let $\tilde{\cA} = \cA_1 *_{\cB} \cA_2$ be the amalgamated free product (with its canonical trace $\tilde{\tau}$).  Let $\tilde{X}$, $\tilde{X}'$, $\tilde{Y}$, and $\tilde{Y}'$ be the images of the original variables in $\tilde{\cA}$.  Then using free independence
		\begin{align*}
			\norm*{\tilde{X} - \tilde{Y}}_{L^2(\tilde{\cA})_{\sa}^m}^2 &= \norm*{\tilde{X} - \tilde{X}'}_{L^2(\tilde{\cA})_{\sa}^m}^2 + \norm*{\tilde{X}' - \tilde{Y}'}_{L^2(\tilde{\cA})_{\sa}^m}^2 + \norm*{\tilde{Y}' - \tilde{Y}}_{L^2(\tilde{\cA})_{\sa}^m}^2 \\
			&= \norm*{X - X'}_{L^2(\cA)_{\sa}^m}^2 + \norm*{X' - Y'}_{L^2(\cA)_{\sa}^m}^2 + \norm*{Y' - Y}_{L^2(\cA)_{\sa}^m}^2 \\
			&= \norm{X - Y}_{L^2(\cA)_{\sa}^m}^2.
		\end{align*}
		Therefore, $(\tilde{X}, \tilde{Y})$ is also an optimal coupling of $\mu$ and $\nu$.  The $\mathrm{W}^*$-subalgebra $\cB \subseteq \tilde{\cA}$ also satisfies
		\[
		\ip{E_{\cB}[\tilde{X}], E_{\cB}[\tilde{Y}]}_{\tilde{\tau}} = \ip{\tilde{X}, \tilde{Y}}_{\tilde{\tau}},
		\]
		and satisfies (1).  Thus, the same arguments as above show that $\cB$ in $\tilde{\cA}$ satisfies (2) and (3).
	\end{proof}

	\section{The displacement interpolation} \label{sec:displacement}
	
	If $(\cA,X,Y)$ is an $L^2$-optimal coupling of $\mu$, $\nu \in \Sigma_m$, then one can consider the displacement interpolation $X_t = (1 - t)X + tY$ for $t \in [0,1]$.  As shown in Proposition \ref{prop:geodesic} the corresponding family of laws defines a geodesic in $(\Sigma_m, d_W^{(2)})$.  In this section, we study how the displacement interpolation interacts with non-commutative Monge-Kantorovich duality and use this to prove 
	Theorem \ref{thm:displacementW*}.
	
	Motivated by analogous arguments in classical optimal transport theory, we approach the proof as follows (see \S \ref{subsec:W*displacement} for more detail).  By Proposition \ref{prop:MKduality2}, there exists an $E$-convex function $f$ such that $\ip{X,Y}_{L^2(\cA)_{\sa}^m} = f^{\cA}(X) + \mathcal{L} f^{\cA}(Y)$, or equivalently $Y \in \eth f^{\cA}(X)$.  Letting $q_t$ be the $\mathrm{W}^*$-function $q_t^{\cA}(X) = (1/2t) \norm{X}_{L^2(\cA)_{\sa}^m}^2$, we observe that $X_t \in \eth f_t^{\cA}(X)$ where $f_t = (1 - t) q_1 + t f$.  Hence, $X \in \eth (\mathcal{L} f_t)^{\cA}(X_t)$.  In order to show that $X \in L^2(\mathrm{W}^*(X_t))_{\sa}^m$, we want to understand the regularity properties of $\mathcal{L} f_t$.
	
	It is well-known that for a convex function $f$ on a Hilbert space $H$, the Legendre transform of $f(x) + (t/2) \norm{x}^2$ is given by the inf-convolution $g_t = \inf_{y \in H} [f^*(y) + (1/2t)\norm{x - y}^2]$, where $f^*$ is the Legendre transform of $f$.  Furthermore, $g_t$ has a Lipschitz gradient for every $t > 0$, and it satisfies the Hamilton-Jacobi equation
	\[
	\frac{d}{dt} g_t = - \frac{1}{2} \norm{\nabla g_t}^2.
	\]
	This can be checked by hand, or deduced for instance from \cite[\S 2, Theorem 1]{BdP1981}; also relevant to Hamilton-Jacobi equations on Hilbert space are \cite{BdP1985b,BdP1985a,CrLi1985,CrLi1986a,CrLi1986b,LL1986}.
	
	In this section, we adapt the theory of inf-convolutions to the setting tracial $\mathrm{W}^*$-functions. In \S \ref{subsec:infconvolution}, we define inf-convolutions of $\mathrm{W}^*$-functions and prove their basic properties.  In \S \ref{subsec:infconvolutions2}, we describe how inf-convolutions interact with $E$-convexity and semi-concavity.  In \S \ref{subsec:W*displacement}, we conclude the proof of Theorem \ref{thm:displacementW*}.
	
	We emphasize that the novelty in our work is not in the form of the Hamilton-Jacobi equation but rather in the fact that we study variables from infinite-dimensional non-commutative algebras and want the function to be defined consistently with respect to inclusions of these algebras (that is, to be a tracial $\mathrm{W}^*$-function).  This means for instance that if $f$ and $g$ are tracial $\mathrm{W}^*$-functions and $f \square g$ is their inf-convolution as defined below, then $(f \square g)^{\cA}$ need not agree with the inf-convolution of $f^{\cA}$ and $g^{\cA}$ as functions on the Hilbert space $L^2(\cA)_{\sa}^m$ (Remark \ref{rem:infconvolutiondifferent}); however, they do agree if $f$ and $g$ are $E$-convex (Lemma \ref{lem:infconvolutionconvex}).  Hence, a notion of viscosity solutions compatible with our theory of inf-convolutions will thus have to take into account the inclusions of one tracial $\mathrm{W}^*$-algebra into another.
	
	\subsection{Inf-convolutions} \label{subsec:infconvolution}
	
	We begin with the definition and basic properties of the inf-convolution.
	
	\begin{definition} \label{def:infconvolution}
		Let $f, g$ be two $\mathrm{W}^*$-functions with values in $[-\infty,\infty]$.  We define the \emph{inf-convolution} $f \square g$ by
		\[
		(f \square g)^{\cA}(X) = \inf \left\{ f^{\cB}(\iota(X) - Y) + g^{\cB}(Y) | \iota: \cA \to \cB \text{ embedding, } Y \in L^2(\cB)_{\sa}^m \right\}.
		\]
	\end{definition}
	
	\begin{lemma}
		The object $f \square g$ is a $\mathrm{W}^*$-function.
	\end{lemma}
	
	\begin{proof}
		Let $\iota: \cA \to \cB$ be an inclusion, and we first show that
		\begin{equation} \label{eq:infconvolution1}
			(f \square g)^{\cA}(X) \leq (f \square g)^{\cB}(\iota(X)).
		\end{equation}
		If $\iota': \cB \to \cC$ is another inclusion and $Y \in L^2(\cB)_{\sa}^m$ as in the definition of $(f \square g)^{\cB}$, then of course $\iota' \circ \iota$ is an inclusion and which can be used in the definition of $(f \square g)^{\cA}$.  This shows \eqref{eq:infconvolution1}.
		
		Conversely, suppose that $\iota': \cA \to \cC$ is an inclusion and $Y \in L^2(\cC)_{\sa}^m$ as in the definition of $(f \square g)^{\cA}$.  Then let $\tilde{\cC}$ be the free product of $\cB$ and $\cC$ with amalgamation over the images of $\cA$ in the respective algebras.  Then the image of $Y$ in $\tilde{\cC}$ participates in the infimum defining $(f \square g)^{\cB}(\iota(X))$ and hence $(f \square g)^{\cB}(\iota(X)) \leq (f \square g)^{\cA}(X)$.
	\end{proof}
	
	\begin{lemma}
		The inf-convolution is commutative and associative, that is, if $f$, $g$, $h$ are $\mathrm{W}^*$-functions, then $f \square g = g \square f$ and $(f \square g) \square h = f \square (g \square h)$.
	\end{lemma}
	
	\begin{proof}
		We have
		\[
		(f \square g)^{\cA}(X) = \inf_{\iota: \cA \to \cB} \inf_{Y \in L^2(\cB)_{\sa}^m} [f^{\cB}(\iota(X) - Y) + g^{\cB}(Y)].
		\]
		We substitute $Z = \iota(X) - Y$ and thus obtain
		\[
		\inf_{\iota: \cA \to \cB} \inf_{Y \in L^2(\cB)_{\sa}^m} [f^{\cB}(Z) + g^{\cB}(\iota(X) - Z)] = (g \square f)^{\cA}(X).
		\]
		For associativity,
		\begin{align}
			((f \square g) \square h)^{\cA}(X) &= \inf_{\substack{\iota_1: \cA \to \cB \\ Y \in L^2(\cB)_{\sa}^m}} \left((f \square g)^{\cB}(\iota_1(X) - Y) + h^{\cB}(Y) \right) \nonumber \\
			&= \inf_{\substack{\iota_1: \cA \to \cB \\ Y \in L^2(\cB)_{\sa}^m}} \inf_{\substack{\iota_2: \cB \to \cC \\ Z \in L^2(\cB)_{\sa}^m}} \left( f^{\cC}(\iota_2(\iota_1(X)) - \iota_2(Y) - Z) + g^{\cC}(Z) + h^{\cC}(\iota_2(Y)) \right). \label{eq:associativity1}
		\end{align}
		We claim that is equal to
		\begin{equation}
			\inf_{\substack{\iota: \cA \to \cB \\ Y, Z \in L^2(\cB)_{\sa}^m}} \left( f^{\cC}(\iota(X) - Y - Z) + g^{\cB}(Z) + h^{\cB}(Y) \right), \label{eq:associativity2}
		\end{equation}
		or in other words, in our earlier expression we can without loss of generality impose the condition that $\cC = \cB$ and $\iota_2 = \id$.  The reason is that if we allowed $Z$ to come only from the smaller algebra $\cC$, then the infimum could only increase, hence by shrinking $\cC$ to $\cB$, \eqref{eq:associativity2} $\geq$ \eqref{eq:associativity1}.  On the other hand, if in \eqref{eq:associativity1}, we allowed $Y$ to come from the larger algebra $\cC$ instead of $\cB$, then the infimum could only decrease, and hence by enlarging $\cB$ to $\cC$, we see that \eqref{eq:associativity1} $\leq$ \eqref{eq:associativity2}.  Now the expression \eqref{eq:associativity2} is symmetric in $g$ and $h$, and hence
		\[
		(f \square g) \square h = (f \square h) \square g.
		\]
		This relation, together with commutativity, implies the associativity relation since
		\[
		(f \square g) \square h = (g \square f) \square h = (g \square h) \square f = f \square (g \square h). \qedhere
		\]
	\end{proof}
	
	The relationship between inf-convolution and Legendre transform is exactly what one would expect based on the classical case.
	
	\begin{lemma} \label{lem:Legendreinfconvolution}
		Let $f$ and $g$ be $\mathrm{W}^*$-functions.  Then
		\[
		\mathcal{L}(f \square g) = \mathcal{L}f + \mathcal{L}g.
		\]
	\end{lemma}
	
	\begin{proof}
		Observe that
		\begin{align*}
			\mathcal{L}(f \square g)^{\cA}(X) &= \sup_{\substack{\iota_1: \cA \to \cB \\ Y \in L^2(\cB)_{\sa}^m}} \left( \ip{\iota_1(X),Y}_{L^2(\cB)_{\sa}^m} - (f \square g)^{\cB}(Y) \right) \\
			&= \sup_{\substack{\iota_1: \cA \to \cB \\ Y \in L^2(\cB)_{\sa}^m}} \left( \ip{\iota_1(X),Y}_{L^2(\cB)_{\sa}^m} - \inf_{\substack{\iota_2: \cB \to \cC \\ Z \in L^2(\cC)_{\sa}^m} } \left( f^{\cC}(\iota_2(Y) - Z) + g^{\cC}(Z) \right) \right),
		\end{align*}
		where we take the supremum over $\cB$ and $\cC \in \mathbb{W}$ and inclusions $\iota_1: \cA \to \cB$ and $\iota_2: \cB \to \cC$ and $Y \in L^2(\cB)_{\sa}^m$ and $Z \in L^2(\cC)_{\sa}^m$.  This can be rewritten as
		\[
		\sup_{\substack{\iota_1: \cA \to \cB \\ Y \in L^2(\cB)_{\sa}^m}} \sup_{\substack{\iota_2: \cB \to \cC \\ Z \in L^2(\cC)_{\sa}^m} } \left( \ip{\iota_2(\iota_1(X)),\iota_2(Y)}_{L^2(\cC)_{\sa}^m} - f^{\cC}(\iota_2(Y) - Z) - g^{\cC}(Z) \right).
		\]
		We can assume without loss generality that $\cB = \cC$ and $\iota_2 = \id$.  Indeed, allowing $Y$ to range over the larger space $L^2(\cC)_{\sa}^m$ rather than $L^2(\cB)_{\sa}^m$ would only increase the supremum, but on the other hand, restricting $Z$ to the smaller space $L^2(\cB)_{\sa}^m$ instead of $L^2(\cC)_{\sa}^m$ would only decrease the supremum.  Thus, taking $\cB = \cC$ and renaming $\iota_1$ to $\iota$, we obtain
		\[
		\sup_{\iota: \cA \to \cB} \sup_{Y,Z \in L^2(\cB)_{\sa}^m} \left( \ip{\iota(X),Y}_{L^2(\cB)_{\sa}^m} - f^{\cB}(Y - Z) - g^{\cB}(Z) \right).
		\]
		Substituting $Z' = Y - Z$, we have
		\begin{multline} \label{eq:sumofLegendre?}
			\sup_{\iota: \cA \to \cB} \sup_{Z,Z' \in L^2(\cB)_{\sa}^m} \left( \ip{\iota(X),Z+Z'}_{L^2(\cB)_{\sa}^m} - f^{\cB}(Z') - g^{\cB}(Z) \right) \\
			= \sup_{\iota: \cA \to \cB} \sup_{Z,Z' \in L^2(\cB)_{\sa}^m} \left( \ip{\iota(X),Z'}_{L^2(\cB)_{\sa}^m} - f^{\cB}(Z') + \ip{\iota(X),Z}_{L^2(\cB)_{\sa}^m} - g^{\cB}(Z) \right).
		\end{multline}
		We want to show that this is equal to
		\begin{multline} \label{eq:sumofLegendre}
			\mathcal{L}f^{\cA}(X) + \mathcal{L}g^{\cA}(X) = 
			\sup_{\iota_1: \cA \to \cB_1} \sup_{Z' \in L^2(\cB_1)_{\sa}^m} \left( \ip{\iota(X),Z'}_{L^2(\cB_1)_{\sa}^m} - f^{\cB_1}(Z') \right) \\
			+ \sup_{\iota_2: \cA \to \cB_2} \sup_{Z \in L^2(\cB_2)_{\sa}^m} \left( \ip{\iota(X),Z}_{L^2(\cB_2)_{\sa}^m} - g^{\cB_2}(Z) \right).
		\end{multline}
		The only difference between the two expressions is that the latter allows $\iota_1: \cA \to \cB_1$ and $\iota_2: \cA \to \cB_2$ to be different, but the former takes them to be the same, and thus a priori \eqref{eq:sumofLegendre?} $\leq$ \eqref{eq:sumofLegendre}.  However, in \eqref{eq:sumofLegendre}, for any given $\cB_1$, $\cB_2$, $\iota_1$ and $\iota_2$, let $\cB$ be the free product of $\cB_1$ and $\cB_2$ with amalgamation over the subalgebras $\iota_1(\cA)$ in the first factor and $\iota_2(\cA)$ in the second factor.  Allowing $Z'$ and $Z$ to range over $L^2(\cB)_{\sa}^m$ rather than $L^2(\cB_1)_{\sa}^m$ and $L^2(\cB_2)_{\sa}^m$ respectively only increases the suprema over $Z$ and $Z'$, and hence \eqref{eq:sumofLegendre} remains unchanged when we restrict to the case $\iota_1 = \iota_2$, so it equals \eqref{eq:sumofLegendre?}.
	\end{proof}
	
	\begin{remark} \label{rem:infconvolutiondifferent}
		Suppose $f$ and $g$ are tracial $\mathrm{W}^*$-functions.  Let $f^{\cA} \square g^{\cA}$ denote the classical inf-convolution of $f^{\cA}$ and $g^{\cA}$ as functions on the Hilbert space $L^2(\cA)_{\sa}^m$.  Then $(f \square g)^{\cA} \leq f^{\cA} \square g^{\cA}$.  However, the following example shows that two functions do not necessarily agree.  Take $m = 2$, and $f^{\cA}(X_1,X_2) = (1/2) \norm{(X_1,X_2)}_{L^2(\cA)^2}^2$ and $g^{\cA}(X_1,X_2) = \tau_{\cA}([X_1,X_2]^2)$.  The formula for $g$ is to be understood in the sense of affiliated operators (see \S \ref{sec:Lp}); since $i[X_1,X_2]$ is a self-adjoint affiliated operator, $-[X_1,X_2]^2$ is positive and hence $\tau_{\cA}([X_1,X_2]^2)$ is well-defined in $[-\infty,0]$; see Theorem \ref{thm:affiliated} (4).  Then $g^{\C} = 0$ because $\C$ is commutative, and hence also $f^{\C} \square g^{\C} = 0$.  On the other hand, let $\iota: \C \to M_2(\C)$ be the canonical inclusion, and let
		\[
		Y_1 = \begin{pmatrix} 0 & 1 \\ 1 & 0 \end{pmatrix}, \qquad Y_2 = \begin{pmatrix} 0 & i \\ -i & 0 \end{pmatrix}, \qquad [Y_1,Y_2] = \begin{pmatrix} -2i & 0 \\ 0 & 2i \end{pmatrix}.
		\]
		Then for $x_1, x_2, t \in \R$,
		\begin{align*}
			(f \square g)^{\C}(x_1,x_2) &\leq \frac{1}{2} \norm{\iota(x_1) - tY_1}_{L^2(M_2(\C))}^2 + \frac{1}{2} \norm{\iota(x_2) - tY_2}_{L^2(M_2(\C))}^2 + t^4 \tau_{M_2(\C)}([Y_1,Y_2]^2) \\
			&= \frac{1}{2} \norm{\iota(x_1) - tY_1}_{L^2(M_2(\C))}^2 + \frac{1}{2} \norm{\iota(x_2) - tY_2}_{L^2(M_2(\C))}^2 - 4t^4.
		\end{align*}
		The first two terms are quadratic in $t$, and thus, taking the infimum over $t \in \R$, we see that $(f \square g)^{\C} = -\infty < f^{\C} \square g^{\C}$.
	\end{remark}
	
	\subsection{Inf-convolutions and regularity of $E$-convex functions} \label{subsec:infconvolutions2}
	
	\begin{lemma} \label{lem:infconvolutionconvex}
		If $f$ and $g$ are $E$-convex tracial $\mathrm{W}^*$-functions with $f < \infty$, then $f \square g$ is $E$-convex.  Moreover, for any $E$-convex $f$ and $g$, we have
		\begin{equation} \label{eq:convexinfconvolution}
			(f \square g)^{\cA}(X) = \inf_{Y \in L^2(\cA)_{\sa}^m} \left( f^{\cA}(X - Y) + g^{\cA}(Y) \right).
		\end{equation}
	\end{lemma}
	
	\begin{proof}
		We prove the second claim first.  Clearly,
		\[
		(f \square g)^{\cA}(X) \leq \inf_{Y \in L^2(\cA)_{\sa}^m} \left( f^{\cA}(X - Y) + g^{\cA}(Y) \right).
		\]
		For the opposite inequality, suppose that $\iota: \cA \to \cB$ is an embedding and $Y \in L^2(\cB)_{\sa}^m$.  Let $E: \cB \to \cA$ be the conditional expectation.  Then by $E$-convexity of $f$ and $g$,
		\[
		f^{\cA}(X - E[Y]) + g^{\cA}(E[Y]) \leq f^{\cB}(\iota(X) - Y) + g^{\cB}(Y),
		\]
		and hence
		\[
		\inf_{Y \in L^2(\cA)_{\sa}^m} \left( f^{\cA}(X - Y) + g^{\cA}(Y) \right) \leq (f \square g)^{\cA}(X).
		\]

		Now let us show that $f \square g$ is $E$-convex when $f < \infty$.  If $g$ is identically $\infty$, then $f \square g$ is identically $\infty$, so there is nothing to prove.  Suppose $g^{\cB}(Y)$ is finite for some $\cB$ and $Y \in L^2(\cB)_{\sa}^m$.  Then $(f \square g)^{\cA}(X) < \infty$ everywhere because, letting $\cC$ be the free product of $\cA$ and $\cB$ and letting $\iota_1: \cA \to \cC$ and $\iota_2: \cB \to \cC$ be the corresponding inclusions,
		\[
		(f \square g)^{\cA}(X) \leq f^{\cC}(\iota_1(X) - \iota_2(Y)) + g^{\cC}(\iota_2(Y)) < \infty.
		\]
		To prove convexity of $(f \square g)^{\cA}$, let $X_0$, $X_1 \in L^2(\cA)_{\sa}^m$, and let $X_t = (1 - t)X_0 + tX_1$ for $t \in (0,1)$.  If $Y_0$, $Y_1 \in L^2(\cA)_{\sa}^m$ and if $Y_t = (1 - t)Y_0 + tY_1$, then
		\begin{align*}
			(f \square g)^{\cA}(X_t) &\leq f^{\cA}(X_t - Y_t) + g^{\cA}(Y_t) \\
			&\leq (1 - t) f^{\cA}(X_0 - Y_0) + t f^{\cA}(X_1 - X_1) + (1 - t) g^{\cA}(Y_0) + t g^{\cA}(Y_1).
		\end{align*}
		Since $Y_0$ and $Y_1$ were arbitrary, we can take the infimum over $Y_0$ and $Y_1$ and apply \eqref{eq:convexinfconvolution} to conclude that
		\[
		(f \square g)^{\cA}(X_t) \leq (1 - t) (f \square g)^{\cA}(X_0) + t (f \square g)^{\cA}(X_1).
		\]
		This shows that $(f \square g)^{\cA}$ is convex.  Furthermore, since $f \square g < \infty$, this relation implies that if $(f \square g)^{\cA}$ is $-\infty$ at one point in $L^2(\cA)_{\sa}^m$, then it is $-\infty$ everywhere.  Moreover, if $(f \square g)^{\cB}$ is $-\infty$, then so $(f \square g)^{\cA}$, as we can see by considering the free product of $\cA$ and $\cB$.
		
		It is automatic from these facts that $(f \square g)^{\cA}$ is lower semi-continuous, since convexity automatically implies lower semi-continuity at points where $(f \square g)^{\cA} < \infty$.
		
		Finally, we must show the monotonicity of $(f \square g)$ under conditional expectation.  Let $\iota: \cA \to \cB$ be an embedding and let $E: \cB \to \cA$ be the corresponding conditional expectation.  If $X, Y \in L^2(\cA)_{\sa}^m$, then
		\[
		(f \square g)^{\cA}(E[X]) \leq f^{\cA}(E[X] - E[Y]) + g^{\cA}(E[Y]) \leq f^{\cB}(X - Y) + g^{\cB}(Y).
		\]
		Since $Y$ on right-hand side was arbitrary, we conclude by \eqref{eq:convexinfconvolution} that $(f \square g)^{\cA}(E[X]) \leq (f \square g)^{\cB}(X)$ as desired.
	\end{proof}
	
	\begin{observation}
		For $t \in (0,\infty)$, let $q_t^{\cA}(X) = (1/2t) \norm{X}_{L^2(\cA)_{\sa}^m}^2$.  For $s, t \in (0,\infty)$, because $q_s$ and $q_t$ are $E$-convex and take finite values, we have
		\[
		q_s \square q_t = \mathcal{L}^2(q_s \square q_t) = \mathcal{L}(\mathcal{L} q_s + \mathcal{L} q_t) = \mathcal{L}(q_{1/s} + q_{1/t}) = \mathcal{L}(q_{1/(s+t)}) = q_{s+t}.
		\]
		Then by associativity of inf-convolution, for any tracial $\mathrm{W}^*$-function $f$, we have
		\[
		q_s \square (q_t \square f) = (q_s \square q_t) \square f = q_{s+t} \square f.
		\]
		Thus, $(q_t \square (\cdot))_{t > 0}$ defines a semigroup acting on tracial $\mathrm{W}^*$-functions.  This is the tracial $\mathrm{W}^*$-analog of the Hopf-Lax semigroup.
	\end{observation}
	
	\begin{definition}
		If $f$ is a tracial $\mathrm{W}^*$-function, we say that $f$ is \emph{convex} if $f^{\cA}$ is convex for every $\cA \in \mathbb{W}$.  We say that $f$ is \emph{semi-concave} if $q_t - f$ is convex for some $t > 0$.
	\end{definition}
	
	\begin{lemma} \label{lem:infconvolutionconcave}
		Suppose $f$ and $g$ are tracial $\mathrm{W}^*$-functions and $q_t - f$ is convex.  Then $q_t - f \square g$ is convex.
	\end{lemma}
	
	\begin{proof}
		Note that
		\begin{align*}
			(q_t - f \square g)^{\cA}(X) &= \sup_{\substack{\iota: \cA \to \cB \\ Y \in L^2(\cB)_{\sa}^m}} \left( q_t^{\cB}(\iota(X)) - f^{\cB}(\iota(X) - Y) - g^{\cB}(Y) \right) \\
			&= \sup_{\substack{\iota: \cA \to \cB \\ Y \in L^2(\cB)_{\sa}^m}} \left( q_t^{\cB}(\iota(X) - Y) - f^{\cB}(\iota(X) - Y) + \frac{1}{t} \ip{\iota(X),Y}_{L^2(\cB)_{\sa}^m} - q_t^{\cB}(Y) - g^{\cB}(Y) \right).
		\end{align*}
		The right-hand side is the supremum of a family convex functions of $X$ and therefore is convex.
	\end{proof}
	
	As a consequence of Lemmas \ref{lem:infconvolutionconvex} and \ref{lem:infconvolutionconcave}, if $f$ is $E$-convex, then $q_t \square f$ is an $E$-convex and semi-concave function.  The next results give a characterization of such functions as well as some of their regularity properties.  These results are quite close to the standard results about convex functions on a Hilbert space, so we do not claim any originality, but nonetheless we include the proofs for the sake of completeness.
	
	\begin{proposition} \label{prop:convexityconversion}
		Let $f$ be an $E$-convex $\mathrm{W}^*$-function that is not identically $\infty$ or $-\infty$.  Then the following are equivalent:
		\begin{enumerate}[(1)]
			\item $f = q_t \square g$ for some $E$-convex function $g$.
			\item $q_t - f$ is convex.
			\item $q_t - f$ is $E$-convex.
			\item $\mathcal{L}f - q_{1/t}$ is convex and lower semi-continuous.
			\item $\mathcal{L}f - q_{1/t}$ is $E$-convex.
		\end{enumerate}
		Moreover, in this case, $f < \infty$ everywhere.
	\end{proposition}
	
	\begin{proof}
		(1) $\implies$ (2) follow from Lemma \ref{lem:infconvolutionconcave}.
		
		(2) $\implies$ (3). Because $q_t - f$ takes finite values everywhere, by Lemma \ref{lem:Econvex}, it suffices to show that for every $X \in L^2(\cA)$, there exists a some $Z \in \eth (q_t-f)^{\cA}(X) \cap L^2(\mathrm{W}^*(X))_{\sa}^m$.  Because $q_t - f$ is convex, it has a subgradient vector $Z$ at $X$, so that
		\[
		q_t^{\cA}(X') - f^{\cA}(X') - q_t^{\cA}(X) + f^{\cA}(X) \geq \ip{X'-X,Z}_{L^2(\cA)_{\sa}^m},
		\] 
		which implies that
		\begin{equation} \label{eq:supergradient}
			f^{\cA}(X') - f^{\cA}(X) \leq \ip{X-X',Z} + \frac{1}{2t} (\norm{X'}_{L^2(\cA)_{\sa}^m}^2 - \norm{X}_{L^2(\cA)_{\sa}^m}^2)
			= \ip{X'-X,Z + (1/t)X}_{L^2(\cA)_{\sa}^m} + \frac{1}{2t} \norm{X' - X}_{L^2(\cA)_{\sa}^m}^2.
		\end{equation}
		Because $f$ is $E$-convex, there exists some $Y \in \eth f^{\cA}(X) \cap L^2(\mathrm{W}^*(X))_{\sa}^m$.  Of course,
		\begin{equation} \label{eq:subgradient}
			f^{\cA}(X') - f^{\cA}(X) \geq \ip{X'-X,Y}_{L^2(\cA)_{\sa}^m}.
		\end{equation}
		This implies that
		\[
		\ip{X'-X,Z + (1/t)X - Y}_{L^2(\cA)_{\sa}^m} \geq - \frac{1}{2t} \norm{X' - X}_{L^2(\cA)_{\sa}^m}^2
		\]
		for all $X'$.  Now take $X' = -tZ + tY$ and obtain
		\[
		-t \norm{Z + (1/t)X - Y}_{L^2(\cA)_{\sa}^m}^2 = \ip{X'-X,Z + (1/t)X - Y}_{L^2(\cA)_{\sa}^m} \geq - \frac{1}{2t} \norm{X' - X}_{L^2(\cA)_{\sa}^m}^2 \geq -\frac{t}{2} \norm{Z + (1/t)X - Y}_{L^2(\cA)_{\sa}^m}^2,
		\]
		which implies that $Z + (1/t)X - Y = 0$, hence $Z = Y - (1/t)X \in L^2(\mathrm{W}^*(X))_{\sa}^m$.
		
		(3) $\implies$ (4).  Note that
		\begin{align*}
			\mathcal{L}f^{\cA}(X) - q_{1/t}^{\cA}(X) &= \sup_{\substack{\iota:\cA \to \cB \\ Y \in L^2(\cB)_{\sa}^m}} \left( \ip{\iota(X),Y}_{L^2(\cB)_{\sa}^m} - \frac{t}{2} \norm{\iota(X)}_{L^2(\cB)_{\sa}^m}^2 - f^{\cB}(Y) \right) \\
			&= \sup_{\substack{\iota:\cA \to \cB \\ Z \in L^2(\cB)_{\sa}^m}} \left( \ip{\iota(X),Z+t \iota(X)}_{L^2(\cB)_{\sa}^m} - \frac{t}{2} \norm{\iota(X)}_{L^2(\cB)_{\sa}^m}^2 - f^{\cB}(Z + t\iota(X)) \right) \\
			&= \sup_{\substack{\iota:\cA \to \cB \\ Z \in L^2(\cB)_{\sa}^m}} \left( -\frac{1}{2t} \norm{Z}_{L^2(\cB)_{\sa}^m}^2 + \frac{1}{2t} \norm{Z + t \iota(X)}_{L^2(\cB)_{\sa}^m}^2 - f^{\cB}(Z + t\iota(X)) \right).
		\end{align*}
		Because $q_t - f$ is convex and lower semi-continuous, the right-hand side is the supremum of convex lower semi-continuous functions of $X$, and therefore is convex and lower semi-continuous.
		
		(4) $\implies$ (5).  Let $h = \mathcal{L}f$.  Since $f$ is not identically $-\infty$ or $\infty$, the same is true of $h$.  We assumed in (3) that $h - q_{1/t}$ is convex and lower semi-continuous.  Moreover, if $E: \cB \to \cA$ is a conditional expectation, then $h^{\cB}(X) < \infty$ implies $(h - q_{1/t})^{\cB}(X) < \infty$ implies $(h - q_{1/t})^{\cA}(E[X]) < \infty$ implies $h^{\cA}(E[X]) < \infty$.  Thus, it remains to show that $h^{\cA}(E[X]) \leq h^{\cB}(X)$ whenever $h^{\cA}(E[X])$ is finite.  As in Lemma \ref{lem:Econvex}, it suffices to show that for every $\cA$ and $X \in L^2(\cA)_{\sa}^m$ with $h^{\cA}(X) < \infty$, there exists some subgradient vector $Y \in L^2(\mathrm{W}^*(X),\tau|_{\mathrm{W}^*(X)})_{\sa}^m$.  By $E$-convexity of $h$, there exists some $Z \in \eth h^{\cA}(X) \cap L^2(\mathrm{W}^*(X))_{\sa}^m$.  Then we claim that $Z - tX \in \eth (h-q_{1/t})^{\cA}(X)$.  To prove this, observe that by convexity of $h - q_{1/t}$, for $s \in (0,1)$, and $X' \in L^2(\cA)_{\sa}^m$,
		\begin{align*}
			s (h - q_{1/t})^{\cA}(X') &\geq (h - q_{1/t})^{\cA}((1-s)X + sX') - (1 - s)(h - q_{1/t})^{\cA}(X) \\
			&\geq h^{\cA}(X) + \ip{(1 - s)X + sX'-X, Z}_{L^2(\cA)_{\sa}^m} - q_{1/t}^{\cA}((1-s)X + sX') - h^{\cA}(X) \\
			& \qquad + q_{1/t}^{\cA}(X) + s (h - q_{1/t})^{\cA}(X) \\
			&= s (h - q_{1/t})^{\cA}(X) + s \ip{X' - X,Z}_{L^2(\cA)_{\sa}^m} + q_{1/t}^{\cA}(X) - q_{1/t}^{\cA}((1-s)X + sX') \\
			&= s (h - q_{1/t})^{\cA}(X) + s \ip{X' - X,Z}_{L^2(\cA)_{\sa}^m} + \frac{t}{2} \norm{X}_{L^2(\cA)_{\sa}^m}^2 - \frac{t}{2} \norm{X + s(X' - X)}_{L^2(\cA)_{\sa}^m}^2 \\
			&= s (h - q_{1/t})^{\cA}(X) + s \ip{X' - X,Z - tX}_{L^2(\cA)_{\sa}^m} - \frac{ts^2}{2} \norm{X' - X}_{L^2(\cA)_{\sa}^m}^2.
		\end{align*}
		Dividing by $s$ and sending $s \to 0^+$, we obtain
		\[
		(h - q_{1/t})^{\cA}(X') \geq (h - q_{1/t})^{\cA}(X) + \ip{X' - X, Z - tX}_{L^2(\cA)_{\sa}^m}.
		\]
		Hence, $Z - tX \in \eth (h - q_{1/t})^{\cA}(X)$.  Since $Z - tX \in L^2(\mathrm{W}^*(X))_{\sa}^m$, the proof is complete.
		
		(5) $\implies$ (1).  Since $\mathcal{L}f - q_{1/t}$ is $E$-convex, we have $\mathcal{L}f - q_{1/t} = \mathcal{L} g$ for some $E$-convex function $g$ by Proposition \ref{prop:Legendre}.  Thus, since $g$ and $q_{1/t}$ are both $E$-convex, we have
		\[
		f = \mathcal{L}^2 f = \mathcal{L}(\mathcal{L}g + q_{1/t}) = \mathcal{L} \mathcal{L}(g \square q_t) = g \square q_t,
		\]
		where the last line follows because $g \square q_t$ is $E$-convex by Lemma \ref{lem:infconvolutionconvex}.
		
		Finally, (1) implies that $f < \infty$ everywhere.  Indeed, if $X \in L^2(\cA)_{\sa}^m$, and if $Y$ is some point where $g^{\cB}(Y) < \infty$, then let $\cC$ be the free product of $\cA$ and $\cB$ and let $\iota_1: \cA \to \cC$ and $\iota_2: \cB \to \cC$ be the corresponding inclusions.  Then $(g \square q_t)^{\cA}(X) \leq \frac{1}{2t} \norm{\iota_1(X) - \iota_2(Y)}_{L^2(\cC)_{\sa}^m}^2 + g^{\cC}(\iota_2(Y)) < \infty$.
	\end{proof}
	
	\begin{proposition} \label{prop:convexsemiconcave}
		Let $f$ be an $E$-convex $\mathrm{W}^*$-function taking values in $\R$.  Then the following are equivalent:
		\begin{enumerate}[(1)]
			\item $q_t - f$ is convex.
			\item If $\cA \in \mathbb{W}$ and $Y \in \eth f^{\cA}(X)$ and $Y' \in \eth f^{\cA}(X')$, then $\norm{Y - Y'}_{L^2(\cA)_{\sa}^m} \leq (1/t) \norm{X - X'}_{L^2(\cA)_{\sa}^m}$.
			\item If $\cA \in \mathbb{W}$, then $\eth f^{\cA}(X)$ consists of a single point $\nabla f^{\cA}(X) \in L^2(\mathrm{W}^*(X))_{\sa}^m$, and $\nabla f^{\cA}$ defines a $(1/t)$-Lipschitz function $L^2(\cA)_{\sa}^m \to L^2(\cA)_{\sa}^m$.
			\item For each $\cA$ and $X \in L^2(\cA)_{\sa}^m$ and  $Y \in \eth f^{\cA}(X)$, we have
			\begin{equation} \label{eq:convexsemiconcave}
				\ip{X'-X,Y}_{L^2(\cA)_{\sa}^m} \leq f^{\cA}(X') - f^{\cA}(X) \leq \ip{X'-X, Y}_{L^2(\cA)_{\sa}^m} + \frac{1}{2t} \norm{X' - X}_{L^2(\cA)_{\sa}^m}^2
			\end{equation}
			for all $X' \in L^2(\cA)_{\sa}^m$.
		\end{enumerate}
	\end{proposition}
	
	\begin{proof}
		(1) $\implies$ (2).  By the previous Proposition \ref{prop:convexityconversion}, $\mathcal{L}f - q_{1/t}$ is $E$-convex.  Let $Y \in \eth f^{\cA}(X)$ and $Y' \in \eth f^{\cA}(X')$.  Then by Lemma \ref{lem:Legendresubgradient}, we have $X \in \eth \mathcal{L}f^{\cA}(Y)$ and $X' \in \eth \mathcal{L}f^{\cA}(Y')$.  By the same argument as (4) $\implies$ (5) in the proof of Proposition \ref{prop:convexityconversion}, we have $Z := X - tY \in \eth(\mathcal{L}f - q_{1/t})(Y)$ and $Z' := X' - tY' \in \eth(\mathcal{L}f - q_{1/t})(Y')$.  It follows that
		\[
		\ip{Z', Y - Y'}_{L^2(\cA)_{\sa}^m} \leq \mathcal{L}f^{\cA}(Y) - \mathcal{L}f^{\cA}(Y') \leq \ip{Z, Y - Y'}_{L^2(\cA)_{\sa}^m},
		\]
		hence
		\begin{align*}
			0 &\leq \ip{Z' - Z, Y' - Y}_{L^2(\cA)_{\sa}^m} \\
			&= \ip{X' - X - t(Y' - Y), Y' - Y}_{L^2(\cA)_{\sa}^m} \\
			&= \ip{X' - X,Y' - Y}_{L^2(\cA)_{\sa}^m} - t \norm{Y' - Y}_{L^2(\cA)_{\sa}^m}^2 \\
			&\leq \norm{X' - X}_{L^2(\cA)_{\sa}^m} \norm{Y' - Y}_{L^2(\cA)_{\sa}^m} - t \norm{Y' - Y}_{L^2(\cA)_{\sa}^m}^2.
		\end{align*}
		Therefore, $\norm{Y' - Y}_{L^2(\cA)_{\sa}^m} \leq (1/t) \norm{X - X'}_{L^2(\cA)_{\sa}^m}$ as desired.
		
		(2) $\implies$ (3).  By taking $X = X'$ in (2), we see that there is a unique $Y \in \eth f^{\cA}(X)$ and that $X \mapsto Y$ is a $(1/t)$-Lipschitz function.  By Lemma \ref{lem:Econvex}, we know that $\eth f^{\cA}(X)$ contains some point in $L^2(\mathrm{W}^*(X))_{\sa}^m$, and this point must equal $Y$.
		
		(3) $\implies$ (4).  Let $\cA$ and $X$ be given.  By our assumption of (3), there is a unique point $Y = \nabla f^{\cA}(X)$ in $\eth f^{\cA}(X)$.  Let $X' \in L^2(\cA)_{\sa}^m$.  The lower bound $\ip{X' - X,Y}_{L^2(\cA)_{\sa}^m} \leq f^{\cA}(X') - f^{\cA}(X)$ follows immediately from convexity.  For the upper bound, let $X_t = (1 - t)X' + tX$ and let $Y_t = \nabla f^{\cA}(X_t)$.
		
		For $n \in \N$, observe that
		\begin{align*}
			f^{\cA}(X') - f^{\cA}(X) &= \sum_{j=1}^n \left( f^{\cA}(X_{j/n}) - f^{\cA}(X_{(j-1)/n}) \right) \\
			&\leq \sum_{j=1}^n \ip{X_{j/n} - X_{(j-1)/n}, Y_{j/n}}_{L^2(\cA)_{\sa}^m} \\
			&\leq \sum_{j=1}^n \ip{X_{j/n} - X_{(j-1)/n}, Y}_{L^2(\cA)_{\sa}^m} + \sum_{j=1}^n \norm{X_{j/n} - X_{(j-1)/n}}_{L^2(\cA)_{\sa}^m} \norm{Y_{j/n} - Y}_{L^2(\cA)_{\sa}^m} \\
			&\leq \ip{X' - X,Y}_{L^2(\cA)_{\sa}^m} + \sum_{j=1}^n \frac{1}{n} \norm{X' - X}_{L^2(\cA)_{\sa}^m} \frac{1}{t} \norm{X_{j/n} - X}_{L^2(\cA)_{\sa}^m} \\
			&\leq \ip{X' - X,Y}_{L^2(\cA)_{\sa}^m} + \frac{1}{t} \norm{X' - X}_{L^2(\cA)_{\sa}^m}^2 \sum_{j=1}^n \frac{j}{n^2} \\
			&= \ip{X' - X,Y}_{L^2(\cA)_{\sa}^m} + \frac{1}{t} \norm{X' - X}_{L^2(\cA)_{\sa}^m}^2 \frac{n(n+1)}{2n^2}.
		\end{align*}
		Taking $n \to \infty$ shows that $f^{\cA}(X') - f^{\cA}(X) \leq \ip{X' - X,Y}_{L^2(\cA)_{\sa}^m} + (1/2t) \norm{X' - X}_{L^2(\cA)_{\sa}^m}^2$ as desired.
		
		(4) $\implies$ (1).  Let $\cA \in \mathbb{W}$.  We show that $(q_t - f)^{\cA}$ is convex by exhibiting a subgradient vector for every $X \in L^2(\cA)_{\sa}^m$.  Let $Y \in \eth f^{\cA}(X)$ and let $X' \in L^2(\cA)_{\sa}^m$.  By (4),
		\begin{align*}
			(q_t - f)^{\cA}(X') - (q_t - f)^{\cA}(X) &\geq \frac{1}{2t} \norm{X'}_{L^2(\cA)_{\sa}^m}^2 - \frac{1}{2t} \norm{X}_{L^2(\cA)_{\sa}^m}^2 - \ip{X' - X,Y}_{L^2(\cA)_{\sa}^m}^2 - \frac{1}{2t} \norm{X' - X}_{L^2(\cA)_{\sa}^m}^2 \\
			&= \ip{X' - X, -Y + (1/t)X}_{L^2(\cA)_{\sa}^m}.
		\end{align*}
		Hence, $-Y + (1/t)X$ is a subgradient vector for $q_t - f$ at $X$ as desired.
	\end{proof}
	
	\subsection{Main results on the displacement interpolation} \label{subsec:W*displacement}
	
	We start out by proving Theorem \ref{thm:displacementW*} which states that if $(\cA,X,Y)$ is an $L^2$ optimal coupling and $X_t = (1 - t)X + tY$, then $\mathrm{W}^*(X_t) = \mathrm{W}^*(X,Y)$ for all $t \in (0,1)$.
	
	\begin{proof}[{Proof of Theorem \ref{thm:displacementW*}}]
		By Proposition \ref{prop:MKduality2}, there exists an $E$-convex function $f$ such that $Y \in \partial f^{\cA}(X)$.  Let $f_t = (1 - t) q_1 + tf$, where $q_1^{\cA}(X) = (1/2) \norm{X}_{L^2(\cA)_{\sa}^m}^2$.  Since $(1 - t)X$ is a subgradient to $(1-t)q_1$ at $X$ and $tY$ is a subgradient to $f^{\cA}$ at $X$, we have $X_t \in \eth f_t^{\cA}(X)$.  By Lemma \ref{lem:Legendresubgradient}, we have $X \in \eth \mathcal{L} f_t^{\cA}(X_t)$.  Since $f_t - q_{1/(1-t)} = f_t - (1 - t)q_1 = tf$ is $E$-convex, $q_{1-t} - \mathcal{L}f_t$ is $E$-convex by Proposition \ref{prop:convexityconversion}.  Hence, by Proposition \ref{prop:convexsemiconcave}, $\eth \mathcal{L} f_t^{\cA}(X_t)$ consists of a single point which is in $L^2(\mathrm{W}^*(X_t))_{\sa}^m$.  But we already know that $X \in \eth \mathcal{L} f_t^{\cA}(X_t)$, and therefore $X \in L^2(\mathrm{W}^*(X_t))_{\sa}^m$.
		
		A symmetrical argument shows that $Y \in L^2(\mathrm{W}^*(X_t))_{\sa}^m$.  Therefore, $\mathrm{W}^*(X,Y) \subseteq \mathrm{W}^*(X_t)$.  The reverse inclusion $\mathrm{W}^*(X_t) \subseteq \mathrm{W}^*(X,Y)$ is obvious since $X_t = (1 - t)X + tY$.
	\end{proof}
	
	It follows from the triangle inequality that $(\cA,X_s,X_t)$ is an optimal coupling of the laws of $X_s$ and $X_t$ (see Proposition \ref{prop:geodesic}).  Another way to show that is, given an $E$-convex function $f$ such that $Y \in \eth f^{\cA}(X)$, to derive $E$-convex functions $f_{t,s}$ for $s, t \in [0,1]$ such that $X_t \in \eth f_{t,s}^{\cA}(X_s)$.  The next proposition gives an explicit construction of $f_{t,s}$ from $f$, and gives the properties of $f_{t,s}$.  The specific cases relevant to the displacement interpolation are then summarized in Corollary \ref{cor:optimaltransportinterpolation}.  All of these results are completely analogous to the classical statements.
	
	\begin{proposition} \label{prop:functioninterpolation}
		Let $f$ be an $E$-convex function.  For $s, t \in [0,1]$, define $f_{t,s}$ as follows: For $s = 0$, set
		\[
		f_{t,0} = (1 - t) q_1 + t f; \qquad f_{0,t} = \mathcal{L} f_{t,0};
		\]
		if $s > 0$ and $s \leq t$, set
		\[
		f_{t,s}^{\cA}(X) = \inf_{Y \in L^2(\cA)_{\sa}^m} \left( \frac{t}{2s} \norm{X}_{L^2(\cA)_{\sa}^m}^2 - \frac{t - s}{s} \ip{X,Y}_{L^2(\cA)_{\sa}^m} + \frac{(t - s)(1 - s)}{2s} \norm{Y}_{L^2(\cA)_{\sa}^m}^2 + (t - s) f^{\cA}(Y) \right);
		\]
		if $s > 0$ and $s \geq t$, set
		\[
		f_{t,s}^{\cA}(X) = \sup_{Y \in L^2(\cA)_{\sa}^m} \left( \frac{t}{2s} \norm{X}_{L^2(\cA)_{\sa}^m}^2 - \frac{t - s}{s} \ip{X,Y}_{L^2(\cA)_{\sa}^m} + \frac{(t - s)(1 - s)}{2s} \norm{Y}_{L^2(\cA)_{\sa}^m}^2 + (t - s) f^{\cA}(Y) \right).
		\]
		(In particular, $f_{t,t} = q_1$ for all $t \in [0,1]$.)  Then we have the following:
		\begin{enumerate}[(1)]
			\item $f_{t,s}$ is $E$-convex and $f_{s,t} = \mathcal{L}f_{t,s}$.
			\item If $s \leq t$, then $f_{t,s} - \frac{1 - t}{1 - s} q_1$ is $E$-convex for $s < 1$ and $\frac{t}{s} q_1 - f_{t,s}$ is $E$-convex for $s > 0$.
			\item If $t \leq s$, then $f_{t,s} - \frac{t}{s} q_1$ is $E$-convex for $s > 0$ and $\frac{1 - t}{1 - s} q_1 - f_{t,s}$ is $E$-convex for $s < 1$.
			\item In particular, if $s \in (0,1)$ and $X \in L^2(\cA)_{\sa}^m$, then $\eth f_{t,s}^{\cA}(X)$ consists of a unique point $\nabla f_{t,s}^{\cA}(X)$ and $\nabla f_{t,s}^{\cA}$ is Lipschitz.
			\item Suppose $0 \leq s < t \leq 1$.  If $u \in (s,t)$, then
			\[
			f_{u,s} = \frac{t - u}{t - s} q_1 + \frac{u - s}{t - s} f_{t,s}
			\]
			and
			\[
			f_{t,u} = \left( \frac{t - s}{u - s} q_1 \right) \square \left( \frac{t - u}{t - s} f_{t,s} \left( \frac{t - s}{t - u} (\cdot) \right) \right).
			\]
			\item Suppose $0 \leq s < t \leq 1$ and $X, Y \in L^2(\cA)_{\sa}^m$ with $Y \in \eth f_{t,s}^{\cA}(X)$.  For $u \in [s,t]$, let
			\[
			X_u = \frac{t - u}{t - s} X + \frac{u - s}{t - s} Y.
			\]
			Then $X_u \in \eth f_{u,s}^{\cA}(X)$ and $Y \in \eth f_{t,u}^{\cA}(X_u)$.
			\item For $s, t, u \in (0,1)$, we have $\nabla f_{u,t} \circ \nabla f_{t,s} = \nabla f_{u,s}$.
		\end{enumerate}
	\end{proposition}
	
	The next corollary describes the most relevant cases of the proposition for optimal transport; the claims are special cases of (4) and (6) of the proposition.
	
	\begin{corollary} \label{cor:optimaltransportinterpolation}
		Let $(\cA,X,Y)$ be an optimal coupling of $\mu$, $\nu \in \Sigma_m$.  Let $f$ be an $E$-convex function such that $Y \in \eth f(X)$.  Let $f_{t,s}$ be as in Proposition \ref{prop:functioninterpolation}.  Let $X_t = (1-t)X + tY$ for $t \in [0,1]$.  Then $X_t \in \eth f_{t,s}(X_s)$ for all $s,t \in [0,1]$.  In particular, if $s \in (0,1)$, then $f_{t,s}$ has a Lipschitz gradient and we have $X_t = \nabla f_{t,s}(X_s)$.
	\end{corollary}
	
	In order to prove Proposition \ref{prop:functioninterpolation}, we need the following scaling relation for Legendre transform.
	
	\begin{lemma} \label{lem:Legendrescaling}
		Let $f$ be a tracial $\mathrm{W}^*$-function and let $c > 0$.  Then $\mathcal{L}(cf)^{\cA}(X) = c \mathcal{L}f^{\cA}(c^{-1}X)$.
	\end{lemma}
	
	\begin{proof}
		Observe that
		\begin{align*}
			\mathcal{L}(cf)^{\cA}(X) &= \sup_{\substack{\iota: \cA \to \cB \\ Y \in L^2(\cB)_{\sa}^m}} \left( \ip{\iota(X),Y}_{L^2(\cB)_{\sa}^m} - cf^{\cB}(Y) \right) \\
			&= \sup_{\substack{\iota: \cA \to \cB \\ Y \in L^2(\cB)_{\sa}^m}} \left( c\ip{\iota(c^{-1}X),Y}_{L^2(\cB)_{\sa}^m} - cf^{\cB}(Y) \right) \\
			&= c \mathcal{L}f^{\cA}(c^{-1}X). \qedhere
		\end{align*}
	\end{proof}
	
	The bulk of the proof of the proposition is the following lemma which explains how $f_{s,t}$ were obtained through addition of and inf-convolution with quadratic functions, using the same idea as in the proof of Theorem \ref{thm:displacementW*}.
	
	\begin{lemma} \label{lem:fts}
		Consider the setup of Proposition \ref{prop:functioninterpolation}.  If $0 < s < t \leq 1$, then
		\begin{align}
			f_{t,s} &= \frac{1-t}{1-s} q_1 + \frac{t-s}{1-s} \left[ \left( \frac{1}{s} q_1 \right) \square \left( (1 - s) f\left( \frac{1}{1 - s}(\cdot) \right) \right) \right] \label{eq:fts1} \\
			\mathcal{L}f_{t,s} &= \left( \frac{1 - s}{1 - t} q_1 \right) \square \left[ \frac{t-s}{1-s} [sq_1 + (1 - s) \mathcal{L} f]\left( \frac{1-s}{t-s} (\cdot) \right) \right] \label{eq:fts2}
		\end{align}
		and
		\begin{align}
			f_{t,s} &= \frac{t}{s} q_1 \square \left[ \frac{t - s}{t} [(1-t)q_1 + tf]\left( \frac{t}{t - s}(\cdot) \right) \right] \label{eq:fts3} \\
			\mathcal{L}f_{t,s} &= \frac{s}{t} q_1 + \frac{t - s}{t} \left[ \left(\frac{1}{1-t} q_1\right) \square t \mathcal{L}f \left(\frac{1}{t} (\cdot) \right) \right]. \label{eq:fts4}
		\end{align}
	\end{lemma}
	
	\begin{proof}
		Fix $\cA \in \mathbb{W}$ and $X \in L^2(\cA)_{\sa}^m$, and evaluate the right-hand side of \eqref{eq:fts1} at $X$ to obtain
		\begin{multline*}
			\frac{1-t}{1-s} q_1(X) + \frac{t-s}{1-s} \left[ \left( \frac{1}{s} q_1 \right) \square \left( (1 - s) f\left( \frac{1}{1 - s}(\cdot) \right) \right) \right](X) \\
			= \frac{1-t}{2(1-s)} \norm{X}_{L^2(\cA)_{\sa}^m}^2 + \frac{t-s}{1-s} \inf_{Y \in L^2(\cA)_{\sa}^m} \left[ \frac{1}{2s} \norm{X - Y}_{L^2(\cA)_{\sa}^m}^2 + (1 - s) f^{\cA}\left(\frac{1}{1-s} Y \right) \right],
		\end{multline*}
		where we have used the result from Lemma \ref{lem:infconvolutionconvex} that it suffices to take the infimum over $Y \in L^2(\cA)_{\sa}^m$ rather than $Y$ in $L^2(\cB)_{\sa}^m$ for some larger tracial $\mathrm{W}^*$-algebra $\cB$.  Next, we substitute $(1 - s)Y$ instead of $Y$ to obtain
		\begin{multline*}
			\frac{1-t}{2(1-s)} \norm{X}_{L^2(\cA)_{\sa}^m}^2 + \frac{t-s}{1-s} \inf_{Y \in L^2(\cA)_{\sa}^m} \left[ \frac{1}{2s} \norm{X - (1-s)Y}_{L^2(\cA)_{\sa}^m}^2 + (1 - s) f^{\cA}(Y) \right]
			\\ = \inf_{Y \in L^2(\cA)_{\sa}^m} \biggl[ \frac{1-t}{2(1-s)} \norm{X}_{L^2(\cA)_{\sa}^m}^2 + \frac{t - s}{2s(1-s)} \norm{X}_{L^2(\cA)_{\sa}^m}^2 - \frac{t - s}{s} \ip{X,Y}_{L^2(\cA)_{\sa}^m} \\ + \frac{(t-s)(1-s)}{2s} \norm{Y}_{L^2(\cA)_{\sa}^m}^2 + (t - s) f^{\cA}(Y) \biggr].
		\end{multline*}
		Combining the two coefficients in front of $\norm{X}_{L^2(\cA)_{\sa}^m}^2$, we arrive at the formula for $f_{t,s}^{\cA}(X)$.
		
		The equation \eqref{eq:fts2} is obtained from \eqref{eq:fts1} by applying the Legendre transform, using the fact that $\mathcal{L} (cq_1) = c^{-1} q_1$ for $c > 0$, the relation between Legendre transform and inf-convolution in Lemma \ref{lem:Legendreinfconvolution}, and the scaling relation Lemma \ref{lem:Legendrescaling}.
		
		The proof of \eqref{eq:fts3} is similar to the proof of \eqref{eq:fts1}, and then \eqref{eq:fts4} is obtained by taking the Legendre transform.
	\end{proof}
	
	\begin{proof}[{Proof of Proposition \ref{prop:functioninterpolation}}]
		(1) It is immediate that $f_{t,0}$ and $f_{0,t}$ are $E$-convex and are Legendre transforms of each other. Also, in the case of $s = t$, we have $f_{t,s} = q_1$, so there is nothing to prove.  For $0 < s < t \leq 1$, it follows from Lemma \ref{lem:fts} that $f_{t,s}$ is $E$-convex for because it is expressed by applying scaling, addition of quadratics, and inf-convolution with quadratics to $f$.
		
		Next, we show that for $0 < s < t$, we have $\mathcal{L}f_{t,s} = f_{s,t}$.  We evaluate $\mathcal{L} f_{t,s}$ starting from \eqref{eq:fts3} as
		\[
		\mathcal{L}f_{t,s} = \frac{s}{t} q_1 + \mathcal{L} \left[ \frac{t - s}{t} [(1-t)q_1 + tf]\left( \frac{t}{t - s}(\cdot) \right) \right]
		\]
		Now we evaluate the second term at some $X \in L^2(\cA)_{\sa}^m$, where $\cA \in \mathbb{W}$.  Using Remark \ref{rem:EconvexLegendre}, we may compute the Legendre transform of an $E$-convex function by taking the Hilbert-space Legendre transform for each $\cA$ (without considering a larger $\mathrm{W}^*$-algebra $\cB$).  This yields
		\[
		\sup_{Y \in L^2(\cA)_{\sa}^m} \left[ \ip{X,Y} - \frac{t - s}{t} [(1-t)q_1 + tf]^{\cA}\left( \frac{t}{t - s}Y \right) \right]. 
		\]
		We substitute $\frac{t-s}{t} Y$ for $Y$ to obtain
		\[
		\sup_{Y \in L^2(\cA)_{\sa}^m} \left[ \frac{t-s}{t} \ip{X,Y}_{L^2(\cA)_{\sa}^m} - \frac{(t - s)(1-t)}{2t} \norm{Y}_{L^2(\cA)_{\sa}^m}^2 - (t - s) f(Y) \right].
		\]
		Adding back the term $\frac{s}{t} q_1^{\cA}(X)$, we obtain
		\[
		\sup_{Y \in L^2(\cA)_{\sa}^m} \left[ \frac{s}{2t} \norm{X}_{L^2(\cA)_{\sa}^m}^2 - \frac{s-t}{t} \ip{X,Y}_{L^2(\cA)_{\sa}^m} + \frac{(s - t)(1-t)}{2t} \norm{Y}_{L^2(\cA)_{\sa}^m}^2 + (s - t) f(Y) \right],
		\]
		which is precisely $f_{s,t}$.
		
		Therefore, for $s > t > 0$, we have $f_{t,s} = \mathcal{L} f_{s,t}$, hence $f_{t,s}$ is $E$-convex and $\mathcal{L} f_{t,s} = f_{s,t}$.  This is the last remaining case.
		
		(2) Let $0 \leq s \leq t \leq 1$.  If $s = 0$, then $f_{t,s} = (1 - t) q_1 + tf$, so $f_{t,s} - \frac{1-t}{1-s} q_1 = tf$ is $E$-convex.  If $s \in (0,1)$, then by \eqref{eq:fts1}, $f_{t,s}$ is $\frac{1-t}{1-s} q_1$ plus an $E$-convex function, hence $f_{t,s} - \frac{1-t}{1-s} q_1$ is $E$-convex.  If $s \in (0,1]$, then by \eqref{eq:fts3}, $f_{t,s}$ is the inf-convolution of $\frac{t}{s} q_1$ with an $E$-convex function and therefore $\frac{t}{s} q_1 - f_{t,s}$ is $E$-convex by Proposition \ref{prop:convexityconversion}.
		
		(3) Let $s \geq t$.  Then $f_{t,s} = \mathcal{L}f_{s,t}$.  Thus, can we argue symmetrically to (2) using \eqref{eq:fts2} and \eqref{eq:fts4}.
		
		(4) This follows from (2) and (3) together with Proposition \ref{prop:convexsemiconcave}.
		
		(5) Consider the first relation $f_{u,s} = \frac{t - u}{t - s} q_1 + \frac{u - s}{t - s} f_{t,s}$.  If $s = 0$, this follows from direct computation from the definition of $f_{u,0}$ and $f_{t,0}$.  In the case $s > 0$, we apply \eqref{eq:fts1} to get
		\begin{align*}
			\frac{t - u}{t - s} q_1 + \frac{u - s}{t - s} f_{t,s} &= \frac{t - u}{t - s} q_1 + \frac{u - s}{t - s} \frac{1 - t}{1 - s} q_1 + \frac{u - s}{1 - s} \left[ \left( \frac{1}{s} q_1 \right) \square \left( (1 - s) f\left( \frac{1}{1 - s}(\cdot) \right) \right) \right] \\
			&= \frac{1 - u}{1 - s} q_1 + \frac{u - s}{1 - s} \left[ \left( \frac{1}{s} q_1 \right) \square \left( (1 - s) f\left( \frac{1}{1 - s}(\cdot) \right) \right) \right] \\
			&= f_{u,s}.
		\end{align*}
		Analogously, using \eqref{eq:fts4}, we obtain for $s \in (0,1)$ that
		\[
		\mathcal{L} f_{t,u} = \frac{u-s}{t-s} q_1 + \frac{t-u}{t-s} \mathcal{L} f_{t,s};
		\]
		in fact, this relation also holds when $s = 0$ by evaluating $\mathcal{L} f_{t,u}$ on the left-hand side with \eqref{eq:fts4} and evaluating $\mathcal{L} f_{t,0}$ on the right-hand side $ \mathcal{L}[(1-t)q_1 + tf] = (\frac{1}{1-t} q_1) \square t \mathcal{L}f(\frac{1}{t}(\cdot))$.  Taking Legendre transforms of the previous equation implies that
		\[
		f_{t,u} = \left( \frac{t - s}{u - s} q_1 \right) \square \left( \frac{t - u}{t - s} f_{t,s} \left( \frac{t - s}{t - u} (\cdot) \right) \right).
		\]
		
		(6) Since $X \in \eth  q_1^{\cA}(X)$ and $Y \in \eth f_{t,s}^{\cA}(X)$, we have
		\[
		X_u = \frac{t-u}{t-s} X + \frac{u-s}{t-s} Y \in \eth \left[ \frac{t-u}{t-s} q_1 + \frac{u-s}{t-s} f_{t,s} \right]^{\cA}(X) = \eth f_{u,s}^{\cA}(X).
		\]
		Since $Y \in \eth f_{t,s}^{\cA}(X)$, we have $X \in \eth (\mathcal{L} f_{t,s})^{\cA}(Y)$.  Hence, using the same relation as in the proof (5),
		\[
		X_u = \frac{u-s}{t-s} Y + \frac{t-u}{t-s} X \in \eth \left[ \frac{u-s}{t-s} q_1 + \frac{t-u}{t-s} \mathcal{L} f_{t,s}\right]^{\cA}(Y) = \eth(\mathcal{L}f_{t,u})^{\cA}(Y).
		\]
		So $X_u \in \eth(\mathcal{L} f_{t,u})^{\cA}(Y)$, so that $Y \in \eth f_{t,u}^{\cA}(X_u)$.
		
		(7) In light of (4), for $s, t \in (0,1)$, the functions $f_{s,t}$ and $f_{t,s}$ have Lipschitz gradients.  They are Legendre transforms of each other, which implies that $X \in \eth f_{s,t}(Y)$ if and only if $Y \in \eth f_{t,s}(X)$.  Hence, $\nabla f_{s,t} = (\nabla f_{t,s})^{-1}$.
		
		Suppose that $s < u < t$.  Let $Y = \nabla f_{t,s}(X)$, and let $X_u = \frac{u-s}{t-s}X + \frac{t-u}{t-s} Y$.  Then by (6), $X_u = \nabla f_{u,s}(X)$ and $Y = \nabla f_{t,u}(X_u)$, hence $\nabla f_{t,s}(X) = Y = \nabla f_{t,u}(X_u) = \nabla f_{t,u} \circ \nabla f_{u,s}(X)$.
		
		So $\nabla f_{t,s} = \nabla f_{t,u} \circ \nabla f_{u,s}$.  By applying $\nabla f_{s,u} = (\nabla f_{u,s})^{-1}$ on the right, we obtain $\nabla f_{t,s} \circ \nabla f_{s,u} = \nabla f_{t,u}$.  By taking inverses, $\nabla f_{u,s} \circ \nabla f_{s,t} = \nabla f_{u,t}$.  In fact, using composition and inverses in this way, we can achieve all permutations of $u$, $s$, and $t$.  The only remaining case is when some of $s, t, u$ are equal to each other, but this follows from the relations $\nabla f_{t,t} = \id$ and $\nabla f_{s,t} = (\nabla f_{t,s})^{-1}$.
	\end{proof}
	
	\section{Optimal couplings, quantum information theory, and operator algebras} \label{sec:qinfo}
	
	In this section, we give several indications of why non-commutative optimal couplings are significantly more complicated than the commutative case by making connections to other results in operator algebras and quantum information theory.  Specifically, using results from \cite{MuRo2020a}, we show that there exist $n \times n$ matrix tuples for which an optimal coupling requires a tracial $\mathrm{W}^*$-algebra of arbitrarily large dimension.  Next, based on \cite{HaMu2015} and \cite{JNVWY2020}, we conclude that there exist matrix tuples for which the optimal coupling requires a non-Connes-embeddable tracial $\mathrm{W}^*$-algebra (that is, it cannot even be approximated by couplings in finite-dimensional algebras).  Next, we show that the topology induced by the Wasserstein distance is strictly stronger than the weak-$*$ topology on $\Sigma_{m,R}$, and we characterize the points at which the two topologies agree.  Finally, we show that $\Sigma_{m,R}$ with the Wasserstein distance is not separable based on \cite[Theorem 1]{Ozawa2004}.
	
	\subsection{Completely positive and factorizable maps}
	
	We recall some standard definitions in operator algebras; see e.g.\ \cite{Paulsen2003}.  If $\cA$ is a tracial $\mathrm{W}^*$-algebra, we denote by $M_n(\cA)$ the algebra $M_n(L^\infty(\cA)) \cong M_n(\C) \otimes L^\infty(\cA)$ equipped with the trace $\tr_n \otimes \tau_{\cA}$ and the weak-$*$ topology given by the entrywise weak-$*$ topology on $L^\infty(\cA)$; it is a standard fact that $M_n(\cA)$ is indeed a tracial $\mathrm{W}^*$-algebra.  If $\Phi: \cA \to \cB$ is a linear map between tracial $\mathrm{W}^*$-algebras, then we define $\Phi^{(n)}: M_n(\cA) \to M_n(\cB)$ as the map obtained from entrywise application of $\Phi$.  If $\cA$ is a tracial $\mathrm{W}^*$-algebra and $a \in L^\infty(\cA)$, then we say that $a \geq 0$ if $a = x^*x$ for some $x \in L^\infty(\cA)$; this is equivalent to $a$ defining a positive operator on $L^2(\cA)$ by left multiplication.\footnote{Of course, definitions make sense more generally for $\mathrm{C}^*$-algebras.}
	
	\begin{definition}
		We say that $\Phi$ is \emph{completely positive} if for every $n \in \N$, if $a \in M_n(\cA)$ with $a \geq 0$, then $\Phi^{(n)}(a) \geq 0$.  For tracial $\mathrm{W}^*$-algebras $\cA$ and $\cB$, we denote by $\operatorname{CP}(\cA,\cB)$ the space of completely positive maps $\cA \to \cB$.  We denote by $\operatorname{UCPT}(\cA,\cB)$ the space of unital completely positive trace-preserving maps.  These maps are known in quantum information theory as \emph{quantum channels} from $\cA$ to $\cB$.
	\end{definition}
	
	\begin{definition}[{Anantharaman-Delaroche \cite{AD2006}}]
		Let $\cA$ and $\cB$ be tracial $\mathrm{W}^*$-algebras.  A linear map $\Phi: \cA \to \cB$ is said to be \emph{factorizable} if there exist tracial $\mathrm{W}^*$-inclusions $\iota_1: \cA \to \cC$ and $\iota_2: \cB \to \cC$ such that $\Phi = \iota_2^* \circ \iota_1$, where $\iota_2^*: \cC \to \cB$ is the conditional expectation adjoint to $\iota_2$.  We also say that $\Phi$ \emph{factorizes through $\cC$} if there exist $\iota_1$ and $\iota_2$ as above.
		
		We denote the space of factorizable maps by $\operatorname{FM}(\cA,\cB)$.  We denote by $\operatorname{FM}_{\operatorname{fin}}(\cA,\cB)$ the set of maps that factorize through a finite-dimensional algebra $\cC$.
	\end{definition}
	
	\begin{proposition} \label{prop:CPmaps}
		Let $\cA$, $\cB$, and $\cC$ be tracial $\mathrm{W}^*$-algebras.
		\begin{enumerate}[(1)]
			\item We have $\operatorname{FM}(\cA,\cB) \subseteq \operatorname{UCPT}(\cA,\cB)$.
			\item $\operatorname{UCPT}(\cA,\cB)$, $\operatorname{FM}(\cA,\cB)$, and $\operatorname{FM}_{\operatorname{fin}}(\cA,\cB)$ are convex sets.
			\item $\operatorname{UCPT}(\cA,\cB)$ and $\operatorname{FM}(\cA,\cB)$ are closed in the pointwise weak-$*$ topology.
			\item If $\Phi \in \operatorname{UCPT}(\cA,\cB)$ and $\Psi \in \operatorname{UCPT}(\cB,\cC)$, then $\Psi \circ \Phi \in \operatorname{UCPT}(\cA,\cC)$.  The same holds with $\operatorname{UCPT}$ replaced by $\operatorname{FM}$.
		\end{enumerate}
	\end{proposition}
	
	This proposition is well-known in operator algebras.  For the sake of exposition, let us recall why $\operatorname{FM}(\cA,\cB) \subseteq \operatorname{UCPT}(\cA,\cB)$.  Let $\Phi \in \operatorname{FM}(\cA,\cB)$, and take a factorization $\Phi = \iota_2^* \iota_1$ where $\iota_1: \cA \to \cC$ and $\iota_2: \cB \to \cC$ are tracial $\mathrm{W}^*$-inclusions.  Since $\iota_1$ and $\iota_2$ are $*$-homomorphisms, they are completely positive and unital.  Now observe that $\ip{\iota_2^*(c),b}_{L^2(\cB)_{\sa}^m} = \ip{c,\iota_2(b)}_{L^2(\cC)_{\sa}^m} \geq 0$ for $c \in M_n(\cC)_+$ and $b \in M_n(\cB)_+$; it follows that $\iota_2^*(c) \geq 0$ in $M_n(\cB)$ for every $c \in M_n(\cC)_+$.  Since $\iota_2$ is unital, $\iota_2^*$ is trace-preserving, and since $\iota_2$ is trace-preserving, $\iota_2^*$ is unital.  Finally, one verifies directly that $\operatorname{UCPT}$ is closed under composition, hence, $\iota_2^* \iota_1 \in \operatorname{UCPT}(\cA,\cB)$.
	
	To show that factorizable maps are closed under composition in (4), one uses amalgamated free products.  For convexity of $\operatorname{FM}(\cA,\cB)$, see e.g.\ \cite[Lemma 2.3.6]{BrownOzawa2008}.
	
	The next lemma summarizes some well-known facts about completely positive maps.
	
	\begin{lemma} \label{lem:CPadjoint}
		Let $\cA$ and $\cB$ be tracial $\mathrm{W}^*$-algebras, and let $\Phi \in \operatorname{UCPT}(\cA,\cB)$.
		\begin{enumerate}[(1)]
			\item $\Phi(X^*) = \Phi(X)^*$ for all $X \in L^\infty(\cA)$.
			\item $\Phi$ extends to a contractive map $L^2(\cA) \to L^2(\cB)$.
			\item There exists a unique $\Phi^* \in \operatorname{UCPT}(\cB,\cA)$ such that $\ip{X,\Phi^*(Y)}_{L^2(\cA)} = \ip{\Phi(X),Y}_{L^2(\cB)}$ for $X \in L^2(\cA)$ and $Y \in L^2(\cB)$.
		\end{enumerate}
	\end{lemma}
	
	%For instance, (1) follows because every $a$ can be written as a complex-linear combination of positive operators.  (2) follows from the inequality $\Phi(a^*a) \geq \Phi(a)^* \Phi(a)$ and the fact that $\Phi$ is trace-preserving.  (3) Since $\Phi$ maps $L^2(\cA)$ into $L^2(\cB)$, it has an adjoint $\Phi^*: L^2(\cB) \to L^2(\cA)$.  An element of $L^2(\cA)$ is positive and 
	
	The connection between factorizable maps and non-commutative optimal couplings is as follows.
	
	\begin{observation} \label{obs:factorizablecoupling}
		Let $\cA$ and $\cB$ be tracial $\mathrm{W}^*$-algebras and let $X \in L^\infty(\cA)_{\sa}^m$ and $Y \in L^\infty(\cB)_{\sa}^m$.  Then
		\[
		C(\lambda_X,\lambda_Y) = \sup_{\Phi \in \operatorname{FM}(\cA,\cB)} \ip{\Phi(X),Y}_{L^2(\cB)_{\sa}^m}.
		\]
	\end{observation}
	
	\begin{proof}
		In fact, we will show that the two sets $\{\ip{X',Y'}_{L^2(\cC)_{\sa}^m}: (\cC,X',Y') \text{ a coupling}\}$ and $\{\ip{\Phi(X),Y}_{L^2(\cB)_{\sa}^m}: \Phi \in \operatorname{FM}(\cA,\cB)\}$ are equal.  Suppose that $(\cC,X',Y')$ is a coupling of $\lambda_X$ and $\lambda_Y$.  Since $\lambda_{X'} = \lambda_X$, there is a tracial $\mathrm{W}^*$ embedding $\iota_1: \mathrm{W}^*(X) \to \cC$ sending $X$ to $X'$.  Similarly, there is a tracial $\mathrm{W}^*$-embedding $\iota_2: \mathrm{W}^*(Y) \to \cC$ sending $Y$ to $Y'$.  Let $\phi_1: \mathrm{W}^*(X) \to \cA$ and $\phi_2: \mathrm{W}^*(Y) \to \cB$ be the canonical inclusion maps, and let $\Phi = \phi_2 \iota_2^* \iota_1 \phi_1^*: \cA \to \cB$, which is factorizable by Proposition \ref{prop:CPmaps} (4) because it is a composition of factorizable maps.  Moreover,
		\[
		\ip{\Phi(X),Y}_{L^2(\cB)_{\sa}^m} = \ip{\iota_1 \phi_1^*(X), \iota_2 \phi_2^*(Y)}_{L^2(\cC)_{\sa}^m} = \ip{\iota_1(X),\iota_2(Y)}_{L^2(\cC)_{\sa}^m}.
		\]
		Conversely, given $\Phi \in \operatorname{FM}(\cA,\cB)$, we may factorize it as $\iota_2^* \iota_1$ for tracial $\mathrm{W}^*$-embeddings $\iota_1: \cA \to \cC$ and $\iota_2: \cB \to \cC$, and let $X' = \iota_1(X)$ and $Y' = \iota_2(Y)$ to obtain a coupling $(\cC,X',Y')$ of $\lambda_X$ and $\lambda_Y$.
	\end{proof}
	
	\subsection{Matrix tuples with optimal couplings of large dimension}
	
	This connection allows us to address a natural question:  Suppose that $\mu$ and $\nu$ are non-commutative laws that can be realized by self-adjoint tuples $X$ and $Y$ in a finite-dimensional algebra; then is there a non-commutative optimal coupling $(\cA,X',Y')$ of $\mu$ and $\nu$ such that $\cA$ is finite-dimensional?  And do we have some control over the dimension?  The classical analog of this question certainly has a positive answer. Indeed, if $\mu$ and $\nu$ are finitely supported measures on $\R^m$, with supports $S$ and $T$ respectively, then a classical optimal coupling is given by a measure $\pi$ on the product space $S \times T$.  Hence, there exist random variables $X$ and $Y \in L^2(S \times T,\pi;\R^m)$ such that $(\cA,X,Y)$ is an optimal coupling of $\mu$ and $\nu$, where $\cA$ is the finite-dimensional algebra $L^\infty(S \times T,\pi)$ equipped with the trace coming from $\pi$.
	
	Our first negative result in the non-commutative setting shows that, even in situations when an optimal coupling can occur in a finite dimensional algebra, there can be no control over its dimension.  This is a consequence of the following result of Musat and R{\o}rdam \cite{MuRo2020a}.
	
	\begin{theorem}[{Musat-R{\o}rdam \cite[Theorem 4.1]{MuRo2020a}}] \label{thm:infdimfactorization}
		If $n \geq 11$, then $\operatorname{FM}_{\fin}(M_n(\C),M_n(\C))$ is not closed, hence there exist factorizable maps $M_n(\C) \to M_n(\C)$ that do not factor through any finite-dimensional algebra. 
	\end{theorem}
	
	In order to translate this result into a statement about non-commutative optimal couplings, we use the following lemma, which is an application of the hyperplane separation theorem, vector space duality, and adjointness of tensor and hom functors.
	
	\begin{lemma} \label{lem:vectorduality}
		Let $L_{\R}(M_n(\C)_{\sa},M_m(\C)_{\sa})$ denote the space of real linear transformations $M_n(\C) \to M_m(\C)$.  Let $K \subseteq L_{\R}(M_n(\C)_{\sa},M_m(\C)_{\sa})$ be a closed convex set, and let $\Phi \not \in K$.  Then there exists $k \leq \min(n^2,m^2)$ and $X \in M_n(\C)_{\sa}^{k}$ and $Y \in M_n(\C)_{\sa}^{k}$ such that
		\[
		\ip{\Phi(X),Y}_{L^2(M_m(\C))_{\sa}^k} > \sup_{\Psi \in K} \ip{\Psi(X),Y}_{L^2(M_m(\C))_{\sa}^k}.
		\]
	\end{lemma}
	
	\begin{proof}
		Recall that there is a linear isomorphism
		\[
		T: L_{\R}(M_n(\C)_{\sa},M_m(\C)_{\sa}) \to L_{\R}(M_n(\C)_{\sa} \otimes M_m(\C)_{\sa},\R) = (M_n(\R)_{\sa} \otimes M_m(\R)_{\sa})^*
		\]
		that sends $\Psi \in L_{\R}(M_n(\C)_{\sa},M_m(\C)_{\sa})$ to the map
		\[
		\psi: M_n(\C)_{\sa} \otimes M_m(\C)_{\sa} \to \R: A \otimes B \mapsto \ip{\Psi(A),B}_{L^2(M_m(\C))_{\sa}}.
		\]
		Of course, $M_n(\C)_{\sa} \otimes M_m(\C)_{\sa}$ is finite-dimensional, so the double dual is isomorphic to the original space.  Applying the hyperplane separation theorem on the real inner-product space $M_n(\C)_{\sa} \otimes M_m(\C)_{\sa}$, we conclude that there exists some $v \in  M_n(\C)_{\sa} \otimes M_m(\C)_{\sa}$ such that
		\[
		T(\Phi)(v) > \sup_{\Psi \in K} T(\Psi)(v).
		\]
		
		Let us decompose $v$ into a sum of simple tensors $v = \sum_{j=1}^k X_j \otimes Y_j$.  The smallest $k$ for which this is possible is called the \emph{tensor rank} of $v$.  We claim that the tensor rank is at most $\min(n^2,m^2)$.  The reason is that for real vector spaces $V$ and $W$, we can identify $V \otimes W$ with $L_{\R}(V,W)$ and then apply the singular value decomposition of the matrix in $L_{\R}(V,W)$ corresponding to a given tensor $v$.  Since a matrix in $L_{\R}(V,W)$ has rank at most $\min(\dim V, \dim W)$, it follows that the tensor rank of $v$ is at most $\min(\dim V, \dim W)$.  In particular, taking $V = M_n(\C)_{\sa}$ and $W = M_m(\C)_{\sa}$, we see that our vector $v \in M_n(\C) \otimes M_m(\C)$ has tensor rank at most $\min(n^2,m^2)$.
		
		Let $X = (X_1,\dots,X_k)$ and $Y = (Y_1,\dots,Y_k)$.  Then for $\Psi \in L_{\R}(M_n(\C)_{\sa},M_m(\C)_{\sa})$, we have
		\[
		T(\Psi)(v) = \sum_{j=1}^k \ip{\Psi(X_j),Y_j}_{L^2(M_m(\C))_{\sa}} = \ip{\Psi(X),Y}_{L^2(M_m(\C))_{\sa}^k}
		\]
		Thus, by our choice of $v$, the tuples $X$ and $Y$ satisfy the desired properties.
	\end{proof}
	
	\begin{corollary} \label{cor:largedimcoupling}
		If $n \geq 11$ and $d \in \N$, then there exist $X, Y \in M_n(\C)_{\sa}^{n^2}$ such that for every optimal coupling $(\cA,X',Y')$ of $\lambda_X$ and $\lambda_Y$, $\cA$ must have dimension at least $d$.  In particular, if $d$ is sufficiently large, then
		\[
		C(\lambda_X,\lambda_Y) > \sup_{U \in \mathcal{U}(M_n(\C))} \ip{UXU^*,Y}_{L^2(M_n(\C))_{\sa}^m}.
		\]
	\end{corollary}
	
	\begin{proof}
		Let $\operatorname{FM}_d(M_n(\C),M_n(\C))$ denote the set of $\operatorname{UCPT}$ maps $M_n(\C) \to M_n(\C)$ that factorize through a tracial $W^*$-algebra $\cA = (A,\tau)$ of dimension at most $d$.  As a consequence of the Artin-Wedderburn theorem, every such $*$-algebra $A$ is a direct sum of at most $d$ matrix algebras of size at most $d^{1/2}$; see e.g.\ \cite{Effros1981}.  Moreover, the trace $\tau_{\cA}$ is a convex combination of the traces on each component.  From these facts, it is not hard to see that $\operatorname{FM}_d(M_n(\C),M_n(\C))$ is compact.
		
		By Theorem \ref{thm:infdimfactorization}, there exists $\Phi \in \overline{\operatorname{FM}_{\fin}(M_n(\C),M_n(\C))}$ that does not factor through a finite-dimensional algebra, and hence $\Phi \in \overline{\operatorname{FM}_{\fin}(M_n(\C),M_n(\C))} \setminus \operatorname{FM}_d(M_n(\C),M_n(\C))$.

		Also, we also remark that a completely positive map $\Phi: M_n(\C) \to M_n(\C)$ satisfies $\Phi(A^*) = \Phi(A)^*$, and therefore it restricts to a real-linear transformation $M_n(\C)_{\sa} \to M_n(\C)_{\sa}$, and $\Phi$ is uniquely determined by its restriction to self-adjoint elements.  Thus, we can Lemma \ref{lem:vectorduality} to conclude that there exists $k \leq n^2$ and $X, Y \in M_n(\C)_{\sa}^{k}$ such that
		\[
		\ip{\Phi(X),Y}_{L^2(M_n(\C))_{\sa}^k} > \sup_{\Psi \in \operatorname{FM}_{d}(M_n(\C),M_n(\C))} \ip{\Psi(X),Y}_{L^2(M_n(\C))_{\sa}^k}.
		\]
		We can without loss of generality take $k = n^2$ because we can always add additional zero entries to our tuples without changing the value of the inner product of $\Psi(X)$ and $Y$.  Hence, by the proof of Observation \ref{obs:factorizablecoupling} any $\Psi \in \operatorname{FM}_{d}(M_n(\C),M_n(\C))$ cannot produce an optimal coupling.
	\end{proof}
	
	\subsection{Optimal couplings and the Connes embedding problem} \label{subsec:CEP}
	
	The situation is even more wild than this.  Based on the work of \cite{JNVWY2020} on Tsirelson's problem and the Connes embedding problem, as well as work of \cite{HaMu2015}, we can conclude that for some $n$, there exist $X$, $Y \in M_n(\C)_{\sa}^{n^2}$, such that a non-commutative optimal coupling of $\lambda_X$ and $\lambda_Y$ cannot even be \emph{approximated} by couplings in finite-dimensional tracial $*$-algebras.  We begin with some background on the Connes embedding problem, which includes first the definition of ultraproducts of tracial $\mathrm{W}^*$-algebras, a tool to turn approximate embeddings into literal embeddings; \cite[Appendix A]{BrownOzawa2008}, \cite[\S 2]{Capraro2010}, \cite[\S 5.4]{ADP}.
	
	Let $\beta \N$ denote the \emph{Stone-{\v C}ech compactification} of the natural numbers; it is a compact space containing $\N$ as an open subset,\footnote{Here $\N$ is equipped with the discrete topology.  To simplify notation, we view $\N$ as a subset of $\beta \N$ rather than considering an inclusion map $\eta: \N \to \beta \N$.} and satisfies the universal property that every function from $\N$ into a compact topological space $K$ extends uniquely to a continuous function $\beta \N \to K$.  In particular, if $(x_n)_{n \in \N}$ is a bounded sequence of real or complex numbers, then $\lim_{n \to \mathcal{U}} x_n$ exists for every $\mathcal{U} \in \beta \N$.  The Stone--{\v C}ech compactification $\beta \N$ is characterized up to a canonical homeomorphism by its universal property.  One construction of $\beta \N$ is by way of ultrafilters, which is why we have used the letter $\mathcal{U}$ for elements of $\beta \N$. In this framework, the elements of $\beta \N \setminus \N$ are known as \emph{non-principal ultrafilters} and the limit $\lim_{n \to \mathcal{U}} x_n$ is also called an \emph{ultralimit}.
	
	Ultraproducts of tracial von Neumann algebras are defined as follows.  For $n \in \N$, let $\cA_n = (A_n,\tau_n)$ be a sequence of tracial $\mathrm{W}^*$-algebras.  Let $\prod_{n \in \N} A_n$ be the set of sequences $(a_n)_{n \in \N}$ such that $\sup_n \norm{a_n}_{L^\infty(\cA_n)} < \infty$, which is a $*$-algebra.  Let
	\[
	I_{\mathcal{U}} = \left\{(a_n)_{n \in \N} \in \prod_{n \in \N} A_n: \lim_{n \to {\mathcal{U}}} \norm{a_n}_{L^2(\cA_n)} = 0 \right\}.
	\]
	Using the non-commutative H\"older's inequality for $L^2$ and $L^\infty$, one can show that $I_{\mathcal{U}}$ is an ideal in $\prod_{n \in \N} A_n$, and therefore, $\prod_{n \in \N} A_n / I_{\mathcal{U}}$ is a $*$-algebra.  We denote by $[a_n]_{n \in \N}$ the equivalence class in $\prod_{n \in \N} A_n / I_{\mathcal{U}}$ of a sequence $(a_n)_{n \in \N} \in \prod_{n \in \N} A_n$. Furthermore, we define a trace on $\prod_{n \in \N} A_n / I_\mathcal{U}$ as follows.  if $(a_n)_{n \in \N} \in \prod_{n \in \N} A_n$, then $(\tau_n(a_n))_{n \in \N}$ is a bounded sequence in $\C$ and therefore $\lim_{n \to \mathcal{U}} \tau_n(a_n)$ exists.  Since $|\tau_n(a_n)| \leq \norm{a_n}_{L^2(\cA_n)}$, we have $\lim_{n \to \mathcal{U}} \tau_n(a_n) = 0$ whenever $(a_n)_{n \in \N} \in I_\mathcal{U}$.  Therefore, there is a well-defined map
	\[
	\tau_{\mathcal{U}}: \prod_{n \in \N} A_n / I_{\mathcal{U}} \to \C
	\]
	given by $\tau_{\mathcal{U}}([a_n]_{n \in \N}) = \lim_{n \to \mathcal{U}} \tau_n(a_n)$.  It turns out the pair $(\prod_{n \in \N} A_n / I_{\mathcal{U}},\tau_{\mathcal{U}})$ is already a tracial $\mathrm{W}^*$-algebra; see \cite[Proposition 5.4.1]{ADP}.  The proof is based on the fact that a tracial $\mathrm{C}^*$-algebra is a $\mathrm{W}^*$-algebra if and only if the operator-norm unit ball is complete in the $L^2$ norm \cite[Proposition 2.6.4]{ADP}.  See also \cite[Appendix A]{BrownOzawa2008}.
	
	We call the tracial $\mathrm{W}^*$-algebra $(\prod_{n \in \N} A_n / I_{\mathcal{U}},\tau_{\mathcal{U}})$ the \emph{ultraproduct of $(\cA_n)_{n \in \N}$ with respect to $\mathcal{U}$} and we denote it by
	\[
	\prod_{n \to \mathcal{U}} \cA_n := \left(\prod_{n \in \N} A_n / I_{\mathcal{U}},\tau_{\mathcal{U}} \right).
	\]
	The inspiration for this notation is that ultraproduct is defined using a combination of Cartesian product and ultralimits (of the $L^2$-norm and the trace); in contrast to Cartesian products, the ultraproduct only cares about the asymptotic behavior of a sequence as $n \to \mathcal{U}$.
	
	\begin{definition}
		A tracial $\mathrm{W}^*$-algebra $\cA$ is \emph{Connes-embeddable} if there exist finite-dimensional tracial $*$-algebras $\cA_n$ for $n \in \N$, an ultrafilter $\mathcal{U} \in \beta \N \setminus \N$, and a tracial $\mathrm{W}^*$-embedding $\phi: \cA \to \prod_{n \to \mathcal{U}} \cA_n$.  The \emph{Connes embedding problem} is the question of whether every tracial $\mathrm{W}^*$-algebra with separable predual is Connes-embeddable.
	\end{definition}
	
	Embeddings into ultraproducts are closely related to convergence of non-commutative laws in $\Sigma_{m,R}$.
	
	\begin{lemma} \label{lem:ultraproductlaws}
		Let $(\cA_n)_{n \in \N}$ be a sequence of tracial $\mathrm{W}^*$-algebras and let $\cA$ be another tracial $\mathrm{W}^*$-algebra.  Let $X \in L^\infty(\cA)_{\sa}^m$ with $\norm{X}_{L^\infty(\cA)_{\sa}^m} \leq R$ and suppose that $X$ generates $\cA$ as a $\mathrm{W}^*$-algebra.  Let $X_n \in L^\infty(\cA_n)_{\sa}^m$ with $\norm{X_n}_{L^\infty(\cA_n)_{\sa}^m} \leq R$.  Then the following are equivalent:
		\begin{enumerate}[(1)]
			\item $\lim_{n \to \mathcal{U}} \lambda_{X_n} = \lambda_X$ with respect to the weak-$*$ topology on $\Sigma_{m,R}$.
			\item There exists a tracial $\mathrm{W}^*$-embedding $\phi: \cA \to \prod_{n \to \mathcal{U}} \cA_n$ such that $\phi(X) = [X_n]_{n \in \N}$.
		\end{enumerate}
	\end{lemma}
	
	\begin{proof}
		(1) $\implies$ (2).  Let $Y = [X_n]_{n\in\N} \in L^\infty(\prod_{n \to \mathcal{U}} \cA_n)_{\sa}^m$.  Let $\tau_\mathcal{U}$ be the trace on the ultraproduct.  Then for every $p \in \C\ip{x_1,\dots,x_d}$, we have
		\[
		\lambda_Y(p) = \tau_\mathcal{U}(p(Y)) = \lim_{n \to \mathcal{U}} \tau_n(p(X_n)) = \lim_{n \to \mathcal{U}} \lambda_{X_n}(p) = \lambda_X(p).
		\]
		Because $\lambda_Y = \lambda_X$, Lemma \ref{lem:lawisomorphism} implies that there is a $\mathrm{W}^*$-embedding $\cA = \mathrm{W}^*(X) \to \mathrm{W}^*(Y) \hookrightarrow \prod_{n \to \mathcal{U}} \cA_n$.
		
		(2) $\implies$ (1).  Let $\phi: \cA \to \prod_{n \to \mathcal{U}} \cA_n$ as above be a tracial $\mathrm{W}^*$-embedding with $\phi(X) = [X_n]_{n \in \N}$.  Using the fact that $\phi$ preserves addition and multiplication as well as the definition of the trace $\tau_\mathcal{U}$ on the ultraproduct,
		\[
		\lambda_X(p) = \tau(p(X)) = \tau_\mathcal{U}(\phi(p(X))) = \tau_\mathcal{U}(p(\phi(X))) = \lim_{n \to \mathcal{U}} \tau_n(p(X_n)) = \lim_{n \to \mathcal{U}} \lambda_{X_n}(p).
		\]
		Therefore, $\lim_{n \to \mathcal{U}} \lambda_{X_n} = \lambda_X$ in the weak-$*$ topology, as desired.
	\end{proof}
	
	\begin{definition}
		Let $\Sigma_{m,R}^{\fin}$ be the set of non-commutative laws $\mu$ in $\Sigma_{m,R}$ such that $\mu = \lambda_X$ for some $X \in L^\infty(\cA)_{\sa}^m$ where $\cA$ is a finite-dimensional tracial $*$-algebra.
	\end{definition}
	
	The following statement is almost a corollary of Lemma \ref{lem:ultraproductlaws}.
	
	\begin{lemma} \label{lem:ultraproductlaws2}
		Let $\cA$ be a tracial $\mathrm{W}^*$-algebra generated by $X \in L^\infty(\cA)_{\sa}^m$ with $\norm{X}_{L^\infty(\cA)^m} \leq R$.  Then $\cA$ is Connes-embeddable if and only if $\lambda_X$ is in the weak-$*$ closure of $\Sigma_{m,R}^{\fin}$ in $\Sigma_{m,R}$.
	\end{lemma}
	
	\begin{proof}
		If $\lambda_X$ is in the closure of $\Sigma_{m,R}^{\fin}$, then Lemma \ref{lem:ultraproductlaws} implies that $\cA$ is Connes-embeddable.  Conversely, suppose that $\cA$ is Connes-embeddable and $\iota: \cA\to \prod_{n \to \mathcal{U}} \cA_n$ is an embedding into some ultraproduct of finite-dimensional tracial $*$-algebras.  Let $X_n = (X_n^{(1)},\dots,X_n^{(m)}) \in L^\infty(\cA_n)^m$ such that $[X_n]_{n \in \N} = \iota(X)$, and let us also write $X = (X^{(1)},\dots,X^{(m)})$.  By replacing $X_{n}^{(j)}$ with $(X_n^{(j)} + (X_n^{(j)})^*)/2$, we can assume without loss of generality that $X_n^{(j)}$ is self-adjoint.  By assumption $M := \sup_{n \in \N} \norm{X_n}_{L^\infty(\cA)_{\sa}^m} < \infty$.
		
		Although $M$ may be larger than $R$, we can fix this issue through a standard argument with functional calculus.  Let $f: [-M,M] \to [-R,R]$ be given by $f(t) = \sgn(t) \min(|t|,R)$.  By the Weierstrass approximation theorem, there exists a sequence of polynomials $(f_k)_{k \in \N}$ converging uniformly to $f$ on $[-M,M]$.  Note that $f_k(\iota(X^{(j)})) = [f_k(X_n^{(j)})]_{n \to \N}$.  By the spectral mapping theorem, for each $j$, $k$, and $n$,
		\[
		\norm{f_k(X_n^{(j)}) - f(X_n^{(j)})}_{L^\infty(\cA_n)} \leq \sup_{t \in [-M,M]} |f_k(t) - f(t)|,
		\]
		and the same estimate holds for $f_k(\iota(X^{(j)})) - f(\iota(X^{(j)}))$.  Taking $k \to \infty$, we obtain $f(\iota(X^{(j)})) = [f(X_n^{(j)})]_{n \in \N}$ for each $j$.  Since $\norm{X^{(j)}}_{L^\infty(\cA)} \leq R$, we have $f(\iota(X^{(j)})) = \iota(X^{(j)})$.  Let $Y_n = (f(X_n^{(1)}),\dots,f(X_n^{(m)}))$.  Then $\norm{Y_n}_{L^\infty(\cA_n)^m} \leq R$ and $\iota(X) = [Y_n]_{n \in \N}$, hence $\lambda_X$ is in the closure of $\Sigma_{m,R}^{\fin}$ by Lemma \ref{lem:ultraproductlaws}.
	\end{proof}
	
	Decades of work found many equivalent problems in operator algebras and quantum information theory; for a survey, see e.g.\ \cite{Capraro2010,Ozawa2013}.  In particular, building on the established connections with quantum information theory, Haagerup and Musat showed the following result.
	
	\begin{theorem}[{Haagerup-Musat \cite[Theorem 3.6, 3.7]{HaMu2015}}] \label{thm:CEPfactorizable}
		A factorizable map $\Phi: M_n(\C) \to M_n(\C)$ admits a factorization through a Connes-embeddable algebra if and only if it is in the closure of $\operatorname{FM}_{\fin}(M_n(\C),M_n(\C))$.  Moreover, the Connes embedding problem has a positive answer if and only if
		\[
		\operatorname{FM}(M_n(\C),M_n(\C)) = \overline{\operatorname{FM}_{\fin}(M_n(\C),M_n(\C))} \text{ for all } n \in \N.
		\]
	\end{theorem}
	
	A negative answer to the Connes embedding problem was announced in \cite{JNVWY2020}.  This implies the following corollary.
	
	\begin{corollary} \label{cor:nonConnes}
		There exist $n \in \N$ and $X, Y \in M_n(\C)_{\sa}^{n^2}$ such that
		\[
		C(\lambda_X,\lambda_Y) = \sup_{\Phi \in \operatorname{FM}(M_n(\C),M_n(\C))} \ip{\Phi(X),Y}_{L^2(M_n(\C))_{\sa}^{n^2}} > \sup_{\Phi \in \operatorname{FM}_{\fin}(M_n(\C),M_n(\C))} \ip{\Phi(X),Y}_{L^2(M_n(\C))_{\sa}^{n^2}}.
		\]
		Moreover, a non-commutative optimal coupling of $\lambda_X$ and $\lambda_Y$ does not exist in any Connes-embeddable tracial $\mathrm{W}^*$-algebra.
	\end{corollary}
	
	\begin{proof}
		Let $K = \overline{\operatorname{FM}_{\operatorname{fin}}(M_n(\C),M_n(\C))}$, which is compact and convex. Because the Connes embedding problem has a negative answer \cite{JNVWY2020}, there exists $\Phi \in \operatorname{FM}(M_n(\C),M_n(\C)) \setminus K$.  By Lemma \ref{lem:vectorduality}, there exist $X$, $Y \in M_n(\C)_{\sa}^{n^2}$ such that
		\[
		\ip{\Phi(X),Y}_{L^2(M_n(\C))_{\sa}^{n^2}} > \sup_{\Psi \in K} \ip{\Psi(X),Y}_{L^2(M_n(\C))_{\sa}^{n^2}}.
		\]
		Hence, by Theorem \ref{thm:CEPfactorizable}, if $\Psi$ factors through a Connes-embeddable algebra, then $\ip{\Psi(X),Y}_{L^2(M_n(\C))_{\sa}^{n^2}}$ cannot be optimal.  Thus, by the proof of Observation \ref{obs:factorizablecoupling}, a coupling of $\lambda_X$ and $\lambda_Y$ in a Connes-embeddable algebra cannot be optimal.
	\end{proof}
	
	\begin{remark} \label{rem:largedimcoupling}
		Although Corollary \ref{cor:nonConnes} is much stronger than Corollary \ref{cor:largedimcoupling} as stated, they are based on different types of phenomena.  Corollary \ref{cor:nonConnes} relies on the existence of factorizable maps $M_n(\C) \to M_n(\C)$ that cannot be approximated by elements of $\operatorname{FM}_{\fin}(M_n(\C),M_n(\C))$ (of which there are not yet explicit examples known).  Meanwhile, Corollary \ref{cor:largedimcoupling} relies on the existence of factorizable maps that are approximated by elements of $\operatorname{FM}_{\fin}(M_n(\C),M_n(\C))$ but are not in $\operatorname{FM}_{\fin}(M_n(\C),M_n(\C))$ (of which \cite{MuRo2020a} gave explicit examples).  Thus, the proof of Corollary \ref{cor:largedimcoupling} shows that for $n \geq 11$ and $d \in \N$, there exist tuples $X$ and $Y$ from $M_n(\C)_{\sa}^{n^2}$ such that
		\[
		\sup_{\Phi \in \overline{\operatorname{FM}_{\fin}(M_n(\C),M_n(\C))}} \ip{\Phi(X),Y}_{L^2(M_n(\C))_{\sa}^{n^2}} > \sup_{\Psi \in \operatorname{FM}_d(M_n(\C),M_n(\C))} \ip{\Psi(X),Y}_{L^2(M_n(\C))_{\sa}^{n^2}}.
		\]
		Hence, a coupling on an algebra $\cA$ of dimension at most $d$ cannot even be optimal among couplings in Connes-embeddable algebras.
	\end{remark}
	
	\subsection{The Wasserstein and weak-$*$ topologies} \label{subsec:twotopologies}
	
	At the beginning, we equipped $\Sigma_{m,R}$ with the weak-$*$ topology as a subset of the algebraic dual of $\C\ip{x_1,\dots,x_m}$.  Meanwhile, because $d_W^{(2)}$ defines a metric on $\Sigma_{m,R}$, it induces another topology, which we will call the \emph{Wassertein topology}.  We will show that the Wasserstein topology is strictly stronger than the weak-$*$ topology.  This is to be contrasted with classical probability theory where the weak-$*$ topology on the space of probability measures on $[-R,R]^m$ is metrized by the $L^2$-Wasserstein distance.
	
	Our first step is to prove an ultraproduct characterization of Wasserstein convergence analogous to Lemma \ref{lem:ultraproductlaws}.
	
	\begin{lemma} \label{lem:ultraproductWasserstein}
		Let $(\cA_n)_{n \in \N}$ be a sequence of tracial $\mathrm{W}^*$-algebras and let $\cA$ be another tracial $\mathrm{W}^*$-algebra.  Let $X \in L^\infty(\cA)_{\sa}^m$ with $\norm{X}_{L^\infty(\cA)_{\sa}^m} \leq R$ and suppose that $X$ generates $\cA$.  Let $X_n \in L^\infty(\cA_n)_{\sa}^m$ with $\norm{X_n}_{L^\infty(\cA_n)_{\sa}^m} \leq R$.  Then the following are equivalent:
		\begin{enumerate}[(1)]
			\item $\lim_{n \to \mathcal{U}} \lambda_{X_n} = \lambda_X$ with respect to Wasserstein distance.
			\item There exists a tracial $\mathrm{W}^*$-embedding $\phi: \cA \to \prod_{n \to \mathcal{U}} \cA_n$ and a factorizable map $\Phi_n \in \operatorname{FM}(\cA,\cA_n)$ (for each $n \in \N$) such that
			\[
			\phi(X) = [X_n]_{n \in \N}, \quad \phi(Z) = [\Phi_n(Z)]_{n \in \N} \text{ for all } Z \in L^\infty(\cA).
			\]
		\end{enumerate}
	\end{lemma}
	
	\begin{proof}
		(1) $\implies$ (2).  The limit $\lim_{n \to \mathcal{U}} \lambda_{X_n} = \lambda_X$ in Wasserstein distance means that there exists tracial $\mathrm{W}^*$ algebras $\cB_n$ and tracial $\mathrm{W}^*$-embeddings $\pi_n: \mathrm{W}^*(X_n) \to \cB_n$ and $\rho_n: \cA \to \cB_n$ such that $\norm{\pi_n(X_n) - \rho_n(X)}_{L^2(\cB_n)_{\sa}^m} \to 0$ as $n \to \mathcal{U}$.  Let $\cC_n$ be the free product of $\cA_n$ and $\cB_n$ with amalgamation over $\mathrm{W}^*(X_n)$, and let $\tilde{\pi}_n: \cA_n \to \cC_n$ and $\tilde{\rho}_n: \cA \to \cC_n$ be the corresponding tracial $\mathrm{W}^*$-embeddings.  It is straightforward to check that these induce tracial $\mathrm{W}^*$-embeddings
		\[
		\tilde{\pi}: \prod_{n \to \mathcal{U}} \cA_n \to \prod_{n \to \mathcal{U}} \cC_n, \qquad \tilde{\rho}: \cA \to \prod_{n \to \mathcal{U}} \cC_n
		\]
		such that $\tilde{\pi}(\phi(X)) = \tilde{\pi}([X_n]_{n\in \N}) = \rho(X)$.  Since $\tilde{\pi} \circ \phi$ and $\tilde{\rho}$ are tracial $\mathrm{W}^*$-embeddings, we have $\tilde{\pi}(\phi(Z)) = \tilde{\rho}(Z)$ for all $Z \in L^\infty(\cA)$ (because for instance every element of $L^\infty(\cA)$ can be approximated in $L^2(\cA)$ by non-commutative polynomials of $X$).
		
		Let $\tilde{\pi}_n^*$ and $\tilde{\pi}^*$ be the trace-preserving conditional expectations adjoint to $\tilde{\pi}_n$ and $\tilde{\pi}$.  We claim that for $Y = [Y_n]_{n \in \N} \in \prod_{n \to \mathcal{U}} \cC_n$, we have
		\[
		\tilde{\pi}^*(Y) =  [\tilde{\pi}_n^*(Y_n)]_{n \in \N}.
		\]
		Let $\tilde{\cA} = \prod_{n \to \mathcal{U}} \cA_n$ and $\tilde{\cC} = \prod_{n \to \mathcal{U}} \cC_n$.  Note that $[\tilde{\pi}_n^*(Y_n)]_{n \in \N}$ is in the $\mathrm{W}^*$-subalgebra $\tilde{\cA} = \prod_{n \to \mathcal{U}} \cA_n$.  Moreover, for every $Z = [Z_n] \in \prod_{n \to \mathcal{U}} \cA_n$, we have
		\[
		\ip{Y,\tilde{\pi}(Z)}_{L^2(\tilde{\cC})} = \lim_{n \to \mathcal{U}} \ip{Y_n, \tilde{\pi}_n(Z_n)}_{L^2(\cC_n)} = \lim_{n \to \mathcal{U}} \ip{\tilde{\pi}_n^*(Y_n),Z_n}_{L^2(\cA_n)} = \ip{[\tilde{\pi}_n^*(Y_n)]_{n \in \N}, Z}_{L^2(\tilde{\cA})}.
		\]
		Thus, $\tilde{\pi}^*(Y) = [\tilde{\pi}_n^*(Y_n)]_{n \in \N}$, as desired.  As noted above, for every $Z \in \cA$, we have $\tilde{\pi}(\phi(Z)) = \tilde{\rho}(Z)$ and hence $\phi(Z) = \tilde{\pi}^* \tilde{\pi} \phi(Z) = \tilde{\pi}^* \tilde{\rho}(Z)$.  This implies that
		\[
		[\tilde{\pi}_n^* \tilde{\rho}_n(Z)]_{n \in \N} = \tilde{\pi}^* \tilde{\rho}(Z) = \phi(Z).
		\]
		Therefore, $\Phi_n := \tilde{\pi}_n^* \tilde{\rho}_n$ is a factorizable map fulfilling condition (2).
		
		(2) $\implies$ (1).  Let $\phi$ and $\Phi_n$ be as in (2).  Then $[X_n]_{n \in \N} = \phi(X) = [\Phi_n(X)]_{n \in \N}$ belongs to $\prod_{n \to \mathcal{U}} \cA_n$.  Letting $E_n$ be the trace-preserving conditional expectation $\cA_n \to \mathrm{W}^*(X_n)$, the map $E_n \circ \Phi_n: \mathrm{W}^*(X) \to \mathrm{W}^*(X_n)$ is factorizable by Proposition \ref{prop:CPmaps} (4), hence by Observation \ref{obs:factorizablecoupling},
		\[
		C(\lambda_{X_n},X) \geq \ip{E_n \circ \Phi_n(X),X_n}_{L^2(\mathrm{W}^*(X_n))_{\sa}^m} = \ip{\Phi_n(X),X_n}_{L^2(\cA_n)_{\sa}^m}.
		\]
		Therefore,
		\begin{align*}
			\lim_{n \to \mathcal{U}} d_W^{(2)}(\lambda_{X_n},\lambda_X)^2 &= \lim_{n \to \mathcal{U}} \left( \norm{X_n}_{L^2(\cA_n)_{\sa}^m}^2 + \norm{X}_{L^2(\cA)_{\sa}^m}^2 - 2C(\lambda_{X_n},\lambda_X) \right) \\
			&\leq \lim_{n \to \mathcal{U}} \left( \norm{X}_{L^2(\cA)_{\sa}^m}^2 + \norm{X_n}_{L^2(\cA_n)_{\sa}^m}^2 - 2\ip{\Phi_n(X),X_n}_{L^2(\cA_n)_{\sa}^m} \right) \\
			&= \norm{\phi(X)}_{L^2(\prod_{n \to \mathcal{U}} \cA_n)_{\sa}^m}^2 + \norm{\phi(X)}_{L^2(\prod_{n \to \mathcal{U}} \cA_n)_{\sa}^m}^2 - 2 \ip{\phi(X),\phi(X)}_{L^2(\prod_{n \to \mathcal{U}} \cA_n)_{\sa}^m}^2 \\
			&= 0.
		\end{align*}
		Hence, $\lim_{n \to \mathcal{U}} \lambda_{X_n} = \lambda_X$ in Wasserstein distance.
	\end{proof}
	
	The next corollary was observed in \cite[Proposition 1.4(b)]{BV2001}, and can be proved in several ways (see for instance \cite[Lemma 2.10, Corollary 2.11]{HJNS2021} for another method), but we will deduce it as a consequence of the ultraproduct characterizations for weak-$*$ and Wasserstein convergence.
	
	\begin{corollary} \label{cor:twotopologies}
		The Wasserstein topology on $\Sigma_{m,R}$ refines the weak-$*$ topology.
	\end{corollary}
	
	\begin{proof}
		Fix $\mathcal{U} \in \beta \N \setminus \N$.  Using the Urysohn subsequence principle, it suffices to show that if $\mu_n, \mu \in \Sigma_{m,R}$ and $\lim_{n \to \mathcal{U}} \mu_n = \mu$ in the Wasserstein distance, then $\lim_{n \to \mathcal{U}} \mu_n \to \mu$ in the weak-$*$ topology.  Letting $(\cA_n,X_n)$ and $(\cA,X)$ be the GNS realizations of $\mu_n$ and $\mu$, Lemma \ref{lem:ultraproductWasserstein} implies that there is a tracial $\mathrm{W}^*$-embedding $\cA \to \prod_{n \to \mathcal{U}} \cA_n$ with $\phi(X) = [X_n]_{n \in \N}$.  By Lemma \ref{lem:ultraproductlaws}, this implies that $\lim_{n \to \mathcal{U}} \mu_n = \mu$ in the weak-$*$ topology.
	\end{proof}
	
	The next observation is closely related.
	
	\begin{lemma} \label{lem:LSC}
		The metric $d_W^{(2)}$ is weak-$*$ lower semi-continuous on $\Sigma_{m,R} \times \Sigma_{m,R}$.
	\end{lemma}
	
	\begin{proof}
		Fix $\mathcal{U} \in \beta \N \setminus \N$.  Again using the Urysohn subsequence principle, it suffices to show that for every pair of sequences $(\mu_n)_{n \in \N}$ and $(\nu_n)_{n \in \N}$ in $\Sigma_{m,R}$, letting $\mu = \lim_{n \to \mathcal{U}} \mu_n$ and $\nu = \lim_{n \to \mathcal{U}} \nu_n$, we have $d_W^{(2)}(\mu,\nu) \leq \lim_{n \to \mathcal{U}} d_W^{(2)}(\mu_n,\nu_n)$.  Let $(\cA_n,X_n,Y_n)$ be an optimal couplings of $\mu_n$ and $\nu_n$.  Let $(\cB,X)$ and $(\cC,Y)$ be the GNS realizations of $\mu$ and $\nu$.  By Lemma \ref{lem:ultraproductlaws}, there exist tracial $\mathrm{W}^*$-embeddings $\phi: \cB \to \prod_{n \to \mathcal{U}} \cA_n$ and $\psi: \cC\to \prod_{n \to \mathcal{U}} \cA_n$ such that $\phi(X) = [X_n]_{n \in \N}$ and $\psi(Y) = [Y_n]_{n \in \N}$.  Then
		\[
		d_W^{(2)}(\mu,\nu) \leq \norm{\phi(X) - \psi(Y)}_{L^2(\prod_{n \to \mathcal{U}} \cA_n)_{\sa}^m} = \lim_{n \to \mathcal{U}} \norm{X_n - Y_n}_{L^2(\cA_n)} = \lim_{n \to \mathcal{U}} d_W^{(2)}(\mu_n,\nu_n).  \qedhere
		\]
	\end{proof}
	
	We will use Lemmas \ref{lem:ultraproductlaws} and \ref{lem:ultraproductlaws2} to characterize when the Wasserstein and weak-$*$ topologies agree at a point in $\Sigma_{m,R}$ in terms of a certain stability property.  To fix terminology, if $S$ is a set and $\mathscr{T}_1$ and $\mathscr{T}_2$ are two topologies on $S$, we say that $\mathscr{T}_1$ and $\mathscr{T}_2$ \emph{agree at $x \in S$} if every $\mathscr{T}_1$-neighborhood of $x$ is contained in a $\mathscr{T}_2$-neighborhood of $x$ and vice versa.  If the topologies are metrizable, this is equivalent to saying that a sequence $x_n$ converges to $x$ with respect to $\mathscr{T}_1$ if and only if it converges to $x$ with respect $\mathscr{T}_2$.  Furthermore, if $\mathcal{U}$ is a given non-principal ultrafilter on $\N$, then agreement of the two topologies at $x$ is equivalent to the claim that $\lim_{n \to \mathcal{U}} x_n = x$ with respect to $\mathscr{T}_1$ if and only if $\lim_{n \to \mathcal{U}} x_n = x$ with respect to $\mathscr{T}_2$.
	
	\begin{definition}[$\operatorname{FM}$-lifting]
		Let $\cA$ be a tracial $\mathrm{W}^*$-algebra with separable predual, and let $\mathcal{U}$ be a free ultrafilter on $\N$.  If $\cA_n$ is a sequence of tracial $\mathrm{W}^*$-algebras and $\phi: \cA \to \prod_{n \to \mathcal{U}} \cA_n$ is a tracial $\mathrm{W}^*$-embedding, then an \emph{$\operatorname{FM}$-lifing} of $\phi$ is a sequence $(\Phi_n)_{n \in \N}$, where $\Phi_n \in \operatorname{FM}(\cA,\cA_n)$, such that $\phi(Z) = [\Phi_n(Z)]_{n \in \N}$ for all $Z \in L^\infty(\cA)$.
	\end{definition}
	
	Note that the sequence $\Phi_n$ in Lemma \ref{lem:ultraproductWasserstein} (2) is an $\operatorname{FM}$-lifting of $\phi$.  In other words, Lemma \ref{lem:ultraproductWasserstein} describes convergence in Wasserstein distance in terms of ultraproduct embeddings that have $\operatorname{FM}$-liftings.
	
	\begin{definition}[$\operatorname{FM}$-stability]
		We say that $\cA$ is \emph{$\operatorname{FM}$-stable} if every tracial $\mathrm{W}^*$-embedding $\phi: \cA \to \prod_{n \to \mathcal{U}} \cA_n$ into the ultraproduct of any sequence of tracial $\mathrm{W}^*$-algebras $\cA_n$ has an $\operatorname{FM}$-lifting.
	\end{definition}
	
	Our notion of $\operatorname{FM}$-stability is analogous and closely related to the notions of tracial stability and $\operatorname{UCP}$-stability studied in \cite{AKE2021,HaSh2018}.  Analogously to \cite[Remark 2.2]{AKE2021}, the definition of $\operatorname{FM}$-stability can be restated as an approximation property without reference to ultraproducts.  This implies in particular that the definition is independent of the choice of non-principal ultrafilter $\mathcal{U}$ (hence it amounts to the same thing whether require the lifting condition for a particular non-principal ultrafilter or for all non-principal ultrafilters).
	
	\begin{proposition} \label{prop:twotopologies}
		Let $\mu \in \Sigma_{m,R}$ and let $(\cA,X)$ be the GNS realization of $\mu$ as in Proposition \ref{prop:GNS}.  Then the following are equivalent:
		\begin{enumerate}[(1)]
			\item The weak-$*$ and Wasserstein topologies on $\Sigma_{m,R}$ agree at $\mu$.
			\item $\cA$ is $\operatorname{FM}$-stable.
		\end{enumerate}
	\end{proposition}
	
	\begin{proof}
		(1) $\implies$ (2).  Let $\mathcal{U} \in \beta \N \setminus \N$.  Assume that the weak-$*$ and Wasserstein topologies agree at $\mu$.  Let $(\cA_n)_{n \in \N}$ be a sequence of tracial $\mathrm{W}^*$-algebras, and let $\phi: \cA \to \prod_{n \to \mathcal{U}} \cA_n$ be a tracial $\mathrm{W}^*$-embedding.  Express $\phi(X)$ as $[X_n]_{n \in \N}$ where $X_n \in L^2(\cA_n)_{\sa}^m$ and $\sup_n \norm{X_n}_{L^\infty(\cA_n)_{\sa}^m} < \infty$.  Arguing with functional calculus as in Lemma \ref{lem:ultraproductlaws2}, we can arrange that $\norm{X_n}_{L^\infty(\cA_n)_{\sa}^m} \leq R$.  By Lemma \ref{lem:ultraproductlaws}, we have $\lambda_{X_n} \to \lambda_X$ in the weak-$*$ topology on $\Sigma_{m,R}$.  Hence, by hypothesis $\lambda_{X_n} \to \lambda_X$ in the Wasserstein distance as $n \to \mathcal{U}$.  By Lemma \ref{lem:ultraproductWasserstein}, this implies that $\phi$ has an $\operatorname{FM}$-lifting.
		
		(2) $\implies$ (1).  Conversely, suppose that $\cA$ is $\operatorname{FM}$-stable.  To show that the weak-$*$ and Wasserstein topologies on $\Sigma_{m,R}$ agree at $\mu$, using the Urysohn subsequence principle, it suffices to show that if $(\mu_n)_{n \in \N}$ is a sequence such that $\mu_n \to \mu$ weak-$*$ as $n \to \mathcal{U}$, then $d_W^{(2)}(\mu_n,\mu) \to 0$ as $n \to \mathcal{U}$.  Let $(\cA_n,X_n)$ be the GNS-realization of $\mu_n$.  By Lemma \ref{lem:ultraproductlaws}, the tuple $[X_n]_{n \in \N}$ in $\prod_{n \to \mathcal{U}} \cA_n$ has the same law as $X$, and therefore, there exists a tracial $\mathrm{W}^*$-embedding $\phi: \cA \to \prod_{n \to \mathcal{U}} \cA_n$ with $\phi(X) = [X_n]_{n \in \N}$.  By $\operatorname{FM}$-stability of $\cA$, there exist factorizable completely positive maps $\Phi_n: \cA \to \cA_n$ such that $\phi(Z) = [\Phi_n(Z)]_{n \in \N}$ for all $Z \in L^\infty(\cA)$.  Hence, by Lemma \ref{lem:ultraproductWasserstein}, $\lim_{n \to \mathcal{U}} \mu_n = \mu$ in Wasserstein distance.
	\end{proof}
	
	Next, we will show using the work of Connes \cite{Connes1976} that if the weak-$*$ and Wasserstein topologies agree at $\mu$ and the corresponding tracial $\mathrm{W}^*$-algebra $\cA$ is Connes-embeddable, then in fact $\cA$ is approximately finite-dimensional.  We recall the following theorem of Connes \cite{Connes1976} that shows that approximate finite-dimensionality is equivalent to semi-discreteness for tracial $\mathrm{W}^*$-algebras (and these are also equivalent, famously, to the two other conditions of injectivity and amenability); related proofs can also be found in \cite{Popa1986injective}, \cite[\S XIV]{TakesakiIII}, \cite[\S 6.2, 6.3, 9.3]{BrownOzawa2008}, \cite[\S 11]{ADP}.
	
	\begin{theorem}[{Connes \cite{Connes1976}}]
		Let $\cA = (A,\tau)$ be a tracial $\mathrm{W}^*$-algebra with separable predual.  The following are equivalent:
		\begin{enumerate}[(1)]
			\item $\cA$ is \emph{approximately finite-dimensional (AFD)}, that is, there exists a sequence $(A_k)_{k\in \N}$ of finite-dimensional subalgebras with $A_k \subseteq A_{k+1}$ such that $\bigcup_{k \in \N} A_k$ is dense in $A$ with respect to $\norm{\cdot}_{L^2(\cA)}$.
			\item $\cA$ is \emph{semi-discrete}, that is, there exists nets $(\Phi_\alpha)_{\alpha \in I}$ and $(\Psi_\alpha)_{\alpha \in I}$ of completely positive maps $\Phi_\alpha: \cA \to M_{n(\alpha)}(\C)$ and $\Psi_{\alpha}: M_{n(\alpha)}(\C) \to \cA$ such that $\Psi_\alpha \circ \Phi_\alpha(Z) \to Z$ in the weak-$*$ topology for every $Z \in L^\infty(\cA)$.
		\end{enumerate}
	\end{theorem}
	
	We recall a few more results about AFD algebras, which are well-known in operator algebras.  We recall that a \emph{$\mathrm{II}_1$-factor} is an infinite-dimensional tracial von Neumann algebra with trivial center.
	
	% \cite[Lemma 2.4.8]{BrownOzawa2008}
	
	\begin{lemma} \label{lem:folklore} ~
		\begin{enumerate}[(1)]
			\item Let $\cA$ be an AFD tracial $\mathrm{W}^*$-algebra, let $(\cB_n)_{n \in \N}$ be $\mathrm{II}_1$-factors, and let $\mathcal{U}$ be a free ultrafilter on $\N$.  If $\phi$ and $\psi$ are two embeddings of $\cA$ into $\prod_{n \to \mathcal{U}} \cB_n$, then there exists a unitary $U \in \prod_{n \to \mathcal{U}} \cB_n$ such that $U \phi(Z) U^* = \psi(Z)$ for $Z \in L^\infty(\cA)$.  See \cite{Jung2007,AKE2021}.
			\item If $(\cB_n)_{n \in \N}$ are $\mathrm{II}_1$-factors and $U$ is a unitary in $\prod_{n \to \mathcal{U}} \cB_n$, then there exist unitaries $U_n \in L^\infty(\cB_n)$ such that $U = [U_n]_{n \in \N}$.\footnote{Every unitary $u$ in a tracial $\mathrm{W}^*$-algebra can be expressed as $e^{ix}$ for some self-adjoint $x$ using Borel functional calculus (Theorem \ref{thm:affiliated} (3)).  Suppose $U = e^{iX}$ is unitary in $\prod_{n \to \mathcal{U}}$.  Arguing as in the proof of Lemma \ref{lem:ultraproductlaws2}, $X$ can be expressed as $[X_n]_{n \in \N}$ where $X_n \in L^\infty(\cB_n)_{\sa}$, and we have $[e^{iX_n}]_{n \in \N} = U$.} % (The analogous statement for projections is given in \cite[Lemma 5.4.2]{ADP})
		\end{enumerate}
	\end{lemma}
	
	\begin{corollary} \label{cor:FMstable}
		Let $\cA$ be an AFD tracial $\mathrm{W}^*$-algebra.  Then $\cA$ is $\operatorname{FM}$-stable. 
	\end{corollary}
	
	\begin{proof}
		If $\cA = \C$, then the conclusion is immediate, so assume that $\cA \neq \C$.  Let $\phi: \cA \to \prod_{n \to \mathcal{U}} \cA_n$ be a tracial $\mathrm{W}^*$-embedding.  Let $\cB$ be the tracial free product $\cA * \cA_n * L^\infty[0,1]$ (where $L^\infty[0,1]$ has the trace coming from Lebesgue measure).  Then $\cB$ is a $\mathrm{II}_1$ factor by \cite[Theorem 3.7]{Ueda2011} since $\cA \neq \C$ and $L^\infty[0,1]$ is diffuse.  For each, $n$, there is a tracial $\mathrm{W}^*$-embedding $\iota_n: \cA_n \to \cB_n$.  Let $\iota$ be the induced map
		\[
		\iota: \prod_{n \to \mathcal{U}} \cA_n \to \prod_{n \to \mathcal{U}} \cB_n.
		\]
		By construct, there also exists a tracial $\mathrm{W}^*$-embedding $\psi_n: \cA \to \cB_n$.  This sequence produces a tracial $\mathrm{W}^*$-embedding $\psi: \cA \to \prod_{n \to \mathcal{U}} \cB_n$.  By Lemma \ref{lem:folklore}, there exists a unitary $U_n \in L^\infty(\cB_n)$ such that, letting $U = [U_n]_{n \in \N}$, we have $U \iota \circ \phi(Z) U^* = \psi(Z)$ for $Z \in L^\infty(\cA)$.
		
		Let $\Phi_n: \cA \to \cA_n$ be given by $\Phi_n(Z) = \iota_n^*[U_n^* \psi_n(Z) U_n]$.  As observed in the proof of Proposition \ref{prop:twotopologies}, ultraproducts respect conditional expectations and therefore for $Z \in \cA$, we have
		\[
		[\Phi_n(Z)]_{n \in \N} = [\iota_n^*[U_n^* \psi_n(Z) U_n]]_{n \in \N} = \iota^* [U_n^* \psi_n(Z) U_n]_{n \in \N} = \iota^*(U^* \psi(Z) U) = \iota^* \iota \phi(Z) = \phi(Z).
		\]
		Thus, $\Phi_n$ is the desired lifting of $\phi$ to a sequence of factorizable maps.
	\end{proof}
	
	\begin{remark}
		In fact, \cite[Theorem 2.6]{AKE2021} implies the converse of Corollary \ref{cor:FMstable}:  If $\cA$ is Connes-embeddable and $\operatorname{FM}$-stable, then $\cA$ is AFD.  The same statement is implied by the next proposition provided that $\cA$ is finitely generated.
	\end{remark}
	
	\begin{proposition} \label{prop:twotopologies2}
		Let $\mu$ be in the weak-$*$ closure of $\Sigma_{m,R}^{\fin}$, and let $(\cA,X)$ be the GNS realization of $\mu$.  The following are equivalent:
		\begin{enumerate}[(1)]
			\item $\cA$ is approximately finite-dimensional.
			\item $\cA$ is $\operatorname{FM}$-stable.
			\item The weak-$*$ and Wasserstein topologies agree at $\mu$.
			\item $\mu$ is in the Wasserstein closure of $\Sigma_{m,R}^{\fin}$.
		\end{enumerate}
	\end{proposition}
	
	\begin{proof}
		(1) $\implies$ (2) by Corollary \ref{cor:FMstable}.
		
		(2) $\implies$ (3) by Proposition \ref{prop:twotopologies}.
		
		(3) $\implies$ (4)  Since two topologies agree at $\mu$ and $\mu$ is in the weak-$*$ closure of $\Sigma_{m,R}^{\fin}$, it follows that $\mu$ is in the Wasserstein closure of $\Sigma_{m,R}^{\fin}$.
		
		(4) $\implies$ (1).  Assume that (4) holds and we will show that $\cA$ is semi-discrete, hence approximately finite-dimensional by Connes' theorem.  Fix a free ultrafilter $\mathcal{U}$ on $\N$.  Let $\mu_n$ be a sequence in $\Sigma_{n,R}^{\fin}$ such that $\lim_{n \to \mathcal{U}} d_W^{(2)}(\mu_n,\mu) = 0$.  Let $(\cA_n,X_n,Y_n)$ be an optimal coupling of $\mu$ and $\mu_n$.  Since $\mathrm{W}^*(X_n) \cong \mathrm{W}^*(X) = \cA$, we can assume without loss of generality that $\cA \subseteq \cA_n$ and $X_n = X$.  Let $\Phi_n: \cA = \mathrm{W}^*(X) \to \mathrm{W}^*(Y_n)$ be the associated factorizable map.  Since $\mathrm{W}^*(Y_n)$ is finite-dimensional, if we can show that $\Phi_n^* \Phi_n(Z) \to Z$ in $L^2(\cA)_{\sa}^m$ as $n \to \mathcal{U}$ for every $Z \in \cA$, that will imply semi-discreteness of $\cA$ and finish the argument.
		
		The convergence of $\Phi_n^* \Phi_n(Z)$ follows by a similar argument to Proposition \ref{prop:twotopologies}.  Let $\cB_n$ be the free product of two copies of $\cA_n$ with amalgamation over $\mathrm{W}^*(Y_n)$ and let $\pi_n$ and $\rho_n$ be the two inclusions of $\cA$ into the first and second copies of $\cA_n$.  Then $\Phi_n^* \Phi_n = \pi_n^* \rho_n$.  Now $\pi_n$ and $\rho_n$ induce maps
		\[
		\pi, \rho: \cA \to \prod_{n \to \mathcal{U}} \cB_n.
		\]
		Moreover, $\norm{\pi_n(X) - \rho_n(X)}_{L^2(\cB_n)_{\sa}^m} \leq 2 \norm{X - Y_n}_{L^2(\cA_n)_{\sa}^m} \to 0$, and therefore, $\pi(X) = \rho(X)$, so $\pi = \rho$ on all of $L^\infty(\cA)$.  This implies that $\pi^* \rho(Z) = Z$ for $Z \in L^\infty(\cA)$, hence $\lim_{n \to \mathcal{U}} \norm{\pi_n^*\rho_n(Z) - Z}_{L^2(\cA)} = 0$.
	\end{proof}

	\begin{corollary} \label{cor:noncompact}
		For $m > 1$ and $R > 0$, $\Sigma_{m,R}$ is not compact with respect to the Wasserstein topology.
	\end{corollary}
	
	\begin{proof}
		The identity map from $\Sigma_{m,R}$ with the Wasserstein topology to $\Sigma_{m,R}$ with the weak-$*$ topology is a continuous bijection.  If the domain were compact, then it would be a homeomorphism.  The previous proposition would then imply that every $\mu \in \Sigma_{m,R}$ that generates a Connes-embeddable tracial $\mathrm{W}^*$-algebra would in fact generate an AFD tracial $\mathrm{W}^*$-algebra.  However, there are many finitely generated and Connes-embeddable tracial $\mathrm{W}^*$-algebras that are not AFD.
	\end{proof}
	
	Another consequence of Proposition \ref{prop:twotopologies2} is the following:  Let $\Sigma_{m,R}^{\app}$ be the weak-$*$ closure of $\Sigma_{m,R}^{\fin}$; then the laws that generate AFD tracial $\mathrm{W}^*$-algebras are weak-$*$ generic in $\Sigma_{m,R}^{\app}$, in the sense of the Baire category theorem.  This may seem surprising at first because there are many Connes-embedddable tracial $\mathrm{W}^*$-algebras that are not AFD.  However, a closely related model-theoretic statement has already been proved in cite[Theorem 5.1]{Goldbring2021}, namely that $\mathcal{R}$ is the enforceable model of its universal theory.
	
	\begin{corollary} \label{cor:genericity}
	The set of laws $\mu$ that generate an AFD tracial $\mathrm{W}^*$-algebra is a dense $G_\delta$ set in $\Sigma_{m,R}^{\app}$ with respect to the weak-$*$ topology.
	\end{corollary}

	\begin{proof}
	Let $\mathcal{S}$ be the set of such laws.  By definition $S$ is weak-$*$ dense in $\Sigma_{m,R}^{\app}$.  We showed above $\mathcal{S}$ is closed with respect to the Wasserstein distance.  It follows that
	\[
	\mathcal{S} = \bigcap_{k\in\N} \mathcal{V}_k, \text{ where } \mathcal{V}_k := \left\{\mu \in \Sigma_{m,R}^{\app}: d_W^{(2)}(\mu,\nu) < \frac{1}{k} \text{ for some } \nu \in \mathcal{S} \right\}.
	\]
	For each $k$ and each $\nu \in \mathcal{S}$, because the weak-$*$ and Wasserstein topologies agree at $\nu$, there exists a weak-$*$ open set $\mathcal{U}_{k,\nu} \subseteq \Sigma_{m,R}^{\app}$ such that $\nu \in \mathcal{U}_{k,\nu} \subseteq \mathcal{V}_k$.  Let
	\[
	\mathcal{U}_k = \bigcup_{\nu \in \mathcal{S}} \mathcal{U}_{k,\nu}.
	\]
	Then $\mathcal{S} \subseteq \mathcal{U}_k \subseteq \mathcal{V}_k$ and $\mathcal{U}_k$ is weak-$*$ open.  It follows that $\mathcal{S} = \bigcap_{k \in \N} \mathcal{U}_k$ is a $G_\delta$ set in $\Sigma_{m,R}^{\app}$.
	\end{proof}
	
	\subsection{Non-separability of the Wasserstein space} \label{subsec:nonseparability}
	
	We just showed that $\Sigma_{m,R}$ with the Wasserstein distance is not compact for $m > 1$, but in fact we will show that it is not separable using the results of Gromov \cite{Gromov1987}, Olshanskii \cite{Olshanskii1993}, and Ozawa \cite{Ozawa2004}.  We first recall some terminology about groups and their associated $\mathrm{W}^*$-algebras.
	
	Let $\Gamma$ be a group and let $\ell^2(\Gamma)$ be the Hilbert space of square-summable functions on $\Gamma$.  Let $u: \Gamma \to B(\ell^2(\Gamma))$ be the \emph{left regular representation} given by $u(g) \delta_h = \delta_{gh}$, where $\delta_g \in \ell^2(\Gamma)$ is the function which is $1$ at $g$ and zero elsewhere.  The $\mathrm{W}^*$-subalgebra of $B(\ell^2(\Gamma))$ generated by the unitary operators $u(g)$ for $g \in \Gamma$ is called the \emph{group von Neumann algebra of $\Gamma$}.  The map $\tau: L(\Gamma) \to \C$ given by $T \mapsto \ip{\delta_e, T \delta_e}$ is a faithful normal trace on $L(\Gamma)$, so that it is a tracial $\mathrm{W}^*$-algebra.
	
	\begin{definition} \label{def:propertyT}
		A discrete group $\Gamma$ is said to have \emph{property (T)} if there exist generators $g_1$,\dots, $g_m$ and an increasing function $f: [0,\infty) \to [0,\infty)$ with $\lim_{\epsilon \to 0^+} f(\epsilon) = 0$ with the following property:  For every unitary representation $\pi$ of $\Gamma$ on a Hilbert space $H$ and every unit vector $\xi \in H$, if $\max_{j \in [m]} \norm{\pi(g_j) \xi - \xi} < \epsilon$, then there exists $\eta \in H$ such that $\pi(g) \eta = \eta$ for all $g \in \Gamma$ and $\norm{\eta - \xi} < f(\epsilon)$.
	\end{definition}
	
	\begin{theorem}[{Gromov \cite{Gromov1987}, Olshanskii \cite{Olshanskii1993}, and Ozawa \cite[Theorem 1]{Ozawa2004}}] \label{thm:weirdgroups}
		There exists a group $\Gamma$ with property (T) that admits uncountable family $\{\Gamma_\alpha\}_{\alpha \in I}$ of quotient groups that are simple and pairwise non-isomorphic.  (In fact, such a family of quotient groups exists for every group $\Gamma$ that is hyperbolic, torsion-free, and non-cyclic.)
	\end{theorem}
	
	The next lemma will allow us to translate this result into a statement about the space of non-commutative laws.  While the space of non-commutative laws is defined in terms of self-adjoint generators, it is natural in the group setting to consider unitary rather than self-adjoint generators of a tracial $\mathrm{W}^*$-algebra.  However, this issue is easily resolved by taking real and imaginary parts of operators.  More precisely, if $a$ is an operator in a tracial $\mathrm{W}^*$-algebra $\cA$, let $\re(a) = (a + a^*)/2$ and $\im(a) = (a - a^*)/2i$.  Then $\re(a)$ and $\im(a)$ are self-adjoint and $a = \re(a) + i \im (a)$ and $\norm{a}_{L^2(\cA)}^2 = \norm{\re(a)}_{L^2(\cA)}^2 + \norm{\im(a)}_{L^2(\cA)}^2$.
	
	\begin{lemma}
		Let $\Gamma$ be a group with property (T), and let $g_1$, \dots, $g_m \in \Gamma$ and $f: [0,\infty) \to [0,\infty)$ be as in Definition \ref{def:propertyT}.  Let $q_1: \Gamma \to \Gamma_1$ and $q_2: \Gamma \to \Gamma_2$ be quotient group homomorphisms.  For $j = 1,2$, let $\pi_j: \Gamma \to L(\Gamma_j)$ be the quotient map $q_j$ composed with the left regular representation of $\Gamma_j$ and let
		\[
		X_j = \bigl(\re(\pi_j(g_1)), \im(\pi_j(g_1)), \dots, \re(\pi_j(g_m)), \im(\pi_j(g))\bigr) \in L(\Gamma_j)_{\sa}^{2m}.
		\]
		If $f(d_W^{(2)}(\lambda_{X_1},\lambda_{X_2})) < 1/2$, then $\ker(q_1) = \ker(q_2)$ and hence $\Gamma_1 = \Gamma_2$.
	\end{lemma}
	
	\begin{proof}
		Let $\cA$ be a tracial $\mathrm{W}^*$-algebra and let $\iota_j: L(\Gamma_j) \to \cA$ be tracial $\mathrm{W}^*$-embeddings such that $\norm{\iota_1(X_1) - \iota_2(X_2)}_{L^2(\cA)_{\sa}^{2m}} = d_W^{(2)}(\lambda_{X_1},\lambda_{X_2})$.  Note that for $j = 1$, \dots, $m$,
		\begin{align*}
			\norm{\iota_1(\pi_1(g_j)) - \iota_2(\pi_2(g_j))}_{L^2(\cA)}^2 &= \norm{\iota_1(\re(\pi_1(g_j))) - \iota_2(\re(\pi_2(g_j)))}_{L^2(\cA)}^2 + \norm{\iota_1(\im(\pi_1(g_j))) - \iota_2(\im(\pi_2(g_j)))}_{L^2(\cA)}^2 \\
			&\leq \norm{\iota_1(X_1) - \iota_2(X_2)}_{L^2(\cA)_{\sa}^{2m}}^2.
		\end{align*}
		Let $\pi: \Gamma \to B(L^2(\cA))$ be the map given by $\pi(g) \xi = \iota_1(\pi_1(g)) \xi \iota_2(\pi_2(g^{-1}))$ for $\xi \in L^2(\cA)$; note that this is a unitary representation.  The vector $\widehat{1}$ in $L^2(\cA)$ satisfies
		\begin{align*}
			\norm*{\pi(g_j) \widehat{1} - \widehat{1}}_{L^2(\cA)}
			&= \norm*{\iota_1(\pi_1(g_j)) \widehat{1} - \widehat{1}\iota_2(\pi_2(g_2))}_{L^2(\cA)} \\
			&= \norm{\iota_1(\pi_1(g_j)) - \iota_2(\pi_2(g_2))}_{L^2(\cA)} \\
			&\leq d_W^{(2)}(\lambda_{X_1},\lambda_{X_2}).
		\end{align*}
		Hence, by property (T), there exists some $\eta \in L^2(\cA)$ such that $\norm{\widehat{1} - \eta}_{L^2(\cA)} \leq f(d_W^{(2)}(\lambda_{X_1},\lambda_{X_2}))$ and $\pi(g) \eta = \eta$ for all $g \in \Gamma$.  The latter condition implies that $\iota_1(\pi_1(g)) \eta = \eta \iota_2(\pi_2(g))$ for all $g \in \Gamma$.  Therefore, using the triangle inequality and the fact that $\iota_j(\pi_j(g))$ is unitary,
		\[
		\norm*{\iota_1(\pi_1(g))\widehat{1} - \widehat{1} \iota_2(\pi_2(g))}_{L^2(\cA)} \leq 2 \norm*{\widehat{1}- \eta}_{L^2(\cA)} \leq 2 f(d_W^{(2)}(\lambda_{X_1},\lambda_{X_2})) < 1.
		\]
		Hence,
		\[
		|\tau_{\cA}(\iota_1(\pi_1(g))) - \tau_{\cA}(\iota_2(\pi_2(g)))| \leq \norm*{\iota_1(\pi_1(g)) - \iota_2(\pi_2(g))}_{L^2(\cA)} < 1.
		\]
		Now observe that
		\[
		\tau_{\cA}(\iota_j(\pi_j(g))) = \tau_{L(\Gamma_j)}(\pi_j(g)) = \delta_{\pi_j(g) = 1} = \delta_{g \in \ker(q_j)}.
		\]
		Since $\delta_{g \in \ker(q_j)}$ is either zero or one and $|\delta_{g \in \ker(q_1)} - \delta_{g \in \ker(q_2)}| < 1$, we have $\ker(q_1) = \ker(q_2)$.
	\end{proof}
	
	We can now prove Theorem \ref{thm:notseparable} that shows that for sufficiently large $m$, $\Sigma_{m,1}$ is not separable with respect to $d_W^{(2)}$.  The method is similar to \cite[Proof of Theorem 2]{Ozawa2004}.
	
	\begin{proof}[{Proof of Theorem \ref{thm:notseparable}}]
		First, we show that $\Sigma_{2m,1}$ is not separable for some $m$.  Let $\Gamma$ be a property (T) group with an uncountable family $(\Gamma_\alpha)_{\alpha \in I}$ of non-isomorphic quotients.  Let $\pi_\alpha: \Gamma \to L(\Gamma_\alpha)$ be the quotient map composed with the left regular representation.  Let $g_1$, \dots, $g_m$ and $f: [0,\infty) \to [0,\infty)$ witness property (T).  Let $\epsilon$ be sufficiently small that $f(\epsilon) < 1/2$.  Let
		\[
		X_\alpha = (\re(\pi_\alpha(g_1)), \im(\pi_\alpha(g_1)), \dots, \re(\pi_\alpha(g_m)), \im(\pi_\alpha(g_m))).
		\]
		For $\alpha \neq \beta$ in $I$, since $\Gamma_\alpha$ and $\Gamma_\beta$ are not isomorphic, the lemma implies that $f(d_W^{(2)}(\lambda_{X_\alpha},\lambda_{X_\beta})) \geq 1/2$, and therefore $d_W^{(2)}(\lambda_{X_\alpha},\lambda_{X_\beta}) \geq \epsilon$.  Hence, $\{\lambda_{X_\alpha}\}_{\alpha \in I}$ is an uncountable $\epsilon$-separated set in $\Sigma_{2m,1}$ with respect to the Wasserstein distance.
		
		To prove that $\Sigma_{m,R}$ is not separable for general $m > 1$ and $R > 0$, we first observe that there is a bijection between $\Sigma_{m,R}$ and $\Sigma_{m,R'}$ given by rescaling the non-commutative random variables.  Hence, for each $m$, if we prove non-separability for one value of $R$, then it holds for all values of $R$.  Furthermore, we can define a map $\Sigma_{m,R} \to \Sigma_{m+1,R}$ sending the law of $(X_1,\dots,X_m)$ to the law of $(X_1,\dots,X_m,0)$.  It is straightforward to show that this map is isometric with respect to the Wasserstein distance.  Hence, if $\Sigma_{m,R}$ is not separable, then $\Sigma_{m',R}$ is not separable for $m' \geq m$.  Therefore, to prove the theorem, it suffices to show that for some value of $R$, $\Sigma_{2,R}$ is not separable.
		
		We already know that for some $m$, $\Sigma_{m,1}$ is not separable.  Hence, for some $\epsilon > 0$, there is an uncountable family $(\mu_\alpha)_{\alpha \in I}$ of laws in $\Sigma_{m,1}$ that is $\epsilon$-separated with respect to the Wasserstein distance.  Let $(\cA_\alpha, X_\alpha)$ be the GNS realization of $\mu_\alpha$, where $X_\alpha = (X_{\alpha,1},\dots,X_{\alpha,m})$.  Consider the tracial $\mathrm{W}^*$-algebra $M_m(\cA_\alpha)$ with the trace $\tau_\alpha \otimes \tr_m$, and let $Y_\alpha \in M_m(\cA_\alpha)_{\sa}$ be the diagonal matrix with entries $X_{\alpha,1} + 4$, $X_{\alpha,2} + 8$, \dots, $X_{\alpha,m} + 4m$.  Let $U_\alpha \in M_m(\C) \subseteq M_m(\cA_\alpha)$ be the matrix of an $m$-cycle permutation.  By functional calculus, $U_\alpha$ can be expressed as $e^{iZ_\alpha}$ for some self-adjoint $Z_\alpha \in M_m(\C) \subseteq M_m(\cA_\alpha)$ with $\norm{Z_\alpha}_{L^\infty(M_m(\C))} \leq \pi/2$.  Since $U_\alpha$ is the inclusion into $M_m(\cA_\alpha)$ of an element of $M_m(\C)$ that is independent of $\alpha$, there is in fact a polynomial $p$ such that $U_\alpha = p(Z_\alpha)$, and $Z_\alpha$ and $p$ are independent of $\alpha$.  We claim that $d_W^{(2)}(\lambda_{Y_\alpha,Z_\alpha}, \lambda_{Y_\beta,Z_\beta}) \geq (1/K) d_W^{(2)}(\mu_\alpha,\mu_\beta)$ for some $K > 0$, which will imply that $\Sigma_{2,4m+1}$ is not separable and thus prove the theorem.
		
		To accomplish this, we will express $X_{\alpha,j} \otimes I_m$ in $M_m(\cA_\alpha)$ as a function of $Y_\alpha$ and $Z_\alpha$ (in an explicit way which allows us to estimate Wasserstein distances), using a well-known matrix amplification trick.  We first recall a foundational result that the weak-$*$ topology of a $\mathrm{W}^*$-algebra can be recovered from any faithful representation on a Hilbert space; see e.g.\ \cite[Corollary 1.13.3, Proposition 1.16.2, Theorem 1.16.7]{Sakai1971}.  In particular, $\cA_\alpha$ can be faithfully represented on $H = L^2(\cA)$ and $M_m(\cA_\alpha) = \cA_\alpha \otimes M_m(\C)$ can be faithfully represented on the Hilbert space $H \otimes \C^m = H^{\oplus m}$.  Moreover, all the facts about spectral theory and functional calculus on $B(H)$ and $B(H^{\oplus m})$ can be applied to the operators from $\cA_\alpha$ and $M_n(\cA_\alpha)$.  In particular,
		\[
		\Spec(Y_\alpha) = \bigcup_{j=1}^m (\Spec(X_{\alpha,j}) + 4j) \subseteq \bigcup_{j=1}^m [4j-1,4j+1].
		\]
		Let $\gamma_j$ be the rectangular contour in $\C$ bounding the rectangle $[4j-2,4j+2] \times [-1,1]$, so that $\gamma_j$ is separated from $\Spec(Y_\alpha)$ by a distance of $1$.  Using the Cauchy integral formula and functional calculus,
		\[
		\int_{\gamma_j} (z - 4j)(z - X_{\alpha,k})^{-1}\,dz = \delta_{j,k} X_{\alpha,k}.
		\]
		Hence,
		\[
		\int_{\gamma_j} (z - 4j)(z - Y_\alpha)^{-1}\,dz = X_{\alpha,j} \otimes e_{j,j},
		\]
		where $e_{j,j}$ is the $j$th diagonal matrix unit in $M_m(\C)$.  In particular, $X_{\alpha,j} \otimes e_{j,j} \in \mathrm{W}^*(Y_\alpha)$ and thus $X_{\alpha,j} \otimes e_{k,\ell} = U_\alpha^{k-j}(X_{\alpha,j} \otimes e_{j,j}) U_\alpha^{j-\ell} \in \mathrm{W}^*(Y_\alpha,Z_\alpha)$ for every $k, \ell = 1$, \dots, $m$; this implies that $Y_\alpha$ and $Z_\alpha$ generate $M_m(\cA_\alpha)$.  Moreover,
		\begin{equation} \label{eq:Cauchyintegral}
			X_{\alpha,j} \otimes I_m = \sum_{k=1}^m \int_{\gamma_j} U_\alpha^k (z - 4j)(z - Y_\alpha)^{-1}U_\alpha^{-k}\,dz = \sum_{k=1}^m \int_{\gamma_j} p(Z_\alpha)^k (z - 4j)(z - Y_\alpha)^{-1} \overline{p}(Z_\alpha)^k \,dz.
		\end{equation}
		Let $\alpha \neq \beta$.  Then an optimal coupling of $\lambda_{Y_\alpha,Z_\alpha}$ and $\lambda_{Y_\beta,Z_\beta}$ on the tracial $\mathrm{W}^*$-algebra $\cB$ produces two tracial $\mathrm{W}^*$-embeddings $\iota_\alpha: M_m(\cA_\alpha) \to \cB$ and $\iota_\beta: M_m(\cA_\beta) \to \cB$.  Because the Cauchy integral representation \eqref{eq:Cauchyintegral} can be expressed as a Riemann integral, we have
		\[
		\iota_\alpha(X_{\alpha,j} \otimes I_m) = \sum_{k=1}^m \int_{\gamma_j} p(\iota_\alpha(Z_\alpha))^k (z - 4j)(z - \iota_\alpha(Y_\alpha))^{-1} \overline{p}(\iota_\alpha(Z_\alpha))^k \,dz,
		\]
		and the same holds for $\beta$.  Using the resolvent identity and non-commutative H\"older's inequality,
		\begin{align*}
			\norm{(z - \iota_\alpha(Y_\alpha))^{-1} - (z -\iota_\beta(Y_\beta))^{-1}}_{L^2(\cB)} &\leq \norm{(z - \iota_\alpha(Y_\alpha))^{-1}}_{L^\infty(\cB)} \norm{Y_\alpha - Y_\beta}_{L^2(\cB)} \norm{(z - \iota_\beta(Y_\beta))^{-1}}_{L^\infty(\cB)} \\
			&\leq \norm{Y_\alpha - Y_\beta}_{L^2(\cB)}.
		\end{align*}
		Furthermore, one checks easily that $\norm{p(\iota_\alpha(Z_\alpha)) - p(\iota_\beta(Z_\beta))}_{L^2(\cB)} \leq C_p \norm{\iota_\alpha(Z_\alpha) - \iota_\beta(Z_\beta)}_{L^2(\cB)}$ for some constant $C_p$ (since $\norm{Z_\alpha}_{L^\infty(M_m(\cA_\alpha))}$ is bounded by univeral constant).  By estimating the difference between $p(\iota_\alpha(Z_\alpha))^k (z - 4j)(z - \iota_\alpha(Y_\alpha))^{-1} \overline{p}(\iota_\alpha(Z_\alpha))^k$ and $p(\iota_\beta(Z_\beta))^k (z - 4j)(z - \iota_\beta(Y_\beta))^{-1} \overline{p}(\iota_\beta(Z_\beta))^k$ and applying the triangle inequality for integrals, we obtain for some constant $C_p'$ that
		\[
		\norm{\iota_\alpha(X_{\alpha,j} \otimes I_m) - \iota_\beta(X_{\beta,j} \otimes I_m)}_{L^2(\cB)} \leq C_p' \left( \norm{\iota_\alpha(Y_\alpha) - \iota_\beta(Y_\beta)}_{L^2(\cB)}^2 +  \norm{\iota_\alpha(Z_\alpha) - \iota_\beta(Z_\beta)}_{L^2(\cB)}^2  \right)^{1/2}.
		\]
		Since $(X_{\alpha,1} \otimes I_m, \dots, X_{\alpha,m} \otimes I_m)$ has the same non-commutative law as $X_\alpha$, we obtain
		\[
		\epsilon \leq d_W^{(2)}(\lambda_{X_\alpha},\lambda_{X_\beta}) \leq m^{1/2} C_p' d_W^{(2)}(\lambda_{(Y_\alpha,Z_\alpha)},\lambda_{(Y_\beta,Z_\beta)}).
		\]
		Hence, $\{\lambda_{(Y_\alpha,Z_\alpha)}\}_{\alpha \in I}$ is $\epsilon / (m^{1/2} C_p')$-separated in $\Sigma_{2,4m+1}$, as desired.
	\end{proof}
	
	We remark that a similar non-separability result in the context of model theory for operator algebras was shown in \cite[Proposition 4.2.9]{AGKE2020}.  In the model theoretic context, one often encounters triples $(\Omega,\mathscr{T},d)$ where  $(\Omega,\mathscr{T})$ is a topological space and $d$ is a metric on $\Omega$ that is lower semi-continuous with respect to $\mathscr{T}$ and generates a topology that is at least as strong as $\mathscr{T}$; such a triple $(\Omega,\mathscr{T},d)$ is called a \emph{topometric space} \cite{IBY2008}.  In particular, $\Sigma_{m,R}$ with the weak-$*$ topology and Wasserstein distance is a topometric space by Corollary \ref{cor:twotopologies} and Lemma \ref{lem:LSC}. It was shown in \cite[Proposition 3.20]{IBY2008} that if $(\Omega,\mathscr{T},d)$ is a topometric space and $(\Omega,\mathscr{T})$ is second countable and locally compact, then the density character of $(\Omega,d)$ is either countable or equal to the continuum.  Hence, as a corollary of Theorem \ref{thm:notseparable}, the density character of $(\Sigma_{m,R},d_W^{(2)})$ is the continuum (of course since $\Sigma_{m,R}$ with the weak-$*$ topology is compact and metrizable, it is in particular second countable and locally compact).
	
	\section{Further remarks} \label{sec:remarks}
	
	\subsection{Non-commutative optimal couplings and random matrix theory} \label{subsec:randommatrix}
	
	One of the motivations for our paper was the following question.
	
	\begin{question} \label{q:randommatrix}
		Suppose that $X^{(N)}$, $Y^{(N)}$ are random $m$-tuple of self-adjoint $N \times N$ matrices with probability distributions $\mu^{(N)}$ and $\nu^{(N)}$ respectively.  Let $\mu$, $\nu \in \Sigma_{n,R}$. Suppose that almost surely
		\[
		\limsup_{N \to \infty} \norm{X^{(N)}}_{L^\infty(M_N(\C))_{\sa}^m} < R, \qquad \limsup_{N \to \infty} \norm{Y^{(N)}}_{L^\infty(M_N(\C))_{\sa}^m} < R,
		\lim_{N \to \infty} \lambda_{X^{(N)}} = \mu, \qquad \lim_{N \to \infty} \lambda_{Y^{(N)}} = \nu.
		\]
		Does the classical $L^2$-Wasserstein distance of $\mu^{(N)}$ and $\nu^{(N)}$ (as probability measures on $M_N(\C)_{\sa}^m$) converge to the non-commutative $L^2$-Wasserstein distance of $\mu$ and $\nu$?
	\end{question}
	
	The results of \cite{GS2014,DGS2016} combined with \cite{JLS2021} give a positive answer when $\mu^{(N)}$ is a random matrix model with density proportional to $e^{-N^2V^{(N)}}$ where $V^{(N)}: M_N(\C)_{\sa}^m \to \R$ is a sufficiently regular convex function such as the trace of a non-commutative polynomial, and where $\nu^{(N)}$ has density proportional to $e^{-N^2 \norm{X}_{L^2(M_N(\C))_{\sa}^m}^2/2}$ (Gaussian).  The convexity of $V^{(N)}$ is crucial for all these arguments.  By contrast, the present work shows that Question \ref{q:randommatrix} can have a negative answer due to the obstruction of Connes-embeddability.
	
	\begin{proposition} \label{prop:matrixcounterexample}
		Let $X, Y \in M_n(\C)_{\sa}^m$ be matrix tuples such that an optimal coupling of $\lambda_X$ and $\lambda_Y$ requires a non-Connes-embeddable tracial $\mathrm{W}^*$-algebra as in Corollary \ref{cor:nonConnes}.  Suppose $X^{(N)}$ and $Y^{(N)}$ are random (or even deterministic) elements of $M_N(\C)_{\sa}^m$ that converge in non-commutative law to $X$ and $Y$.  Then the classical Wasserstein distance of the probability distributions of $X^{(N)}$ and $Y^{(N)}$ on $M_N(\C)_{\sa}^m$ (with the $L^2$ norm associated to the normalized trace $\tr_N$) does not converge to $d_W^{(2)}(\lambda_X,\lambda_Y)$.
	\end{proposition}
	
	Before proving the proposition, we make some preliminary observations.  Let $\Sigma_{m,R}^{\app}$ denote the space of Connes-embeddable non-commutative laws in $\Sigma_{m,R}$.  Let $d_{W,\app}^{(2)}$ be the non-commutative Wasserstein distance on $\Sigma_{m,R}^{\app}$ defined using only couplings in Connes-embeddable tracial $\mathrm{W}^*$-algebras.  Since $\Sigma_{m,R}^{\app}$ is the weak-$*$ closure of $\Sigma_{m,R}^{\fin}$, it is weak-$*$ compact, which implies the existence of optimal Connes-embeddable couplings.  Moreover, the same reasoning as in Lemma \ref{lem:LSC} shows that $d_{W,\app}^{(2)}$ is weak-$*$ lower semi-continuous.  Of course, Corollary \ref{cor:nonConnes} shows that $d_{W,\app}^{(2)}$ can be strictly greater than $d_W^{(2)}$ (however, we do not know whether these two metrics generate the same topology on $\Sigma_{m,R}^{\app}$).
	
	\begin{proof}[{Proof of Proposition \ref{prop:matrixcounterexample}}]
		Suppose that $X^{(N)}$ and $Y^{(N)}$ are random variables on the diffuse probability space $(\Omega,P)$.  Let $\mu^{(N)}$ and $\nu^{(N)}$ be the classical probability distributions of $X^{(N)}$ and $Y^{(N)}$ as random variables with values in the vector space $M_N(\C)_{\sa}^m$ equipped with inner product associated to $\tr_N$. Let $\widehat{\mu}^{(N)}$ and $\widehat{\nu}^{(N)}$ be the non-commutative laws of $X^{(N)}$ and $Y^{(N)}$ as elements of the tracial $\mathrm{W}^*$-algebra $L^\infty(\Omega,P;M_N(\C))$ with the trace given by $\mathbb{E} \circ \tr_N$.  A classical coupling of the probability distributions $\mu^{(N)}$ and $\nu^{(N)}$ on the probability space $(\Omega,P)$ can be interpreted as a non-commutative coupling on the tracial $\mathrm{W}^*$-algebra $ (L^\infty(\Omega,P;M_N(\C)), \mathbb{E}\circ \tr_N)$, which is Connes-embeddable.  Therefore,
		\[
		\liminf_{N \to \infty} d_W(\mu^{(N)},\nu^{(N)}) \geq \liminf_{N \to \infty} d_{W,\app}^{(2)}(\widehat{\mu}^{(N)},\widehat{\nu}^{(N)}) \geq d_{W,\app}^{(2)}(\lambda_X,\lambda_Y) > d_{W}^{(2)}(\lambda_X,\lambda_Y).  \qedhere
		\]
	\end{proof}
	
	This problem cannot be removed using free probabilistic regularity conditions (conditions such as finite free entropy, finite free Fisher information and so forth; see the introduction of \cite{CN2019} for context).
	
	\begin{proposition}
		Again, let $X, Y \in M_n(\C)_{\sa}^m$ be as in Corollary \ref{cor:nonConnes}.  Let $S$ be a free semicircular $m$-tuple freely independent of $X$ and $Y$.  Then $X + t^{1/2}S$ and $Y + t^{1/2}S$ have finite free microstate entropy (defined in \cite{VoiculescuFE2}) and finite free Fisher information (defined in \cite{VoiculescuFE5}).  However, $d_{W,\app}^{(2)}(\lambda_{X+t^{1/2}S},\lambda_{Y+t^{1/2}}) > d_W^{(2)}(\lambda_{X+t^{1/2}S},\lambda_{Y+t^{1/2}})$ for sufficiently small $t > 0$.  Hence, as in Proposition \ref{prop:matrixcounterexample}, there do not exist random matrix approximations for $\lambda_{X+t^{1/2}S}$ and $\lambda_{X+t^{1/2}S}$ whose classical Wasserstein distance converges to $d_W^{(2)}(\lambda_{X+t^{1/2}S},\lambda_{Y+t^{1/2}})$.
	\end{proposition}
	
	\begin{proof}
		By \cite[Theorem 3.9]{Voiculescu1998}, $X + t^{1/2}S$ and $Y + t^{1/2}S$ have finite free microstate entropy, and by \cite[Corollary 6.14]{VoiculescuFE5}, they have finite free Fisher information. The free product of $M_N(\C)$ and $\mathrm{W}^*(S)$ is Connes-embeddable by \cite[Proposition 3.3]{Voiculescu1998}.  Hence,
		\[
		d_W^{(2)}(\lambda_X,\lambda_{X+t^{1/2}S}) \leq d_{W,\app}^{(2)}(\lambda_X,\lambda_{X+t^{1/2}S}) \leq (mt)^{1/2},
		\]
		and the same holds with $X$ replaced by $Y$.  Thus, using the triangle inequality, $d_{W}^{(2)}(\lambda_{X+t^{1/2}S},\lambda_{Y+t^{1/2}S}) < d_{W,\app}^{(2)}(\lambda_{X+t^{1/2}S},\lambda_{Y+t^{1/2}S})$ for sufficiently small $t > 0$, since this holds at $t = 0$.  The same argument as in Proposition \ref{prop:matrixcounterexample} rules out the possibility of the classical Wasserstein distance for random matrix models converging to $d_W^{(2)}(\lambda_{X+t^{1/2}S},\lambda_{Y+t^{1/2}S})$.
	\end{proof}
	
	Thus, at the very least, Question \ref{q:randommatrix} needs to be reformulated using the Connes-embeddable version of the non-commutative Wasserstein distance.  Even with such a modification, our results illustrate why this question is so difficult.\footnote{Questions of large-$N$ convergence in mean field games are also extremely subtle and require regularity of the putative model for the large-$N$ limit.}  Indeed, in light of \S \ref{subsec:twotopologies}, random matrix models cannot converge in Wasserstein distance to the limiting non-commutative law unless that limiting law produces an approximately finite-dimensional tracial $\mathrm{W}^*$-algebra.  However, ``good behavior'' in random matrix theory and free probability often entails generating a tracial $\mathrm{W}^*$-algebra that is ``similar to'' a free group von Neumann algebra, which is far from being approximately finite-dimensional (see e.g.\ \cite[\S XIV.3]{TakesakiIII}).  The random matrix question suggests a more general question in the framework of non-commutative optimal couplings.
	
	\begin{question} \label{q:convergence}
		Suppose that $\mu_n$, $\nu_n \in \Sigma_{m,R}$ and $\mu_n \to \mu$ and $\nu_n \to \nu$ weak-$*$.  Under what conditions does $d_W^{(2)}(\mu_n,\nu_n) \to d_W^{(2)}(\mu,\nu)$?
	\end{question}
	
	Monge-Kantorovich duality provides one avenue to attack this question.  Indeed, suppose that $(f_n,g_n)$ are admissible pairs of $E$-convex functions minimizing $\mu_n(f_n) + \nu_n(g_n)$.  Suppose that $(f,g)$ is an admissible pair minimizing $\mu(f) + \nu(g)$.  To give a positive answer to Question \ref{q:convergence}, it suffices to show that $\mu_n(f_n) + \nu_n(g_n) \to \mu(f) + \nu(g)$.  Suppose that we somehow show that $f_n \to f$ and $g_n \to g$ uniformly on each operator norm ball, so that $\mu_n(f_n) - \mu_n(f) \to 0$ and $\nu(g_n) - \nu(g) \to 0$.
	
	Then it remains to show that $\mu_n(f) \to \mu(f)$ and $\nu_n(g) \to \nu(g)$.  If $f$ and $g$ take finite values everywhere, then for each $\cA \in \mathbb{W}$, $f^{\cA}$ and $g^{\cA}$ will define continuous functions on $L^2(\cA)_{\sa}^m$, and in particular, $\lambda \mapsto \lambda(f)$ and $\lambda \mapsto \lambda(g)$ are continuous with respect to Wasserstein distance.  However, we only assumed weak-$*$ convergence of $\mu_n \to \mu$ and $\nu_n \to \nu$.  Thus, in order to obtain the convergence of the Wasserstein distance, we would want the stronger condition that $f$ and $g$ are continuous with respect to convergence in law, that is, $\lambda \mapsto \lambda(f)$ and $\lambda \mapsto \lambda(g)$ are weak-$*$ continuous on $\Sigma_{m,R}$ for each $R > 0$.
	
	The examples of Monge-Kantorovich duality in \cite[Lemma 9.10, Remark 9.11]{JLS2021} use functions that are continuous with respect to the weak-$*$ topology on $\Sigma_{m,R}$.  However, we doubt that the optimizers $(f,g)$ in the Monge-Kantorovich duality can always be chosen to be weak-$*$ continuous.  Nonetheless, it is worth investigating in future research how $E$-convex functions and Legendre transforms behave with respect to convergence in law.
	
	\subsection{Bimodule couplings and $\operatorname{UCPT}$-convex functions} \label{subsec:bimodule}
	
	Another operator-algebraic analog of the idea of coupling arises from bimodules over von Neumann algebras, which have been very important in many areas of von Neumann algebras.  For further background, see \cite[Appendix F]{BrownOzawa2008} and \cite[\S 13]{ADP}.
	
	\begin{definition}
		If $A$ and $B$ are $\mathrm{W}^*$-algebras, then a \emph{Hilbert $A$-$B$-bimodule} is a Hilbert space $H$ with an $A$-$B$-bimodule structure, such that the associated maps $A \to B(H)$ and $B \to B(H)$ are weak-$*$ continuous.  Given tracial $\mathrm{W}^*$-algebras $\cA = (A,\tau)$ and $\cB = (B,\sigma)$ and a $A$-$B$-bimodule $H$, we say that a vector $\xi \in H$ is \emph{bitracial} if $\ip{\xi, a\xi} = \tau(a)$ for $a \in A$ and $\ip{\xi, \xi b} = \sigma(b)$ for $b \in B$.
	\end{definition}
	
	For example, suppose that there are tracial $\mathrm{W}^*$-embeddings $\iota_1: \cA \to \cC$ and $\iota_2: \cB \to \cC$.  Then $L^2(\cC)$ is a Hilbert $L^\infty(\cA)$-$L^\infty(\cB)$-bimodule and $\xi = \widehat{1}\in L^2(\cC)$ is a bitracial vector.  Thus, bimodules with bitracial vectors are a generalization of pair of tracial $\mathrm{W}^*$-embeddings.  In the case of a pair of embeddings $\iota_1$ and $\iota_2$, there is an associated factorizable map $\iota_2^* \iota_1: \cA \to \cB$.  In a similar way, general $L^\infty(\cA)$-$L^\infty(\cB)$-bimodules with bitracial vectors correspond to general $\operatorname{UCPT}$-maps.
	
	\begin{lemma}[{See \cite[\S 13.1.2]{ADP}}]
		Let $\cA$, $\cB$ be tracial $\mathrm{W}^*$-algebras.  If $H$ is a Hilbert $L^\infty(\cA)$-$L^\infty(\cB)$-bimodule and $\xi \in H$ is a bitracial vectors, then there exists a unique $\Phi \in \operatorname{UCPT}(\cA,\cB)$ such that $\ip{\xi,a\xi b} = \tau_{\cB}(\Phi(a)b)$ for all $a \in L^\infty(\cA)$ and $b \in L^\infty(\cB)$.  Conversely, $\Phi \in \operatorname{UCPT}(\cA,\cB)$, there exists a Hilbert $L^\infty(\cA)$-$L^\infty(\cB)$-bimodule $H$ and a bitracial vector $\xi$ satisfying $\ip{\xi,a\xi b} = \tau_{\cB}(\Phi(a)b)$.  If we further demand that $H$ is generated by $\xi$ as a Hilbert $L^\infty(\cA)$-$L^\infty(\cB)$-bimodule, then the pair $(H,\xi)$ is unique up to isomorphism.
	\end{lemma}
	
	The bimodules and their associated $\operatorname{UCPT}$-maps lead to an alternative notion of couplings for non-commutative random variables.
	
	\begin{definition}
		Let $\mu$ and $\nu \in \Sigma_m$ be non-commutative laws, and let $(\cA,X)$ and $(\cB,Y)$ be the GNS realizations of $\mu$ and $\nu$ respectively.  A \emph{bimodule coupling} of $\mu$ and $\nu$ is a Hilbert $\cA$-$\cB$-bimodule $H$ together with a bitracial vector $\xi$.  We define $C_{\bim}(\mu,\nu)$ to be the supremum of $\sum_{j=1}^m \ip{\xi, X_j\xi Y_j}$ over all bimodule couplings of $\mu$ and $\nu$, or equivalently,
		\[
		C_{\bim}(\mu,\nu) = \sup_{\Phi \in \operatorname{UCPT}(\cA,\cB)} \ip{\Phi(X),Y}_{L^2(\cB)_{\sa}^m}.
		\]
		We then define $d_{\bim}(\mu,\nu)$ by
		\[
		d_{\bim}(\mu,\nu)^2 = \sum_{j=1}^m \mu(x_j^2) + \sum_{j=1}^m \nu(x_j^2) - 2 C_{\bim}(\mu,\nu).
		\]
	\end{definition}
	
	We remark that $d_{\bim}(\mu,\nu)$ is the infimum of $\norm{X\xi - \xi Y}$ in $H^m$ over Hilbert $L^\infty(\cA)$-$L^\infty(\cB)$-bimodules with bitracial vectors (this follows from \eqref{eq:bimodulecomputation} below).  Moreover, the existence of optimal bimodule couplings can be deduced from the compactness of $\operatorname{UCPT}(\cA,\cB)$ in the pointwise weak-$*$ topology.  The properties of $C_{\bim}$ and $d_{\bim}$ are quite similar to those of $C$ and $d_W^{(2)}$ only with factorizable maps replaced by general $\operatorname{UCPT}$ maps, but we will see in Corollary \ref{cor:nonfactorizable} that they do not agree in general.  But first, for completeness, we give proofs of some of the basic properties with the aid of the following lemma.
	
	\begin{lemma} \label{lem:UCPTmapestimate}
		Let $\cA$ and $\cB$ be tracial $\mathrm{W}^*$-functions and $\Phi \in \operatorname{UCPT}(\cA,\cB)$.  Let $X \in L^\infty(\cA)_{\sa}^m$ and $Y \in L^\infty(\cB)_{\sa}^m$ with $\norm{X}_{L^\infty(\cA)_{\sa}^m} \leq R$ and $\norm{Y}_{L^\infty(\cA)_{\sa}^m} \leq R$.  Then for $i_1$, \dots, $i_\ell \in \{1,\dots,m\}$, we have
		\begin{equation} \label{eq:bimoduleestimate}
		\norm{\Phi(X_{i_1} \dots X_{i_\ell}) - Y_{i_1} \dots Y_{i_\ell}}_{L^2(\cB)} \leq \ell R^{\ell - 1} \left( \norm{X}_{L^2(\cA)_{\sa}^m}^2 - 2 \ip{\Phi(X),Y}_{L^2(\cB)_{\sa}^m} + \norm{Y}_{L^2(\cB)_{\sa}^m} \right)^{1/2}.
		\end{equation}
	\end{lemma}
	
	\begin{proof}
		Let $H$ be an $L^\infty(\cA)$-$L^\infty(\cB)$ bimodule with a bitracial vector $\xi$ such that $\ip{\Phi(Z),W}_{L^2(\cB)} = \ip{\xi, Z \xi W}_{L^2(\cB)}$ for all $Z \in L^\infty(\cA)$ and $W \in L^\infty(\cB)$.  Direct computation shows that
		\begin{align} \label{eq:bimodulecomputation}
			\norm{\Phi(Z) - W}_{L^2(\cB)}^2 &= \norm{\Phi(Z)}_{L^2(\cB)}^2 - 2 \re \ip{\Phi(Z),W}_{L^2(\cB)} + \norm{W}_{L^2(\cB)}^2 \\
			&\leq \norm{Z}_{L^2(\cB)}^2 - 2 \re \ip{\Phi(Z),W}_{L^2(\cB)} + \norm{W}_{L^2(\cB)}^2 \nonumber \\
			&= \norm{Z \xi - \xi W}^2. \nonumber
		\end{align}
		This implies that
		\begin{align*}
			\norm{\Phi(X_{i_1} \dots X_{i_\ell}) - Y_{i_1} \dots Y_{i_\ell}}_{L^2(\cB)} &\leq \norm{X_{i_1} \dots X_{i_\ell} \xi - \xi Y_{i_1} \dots Y_{i_\ell}} \\
			& \leq \sum_{k=1}^\ell \norm{X_{i_1} \dots X_{i_k} \xi Y_{i_{k+1}} \dots Y_{i_\ell} - X_{i_1} \dots X_{i_{k-1}} \xi Y_{i_k} \dots Y_{i_\ell}} \\
			& \leq \sum_{k=1}^\ell \norm{X_{i_1} \dots X_{i_{k-1}}}_{L^\infty(\cA)} \norm{X_{i_k}\xi - \xi Y_{i_k}} \norm{Y_{i_{k+1}} \dots Y_{i_\ell}}_{L^\infty(\cB)} \\
			& \leq \ell R^{\ell-1} \norm{X\xi - \xi Y} \\
			& = \ell R^{\ell-1} \left( \norm{X}_{L^2(\cA)_{\sa}^m}^2 - 2 \ip{\Phi(X),Y}_{L^2(\cB)_{\sa}^m} + \norm{Y}_{L^2(\cB)_{\sa}^m} \right)^{1/2}.  \qedhere
		\end{align*}
	\end{proof}
	
	\begin{proposition}
		$(\Sigma_{m,R},d_{\bim})$ is a complete metric space.  If $\lambda$, $\mu \in \Sigma_{m,R}$, then
		\begin{equation} \label{eq:bimodulelawestimate}
			|\lambda(x_{i_1} \dots x_{i_\ell}) - \mu(x_{i_1} \dots x_{i_\ell})| \leq \ell R^{\ell - 1} d_{\bim}(\lambda,\mu) \leq \ell R^{\ell - 1} d_W^{(2)}(\lambda,\mu),
		\end{equation}
		and in particular, the topology generated by $d_{\bim}$ refines the weak-$*$ topology, and the topology generated by $d_W^{(2)}$ refines the topology generated by $d_{\bim}$.  Moreover, $d_{\bim}$ is lower semi-continuous on $\Sigma_{m,R} \times \Sigma_{m,R}$ with respect to the weak-$*$ topology.
	\end{proposition}
	
	\begin{proof}
		In the following, let $\lambda$, $\mu$, and $\nu \in \Sigma_{m,R}$, and let $(\cA,X)$, $(\cB,Y)$, and $(\cC,Z)$ be their respective GNS realizations. 
		
		First, we prove \eqref{eq:bimodulelawestimate}.  If $\Phi \in \operatorname{UCPT}(\cA,\cB)$, then using \eqref{eq:bimoduleestimate},
		\begin{align*}
			|\lambda(x_{i_1} \dots x_{i_\ell}) - \mu(x_{i_1} \dots x_{i_\ell})| &= |\tau_{\cB}(\Phi(X_{i_1} \dots X_{i_\ell}) - Y_{i_1} \dots Y_{i_\ell})| \\ 
			&\leq \ell R^{\ell-1} \left( \norm{X}_{L^2(\cA)_{\sa}^m}^2 - 2 \ip{\Phi(X),Y}_{L^2(\cB)_{\sa}^m} + \norm{Y}_{L^2(\cB)_{\sa}^m} \right)^{1/2}.
		\end{align*}
		Taking the infimum over $\Phi$, we obtain the first inequality of \eqref{eq:bimodulelawestimate}.  The second inequality follows because $d_{\bim}(\lambda,\mu) \leq d_W^{(2)}(\lambda,\mu)$ since $\operatorname{FM}(\cA,\cB) \subseteq \operatorname{UCPT}(\cA,\cB)$.
		
		Next, we show that $d_{\bim}$ is a metric on $\Sigma_{m,R}$ (postponing the proof of completeness to the end).  Clearly, $d_{\bim}(\lambda,\mu) \geq 0$.  If $d_{\bim}(\lambda,\mu) = 0$, then by \eqref{eq:bimodulelawestimate}, we have $\lambda = \mu$.  Because every $\operatorname{UCPT}$ map has a $\operatorname{UCPT}$ adjoint, we have
		\[
		C_{\bim}(\lambda,\mu) = \sup_{\Phi \in \operatorname{UCPT}(\cA,\cB)} \ip{\Phi(X),Y}_{L^2(\cB)_{\sa}^m} = \sup_{\Phi \in \operatorname{UCPT}(\cA,\cB)} \ip{X, \Phi^*(Y)}_{L^2(\cA)_{\sa}^m} = C_{\bim}(\mu,\lambda),
		\]
		and hence $d_{\bim}(\lambda,\mu) = d_{\bim}(\mu,\lambda)$.  To prove the triangle inequality, we use the fact that $\operatorname{UCPT}$ maps are closed under composition.\footnote{The corresponding notion for bimodules is the \emph{Connes fusion}, and the proof of the triangle inequality is quite natural from this viewpoint; however, we will use a more elementary argument.}  Let $\Phi \in \operatorname{UCPT}(\cA,\cB)$ and $\Psi \in \operatorname{UCPT}(\cB,\cC)$ be $\operatorname{UCPT}$ maps corresponding to optimal bimodule couplings between $\lambda$ and $\mu$ and between $\mu$ and $\nu$ respectively, so that
		\begin{align*}
			d_{\bim}(\lambda,\mu)^2 &= \norm{X}_{L^2(\cA)_{\sa}^m}^2 - 2 \ip{\Phi(X),Y}_{L^2(\cB)_{\sa}^m} + \norm{Y}_{L^2(\cB)_{\sa}^m} \\
			&= \left( \norm{X}_{L^2(\cA)_{\sa}^m}^2 - \norm{\Phi(X)}_{L^2(\cB)_{\sa}^m}^2 \right) + \norm{\Phi(X) - Y}_{L^2(\cB)_{\sa}^m}^2 \\
			&\geq \norm{\Phi(X) - Y}_{L^2(\cB)_{\sa}^m}^2
		\end{align*}
		and
		\begin{align*}
			d_{\bim}(\nu,\mu)^2 &= \norm{Z}_{L^2(\cC)_{\sa}^m}^2 - 2 \ip{\Psi^*(Z),Y}_{L^2(\cB)_{\sa}^m} + \norm{Y}_{L^2(\cB)_{\sa}^m} \\
			&= \left( \norm{Z}_{L^2(\cC)_{\sa}^m}^2 - \norm{\Psi^*(Z)}_{L^2(\cB)_{\sa}^m}^2 \right) + \norm{\Psi^*(Z) - Y}_{L^2(\cB)_{\sa}^m}^2 \\
			&\geq \norm{\Psi^*(Z) - Y}_{L^2(\cB)_{\sa}^m}^2.
		\end{align*}
		Then
		\begin{align*}
			d_{\bim}(\lambda,\nu)^2 &\leq \norm{X}_{L^2(\cA)_{\sa}^m}^2 - 2 \ip{\Psi \circ \Phi(X),Z}_{L^2(\cC)_{\sa}^m} + \norm{Z}_{L^2(\cC)_{\sa}^m}^2 \\
			&= \left( \norm{X}_{L^2(\cA)_{\sa}^m}^2 - \norm{\Phi(X)}_{L^2(\cB)_{\sa}^m}^2 \right) + \norm{\Phi(X) - \Psi^*(Z)}_{L^2(\cB)_{\sa}^m}^2 + \left( \norm{Z}_{L^2(\cC)_{\sa}^m}^2 - \norm{\Psi^*(Z)}_{L^2(\cB)_{\sa}^m}^2 \right) \\
			&\leq \left( \norm{X}_{L^2(\cA)_{\sa}^m}^2 - \norm{\Phi(X)}_{L^2(\cB)_{\sa}^m}^2 \right) + \norm{\Phi(X) - Y}_{L^2(\cB)_{\sa}^m}^2 \\
			& \quad + 2\norm{\Phi(X) - Y}_{L^2(\cB)_{\sa}^m} \norm{\Psi^*(Z) - Y}_{L^2(\cB)_{\sa}^m} \\
			& \quad + \norm{\Psi^*(Z) - Y}_{L^2(\cB)_{\sa}^m}^2 + \left( \norm{Z}_{L^2(\cC)_{\sa}^m}^2 - \norm{\Psi^*(Z)}_{L^2(\cB)_{\sa}^m}^2 \right) \\
			&\leq d_{\bim}(\lambda,\mu)^2 + 2 d_{\bim}(\lambda,\mu) d_{\bim}(\mu,\nu) + d_{\bim}(\mu,\nu)^2.
		\end{align*}
		
		It follows from \eqref{eq:bimodulelawestimate} that the $d_{\bim}$-topology refines the weak-$*$ topology, and the Wasserstein topology refines the $d_{\bim}$-topology.
		
		Next, we show that $d_{\bim}$ is lower semi-continuous with respect to the weak-$*$ topology.  Fix a non-principal ultrafilter $\cU$ on $\N$, and suppose that $(\lambda_n)_{n \in \N}$ and $(\mu_n)_{n \in \N}$ are sequences in $\Sigma_{m,R}$ and $(\cA_n,X_n)$ and $(\cB_n,Y_n)$ are their respective GNS realizations.  Let $\lambda = \lim_{n \to \cU} \lambda_n$ and $\mu = \lim_{n \to \cU} \mu_n$.  Let $\cA = \prod_{n \to \cU} \cA_n$ and $\cB = \prod_{n \to \cU} \cB_n$.  Let $X = [X_n]_{n \in \N} \in L^2(\cA)_{\sa}^m$ and $Y = [Y_n]_{n \in \N} \in L^2(\cB)_{\sa}^m$.  By Lemma \ref{lem:ultraproductlaws}, $X$ and $Y$ have non-commutative laws $\lambda$ and $\mu$ respectively.  Let $\Phi_n \in \operatorname{UCPT}(\cA_n,\cB_n)$ such that $C_{\bim}(\lambda_n,\mu_n) = \ip{\Phi_n(X_n),Y_n}_{L^2(\cB_n)_{\sa}^m}$.  If $(Z_n)_{n \in \N}$ and $(Z_n')_{n \in \N}$ are sequences in $\prod_{n \in \N} \cA_n$ and if $\lim_{n \to \cU} \norm{Z_n - Z_n'}_{L^2(\cA_n)} = 0$, then $\lim_{n \to \cU} \norm{\Phi_n(Z_n) - \Phi_n(Z_n')}_{L^2(\cB_n)} = 0$ because each $\Phi_n$ is a contraction with respect to the $L^2$ norms on $\cA$ and $\cB$.  Therefore, the equivalence class $[\Phi_n(Z_n)]_{n \in \N}$ in $\cB$ only depends on the equivalence class $[Z_n]_{n \in \N}$ in $\cA$, so that the sequence $\Phi_n$ produces a well-defined map $\Phi: \cA \to \cB$.  It is straightforward to check that $\Phi \in \operatorname{UCPT}(\cA,\cB)$.  Let $\Phi': \mathrm{W}^*(X) \to \mathrm{W}^*(Y)$ be the composition of the inclusion $\mathrm{W}^*(X) \to \cA$, the map $\Phi: \cA \to \cB$, and the trace-preserving conditional expectation $\cB \to \mathrm{W}^*(Y)$.  Then
		\[
		C_{\bim}(\lambda,\mu) \geq \ip{\Phi'(X),Y}_{L^2(\mathrm{W}^*(Y))_{\sa}^m} = \ip{\Phi(X),Y}_{L^2(\cB)_{\sa}^m} = \lim_{n \to \cU} \ip{\Phi_n(X_n),Y_n}_{L^2(\cB_n)} = \lim_{n \to \cU} C_{\bim}(\lambda_n,\mu_n).
		\]
		This implies that $d_{\bim}(\lambda,\mu) \leq \lim_{n \to \cU} d_{\bim}(\lambda_n, \mu_n)$, so $d_{\bim}$ is weak-$*$ lower semi-continuous as desired.
		
		Finally, we show that $(\Sigma_{m,R},d_{\bim})$ is complete.  Let $(\lambda_n)_{n \in \N}$ be a Cauchy sequence with respect to $d_{\bim}$.  Using \eqref{eq:bimodulelawestimate}, for each $i_1$, \dots, $i_\ell \in \{1,\dots,m\}$, the sequence $(\lambda_n(x_{i_1} \dots x_{i_\ell}))_{n \in \N}$ is Cauchy and hence converges in $\C$ to some limit $\lambda(x_{i_1} \dots x_{i_\ell})$.  Extend $\lambda$ linearly to a map on $\C\ip{x_1,\dots,x_m} \to \C$, and then it is straightforward to check that $\lambda \in \Sigma_{m,R}$ using Definition \ref{def:NClaw}.  Then because $d_{\bim}$ is weak-$*$ lower semi-continuous,
		\[
		d_{\bim}(\lambda_n,\lambda) \leq \liminf_{k \to \infty} d_{\bim}(\lambda_n, \lambda_k) \leq \sup_{k \geq n} d_{\bim}(\lambda_n, \lambda_k).
		\]
		The right-hand side goes to zero as $n \to \infty$ because $(\lambda_n)_{n \in \N}$ was assumed to be Cauchy in $d_{\bim}$.  This shows that $\lambda_n \to \lambda$ in $d_{\bim}$ as desired.
	\end{proof}
	
	We saw in the preceding argument that $C_{\bim}(\mu,\nu) \geq C(\mu,\nu)$.  In the commutative setting, we have equality by a similar argument as in \cite[Theorem 1.5]{BV2001}.  (For further discussion of bimodules over commutative tracial $\mathrm{W}^*$-algebras, see \cite[Example 13.1.2]{ADP}.)
	
	\begin{lemma}
		Let $\mu$ and $\nu \in \Sigma_{m,R}$ be non-commutative laws that can be realized by elements of commutative tracial $\mathrm{W}^*$-algebras.  Then $C_{\bim}(\mu,\nu) = C(\mu,\nu)$, and there exists an optimal coupling in a commutative tracial $\mathrm{W}^*$-algebra.
	\end{lemma}
	
	\begin{proof}
		Let $(\cA,X)$ and $(\cB,Y)$ be the GNS realizations of $\mu$ and $\nu$.  Consider an optimal bimodule coupling given by a Hilbert $\cA$-$\cB$-bimodule $H$ and a bitracial vector $\xi \in H$.  Let $X_j' \in B(H)$ be the operator of left multiplication by $X_j$, and let $Y_j \in B(H)$ be the operator of right multiplication by $Y_j$.  Let $M$ be the $\mathrm{W}^*$-subalgebra of $B(H)$ generated by $X' = (X_1',\dots,X_m')$ and $Y' = (Y_1',\dots,Y_m')$.  Since $X_i'$ and $Y_j'$ commute and $X_i'$ and $X_j'$ commute and $Y_i'$ and $Y_j'$ commute, $M$ is commutative.  Let $\tau: M \to \C$ be the map $T \mapsto \ip{\xi, T \xi}$.  Since $M$ is commutative, $\tau$ is a trace (it is a state and satisfies $\tau(ab) = \tau(ab)$).  We have not shown that it is normal or faithful, but nonetheless, the map $\gamma = \lambda_{(X',Y')}: \C\ip{x_1,\dots,x_{2m}} \to \C$ given by $p \mapsto \tau(p(X,Y))$ is still an element of $\Sigma_{2m,R}$ according to Definition \ref{def:NClaw}.  Moreover, since $\xi$ was a bitracial vector for $\cA$ and $\cB$, we have $\tau(p(X')) = \tau_{\cA}(p(X)) = \mu(p)$ and $\tau(p(Y')) = \tau_{\cB}(p(Y)) = \nu(p)$.  Therefore, $\gamma$ has the marginals $\mu$ and $\nu$.  If $(\cC,(\widehat{X},\widehat{Y}))$ is the GNS realization of $\gamma$, then $\cC$ is commutative because for any non-commutative polynomials $p$ and $q$ in $2m$ variables,
		\[
		\norm{(pq - qp)(\widehat{X},\widehat{Y})}_{L^2(\cC)}^2 = \gamma[(pq - qp)^*(pq - qp)] = \tau((pq - qp)^*(pq-qp)(X',Y')) = 0,
		\]
		and non-commutative polynomials of $X$ and $Y$ are dense in $L^2(\cC)$ (by Lemma \ref{lem:generators}).  Moreover,
		\[
		\ip*{\widehat{X},\widehat{Y}}_{L^2(\cC)_{\sa}^m} = \sum_{j=1}^m \gamma(x_j x_{m+j}) = \sum_{j=1}^m \tau(X_j' Y_j') = \sum_{j=1}^m \ip{\xi, X_j \xi Y_j}.
		\]
		Hence, $(\cC,\widehat{X},\widehat{Y})$ is a coupling in a commutative tracial $\mathrm{W}^*$-algebra which is also an optimal bimodule coupling of $\mu$ and $\nu$.
	\end{proof}
	
	For general non-commutative laws, the inequality $C(\mu,\nu) \leq C_{\bim}(\mu,\nu)$ can be strict, even for non-commutative laws of matrix tuples.  We can deduce this from another result of Haagerup and Musat that $\operatorname{FM}(M_n(\C),M_n(\C))$ is in general strictly smaller than $\operatorname{UCPT}(M_n(\C),M_n(\C))$, and in particular there is an explicit non-factorizable $\operatorname{UCPT}$ map on $M_3(\C)$.
	
	\begin{theorem}[{Haagerup-Musat \cite[Example 3.1]{HaMu2011}, \cite[Theorems 5.2 and 5.6]{HaMu2015}}] \label{thm:HWchannel}
		For $n > 1$, let $W_n^-: M_n(\C) \to M_n(\C)$ be the Holevo-Werner channel $W_n^-(x) = \frac{1}{n-1}(\Tr_n(x)1 - x^t)$.  Then $W_n^-$ is a $\operatorname{UCPT}$ map, and it is factorizable if and only if $n \neq 3$.  %Moreover, $W_n^-$ is in $\overline{\operatorname{conv}}(\Aut(M_n(\C)))$ if and only if $n$ is even \cite[Corollary 4.7 and Theorem 4.10]{HaMu2015}.  In particular, for odd $n \geq 5$, $W_n^-$ is factorizable but not in the closed convex hull of $\Aut(M_n(\C))$.  %Insert Mendl-Wolff reference
	\end{theorem}
	
	Combining non-factorizability of $W_3^-$ with Lemma \ref{lem:vectorduality} similarly to the proof of Corollary \ref{cor:nonConnes}, we deduce the following corollary.
	
	\begin{corollary} \label{cor:nonfactorizable}
		There exist $X, Y \in M_3(\C)_{\sa}^{9}$ such that $C_{\bim}(\lambda_X,\lambda_Y) > C(\lambda_X,\lambda_Y)$.
	\end{corollary}
	
	This shows that the metrics $d_{\bim}$ and $d_W^{(2)}$ are distinct.  It is unclear to us whether $d_{\bim}$ and $d_W^{(2)}$ generate the same topology.  However, the results of \S \ref{subsec:twotopologies} about the Wasserstein distance adapt to the $\operatorname{UCPT}$ setting without much difficulty.  For instance, we have the following analog of Lemma \ref{lem:ultraproductWasserstein}.
	
	\begin{lemma}
		Let $(\mu_n)_{n \in \N}$ and $\mu$ be non-commutative laws.  Let $(\cA,X)$ be the GNS realization of $\mu$.  Let $\cA_n$ be a tracial $\mathrm{W}^*$-algebra and $X_n \in L^\infty(\cA_n)_{\sa}^m$ such that $\lambda_{X_n} = \mu_n$.  Then the following are equivalent:
		\begin{enumerate}[(1)]
			\item $\lim_{n \to \cU} d_{\bim}(\mu_n,\mu) = 0$.
			\item There exists a tracial $\mathrm{W}^*$-embedding $\phi: \cA \to \prod_{n \to \cU} \cA_n$ and there exists $\Phi_n \in \operatorname{UCPT}(\cA,\cA_n)$ such that
			\[
			\phi(X) = [X_n]_{n \in \N}, \qquad \phi(Z) = [\Phi_n(Z)]_{n \in \N} \text{ for all } Z \in L^\infty(\cA).
			\]
		\end{enumerate}
	\end{lemma}
	
	\begin{proof}
		(1) $\implies$ (2).  By Lemma \ref{lem:ultraproductlaws}, there is a tracial $\mathrm{W}^*$-embedding $\phi: \cA \to \prod_{n \to \cU} \cA_n$ with $\phi(X) = [X_n]_{n \in \N}$.  Let $\Phi_n \in \operatorname{UCPT}(\cA,\cA_n)$ such that $\ip{\Phi_n(X),X_n}_{L^2(\cA_n)_{\sa}^m} = C_{\bim}(\mu_n,\mu)$.  As in the previous lemma, there exists $\Phi \in \operatorname{UCPT}(\cA, \prod_{n \to \cU} \cA_n)$ such that
		\[
		\Phi(Z) = [\Phi_n(Z)]_{n \in \N} \text{ for all } Z \in L^\infty(\cA).
		\]
		It remains to show that $\Phi = \phi$.  Let $X_n = (X_n^{(1)},\dots,X_n^{(m)})$ and $X = (X^{(1)},\dots,X^{(m)})$.  Using \eqref{eq:bimoduleestimate}, for every $i_1$, \dots, $i_\ell \in \{1,\dots,m\}$, we have
		\begin{align*}
			\norm{\Phi_n(X^{(i_1)},\dots,X^{(i_\ell)}) - X_n^{(i_1)},\dots,X_n^{(i_\ell)}}_{L^2(\cA_n)} &\leq \ell R^{\ell-1} \left( \norm{X}_{L^2(\cA)_{\sa}^m}^2 - 2 \ip{\Phi_n(X),X_n} + \norm{X_n}_{L^2(\cA)_{\sa}^m}^2 \right)^{1/2} \\
			&= \ell R^{\ell - 1} d_{\bim}(\mu_n,\mu).
		\end{align*}
		Taking $n \to \cU$, we obtain
		\[
		\norm{\Phi(X^{(i_1)},\dots,X^{(i_\ell)}) - \phi(X^{(i_1)},\dots,X^{(i_\ell)})}_{L^2(\prod_{n \to \cU} \cA_n)} \leq \lim_{n \to \cU} \ell R^{\ell-1} d_{\bim}(\mu_n,\mu) = 0.
		\]
		Hence, $\Phi(p(X)) = \phi(p(X))$ for every non-commutative polynomial $p$.  Since non-commutative polynomials are in $X$ are dense in $L^2(\cA)$ and $\Phi$ and $\phi$ are both contractions with respect to the $L^2$ norm, we have $\Phi = \phi$.
		
		(2) $\implies$ (1).  The proof is the same as in Lemma \ref{lem:ultraproductWasserstein}, so we leave the details to the reader.
	\end{proof}
	
	In a completely analogous way to Proposition \ref{prop:twotopologies}, one can deduce that the weak-$*$ and $d_{\bim}$ topologies agree at some point $\mu \in \Sigma_{m,R}$ if and only if the corresponding tracial $\mathrm{W}^*$-algebra $\cA$ obtained from the GNS construction is \emph{$\operatorname{UCPT}$-stable}, meaning that every tracial $\mathrm{W}^*$-algebra embedding from $\cA$ into some ultraproduct $\prod_{n \to \cU} \cA_n$ of tracial $\mathrm{W}^*$-algebras lifts to a sequence $(\Phi_n)_{n \in \N}$ where $\Phi_n \in \operatorname{UCPT}(\cA,\cA_n)$.  Furthermore, if $\cA$ is Connes-embeddable, then these two conditions are also equivalent to $\cA$ being approximately finite-dimensional; the proof is essentially the same as that of \cite[Theorem 2.6]{AKE2021} or that of Proposition \ref{prop:twotopologies2}.  However, it is unknown how $\operatorname{FM}$-stability and $\operatorname{UCPT}$-stability are related in the non-Connes-embeddable setting.
	
	To circle back to Monge-Kantorovich duality, given the relationship of optimal couplings with factorizable maps on the one hand and $E$-convex functions on the other hand, one might wonder whether there is an alternative version of the theory of convex functions and Legendre transforms that is based on $\operatorname{UCPT}$ maps rather than factorizable maps.  Indeed, this is possible, and we will sketch here some of the basic properties and the parts of the proof that are different from the $E$-convex case.
	
	\begin{definition}
		A $\mathrm{W}^*$-function with values in $[-\infty,\infty]$ is \emph{$\operatorname{UCPT}$-convex} if either $f$ is identically $-\infty$, or else for every $\cA$, $f^{\cA}$ is a convex and lower semi-continuous function with values in $(-\infty,\infty]$, and we have $f^{\cA}(X) \leq f^{\cB}(\Phi(X))$ for every $\cA$, $\cB \in \mathbb{W}$ and $\Phi \in \operatorname{UCPT}(\cA,\cB)$ and $X \in L^2(\cA)_{\sa}^m$.
	\end{definition}
	
	\begin{definition}
		The \emph{$\operatorname{UCPT}$-Legendre transform} of a tracial $\mathrm{W}^*$-function $f$ is the tracial $\mathrm{W}^*$-function $\mathcal{K}f$ given by
		\[
		(\mathcal{K}f)^{\cA}(X) = \sup_{\substack{\cB \in \mathbb{W} \\ \Phi \in \operatorname{UCPT}(\cA,\cB) \\ Y \in L^2(\cB)_{\sa}^m}} \ip{\Phi(X),Y}_{L^2(\cB)_{\sa}^m} - f^{\cB}(Y).
		\]
	\end{definition}
	
	We have the following analog of Proposition \ref{prop:Legendre}.
	
	\begin{proposition} \label{prop:UCPTLegendre}
		If $f$, $g$ be a tracial $\mathrm{W}^*$-functions.
		\begin{enumerate}[(1)]
			\item $\mathcal{K}f$ is $\operatorname{UCPT}$-convex.
			\item If $f \leq g$, then $\mathcal{K}f \geq \mathcal{K} g$.
			\item We have $\mathcal{K}^2 f \leq f$ with equality if and only if $f$ is $\operatorname{UCPT}$-convex.
			\item $\mathcal{K}^2 f$ is the maximal $\operatorname{UCPT}$-convex function that is less than or equal to $f$.
		\end{enumerate}
	\end{proposition}
	
	The proof is essentially the same as that of Proposition \ref{prop:Legendre}, modulo the necessary changes to work with $\operatorname{UCPT}$ maps rather than tracial $\mathrm{W}^*$-embeddings and conditional expectations.  For instance, to show monotonicity of $\mathcal{K}f$ under $\operatorname{UCPT}$ maps, suppose that $\Phi \in \operatorname{UCPT}(\cA,\cB)$ and $X \in L^2(\cA)_{\sa}^m$.  If $\Psi \in \operatorname{UCPT}(\cB,\cC)$, then $\Psi \circ \Phi \in \operatorname{UCPT}(\cA,\cC)$.  Therefore,
	\[
	\mathcal{K} f^{\cA}(X) \geq \sup_{\substack{\cC \in \mathbb{W} \\ \Psi \in \operatorname{UCPT}(\cB,\cC) \\ Y \in L^2(\cC)_{\sa}^m}} \left( \ip{\Psi \circ \Phi(X),Y} - f^{\cC}(Y) \right) = \mathcal{K} f^{\cB}(\Phi(X)).
	\]
	The relationship between the $\operatorname{UCPT}$ Legendre transform and the $E$-convex Legendre transform is as follows (compare the relationship between the $E$-convex Legendre transform and the Hilbert-space Legendre tranform).
	
	\begin{corollary}
		Let $f$ be a tracial $\mathrm{W}^*$-function.
		\begin{enumerate}[(1)]
			\item If $f$ is $\operatorname{UCPT}$-convex, then $f$ is $E$-convex.
			\item $\mathcal{K}f \geq \mathcal{L} f$.
			\item $\mathcal{K}^2 f \leq \mathcal{L}^2 f$.
			\item If $f$ is $\operatorname{UCPT}$-convex, then $\mathcal{K}f = \mathcal{L}f$.
		\end{enumerate}
	\end{corollary}
	
	\begin{proof}
		(1) and (2) are immediate from the definitions of $\mathcal{L}$ and $\mathcal{K}$ since every tracial $\mathrm{W}^*$-embedding is a $\operatorname{UCPT}$ map.
		
		(3)  Observe that $\mathcal{K}^2 f$ is $E$-convex by (1) and $\mathcal{K}^2 f \leq f$.  Therefore, Proposition \ref{prop:Legendre} (4) implies that $\mathcal{K}^2 f \leq \mathcal{L}^2 f$.
		
		(4) We already know that $\mathcal{L}f \leq \mathcal{K}f$.  For the reverse inequality, the idea is already contained in the proof of Proposition \ref{prop:UCPTLegendre} (3).  Note that for $\cA, \cB \in \mathbb{W}$ and $\Phi \in \operatorname{UCPT}(\cA,\cB)$ and $X \in L^2(\cA)_{\sa}^m$ and $Y \in L^2(\cB)_{\sa}^m$, we have
		\[
		\ip{\Phi(X),Y}_{L^2(\cB)_{\sa}^m} - f^{\cB}(Y) \leq \ip{X,\Phi^*(Y)}_{L^2(\cA)_{\sa}^m} - f^{\cA}(\Phi^*(Y)) \leq \mathcal{L} f^{\cA}(X).
		\]
		Taking the supremum over $\cB$, $\Phi$, and $Y$, we obtain $\mathcal{K}f \leq \mathcal{L}f$.
	\end{proof}
	
	The $\operatorname{UCPT}$-analog of Monge-Kantorovich duality is as follows.
	
	\begin{definition}
		A pair of tracial $\mathrm{W}^*$-functions $(f,g)$ with values in $(-\infty,\infty]$ is said to be \emph{$\operatorname{UCPT}$-admissible} if for every $\cA$, $\cB \in \mathbb{W}$ and $X \in L^2(\cA)_{\sa}^m$ and $Y \in L^2(\cB)_{\sa}^m$ and $\Phi \in \operatorname{UCPT}(\cA,\cB)$, we have
		\[
		f^{\cA}(X) + g^{\cB}(Y) \geq \ip{\Phi(X),Y}_{L^2(\cB)_{\sa}^m}.
		\]
	\end{definition}
	
	\begin{proposition}
		$C_{\bim}(\mu,\nu)$ is equal to the infimum of $\mu(f) + \nu(g)$ over all $\operatorname{UCPT}$-admissible pairs of tracial $\mathrm{W}^*$-functions, as well as the infimum of $\mu(f) + \nu(g)$ over all $\operatorname{UCPT}$-admissible pairs of $\operatorname{UCPT}$-convex functions.
	\end{proposition}
	
	The proof is the same as that of Proposition \ref{prop:MKduality1}; similarly, there is an $\operatorname{UCPT}$ analog of Proposition \ref{prop:MKduality2}.  However, although there is an analog of Monge-Kantorovich duality, there are many questions about bimodule couplings for which the answer is not immediately clear:
	\begin{itemize}
		\item Is there a bimodule analog of the displacement interpolation?
		\item Is there a bimodule analog of the $L^p$ Wasserstein distance for $p \neq 2$?
		\item Is there a useful subgradient characterization of $\operatorname{UCPT}$-convexity analogous to Lemma \ref{lem:Econvex}?
		\item Do $d_{\bim}$ and $d_W^{(2)}$ generate the same topology on $\Sigma_{m,R}$?
	\end{itemize}
	
	\appendix
	
	\section{Non-commutative laws and couplings for $L^p$ variables} \label{sec:Lp}
	
	Although we have focused in this paper on the non-commutative $L^2$-Wasserstein distance, Biane and Voiculescu \cite{BV2001} also defined $L^p$ Wasserstein distance for $p \in [1,\infty)$.  Although they only defined the Wasserstein distance for tuples of bounded operators, it is natural to extend the theory to non-commutative $L^p$ spaces. In this section, after reviewing the properties of affiliated operators to a tracial $\mathrm{W}^*$-algebra, we define laws, couplings, and $L^p$ Wasserstein distance for $m$-tuples of self-adjoint operators in non-commutative $L^p$ space, and show the existence of optimal couplings and Wasserstein geodesics.
	
	\subsection{Affiliated operators, $L^p$ spaces} \label{subsec:affiliated}
	
	For background on unbounded operators, refer for instance to \cite[\S VIII]{ReedSimon1972}.  We recall that if $H$ is a Hilbert and $T: H \supseteq \dom(T) \to H$ is a closed densely defined unbounded operator, then $T$ has a polar decomposition as $U |T|$ where $|T|$ is a positive self-adjoint operator with $\dom(|T|) = \dom(T)$ and $U$ is a partial isometry \cite[\S VIII.9]{ReedSimon1972}.
	
	We quote without proof the basic definitions and results about affiliated operators and non-commutative $L^p$ spaces.  Affiliated operators were first studied by Murray and von Neumann in \cite[\S XVI]{MvN1}, and the non-commutative $L^p$ spaces were studied in \cite{Dixmier1953}.  For a modern exposition of affiliated operators and $L^p$ spaces in English, see \cite{daSilva2018} as well as \cite[\S 7.2]{ADP}.
	
	Let $\cA = (A,\tau)$ be a tracial $\mathrm{W}^*$-algebra.  To avoid ambiguity, we use the notation $H_{\cA}$ rather than $L^2(\cA)$ for the completion of $A$ with respect to the inner product $(a,b) \mapsto \tau(a^*b)$.  We still use the notation $\widehat{a}$ for the element of $H_{\cA}$ corresponding to $a \in A$.
	
	\begin{definition}
		Let $\cA = (A,\tau)$ be a tracial $\mathrm{W}^*$-algebra, and let us view $A$ as a subset of $B(H_{\cA})$ as in Theorem \ref{thm:W*GNSrep}.  A closed densely defined operator $T: \dom(T) \to H_{\cA}$ with polar decomposition $U|T|$ is \emph{affiliated to $A$} if $U \in A$ and the spectral projection $1_{S}(|T|) \in A$ for every Borel set $S \subseteq [0,\infty)$.  We denote the set of affiliated operators by $\Aff(\cA)$.
	\end{definition}
	
	\begin{example}
		Let $(\Omega,P)$ be a probability space and let $\cA$ be $L^\infty(\Omega,P)$ equipped with the trace given by integration against $P$.  Then $\Aff(\cA)$ can be canonically identified with measurable functions on $\Omega$ that are finite almost everywhere, viewed as unbounded multiplication operators on $L^2(\Omega,P)$.
	\end{example}
	
	\begin{theorem} \label{thm:affiliated}
		Let $\cA = (A,\tau)$ be a tracial $\mathrm{W}^*$-algebra.
		\begin{enumerate}
			\item $\Aff(\cA)$ satisfies the following properties:
			\begin{itemize}
				\item If $T \in \Aff(A)$, then $T^* \in \Aff(\cA)$.
				\item If $T_1$, $T_2 \in \Aff(\cA)$, then $T_1|_{\dom(T_1) \cap \dom(T_2)} + T_2|_{\dom(T_1) + \dom(T_2)}$ is closeable, and its closure is in $\Aff(\cA)$.
				\item If $T_1$, $T_2 \in \Aff(A)$, then $T_1 T_2|_{T_2^{-1}(\dom(T_1))}$ is closeable and its closure is in $\Aff(\cA)$.
			\end{itemize}
			In this way, $\Aff(\cA)$ can be equipped with the structure of a $*$-algebra.
			\item The canonical inclusion $A \to \Aff(\cA)$ is a $*$-homomorphism.  Moreover, if $T \in \Aff(\cA)$ is a bounded operator, then $T \in A$.
			\item If $T \in \Aff(\cA)$ is a normal operator, and $f$ is a Borel function on its spectrum, then $f(T) \in \Aff(\cA)$.  If $f$ is bounded, then $f(T) \in A$.  There is a unique probability measure $\mu_T$ on $\C$ such that $\tau(f(T)) = \int f\,d\mu_T$ for all bounded Borel functions $f$.  The spectrum of $T$ is exactly the closed support of $\mu_T$.
			\item Let $\Aff(\cA)_+$ be the set of positive operators affiliated to $\cA = (A,\tau)$.  Then $\tau$ extends to a map $\Aff(\cA)_+ \to [0,\infty]$ satisfying
			\[
			\tau(T) = \lim_{n \to \infty} \tau(f_n(T)),
			\]
			whenever $f_n$ is any sequence of nonnegative Borel functions increasing to the identity function on $[0,\infty)$.
		\end{enumerate}
	\end{theorem}
	
	\begin{definition}
		For a tracial $\mathrm{W}^*$-algebra $\cA = (A,\tau)$ and $p \in [1,\infty)$, we define
		\[
		L^p(\cA) = \{T \in \Aff(\cA): \tau(|T|^p) < \infty \},
		\]
		and we write
		\[
		\norm{T}_{L^p(\cA)} = \tau(|T|^p)^{1/p}.
		\]
		As stated above, $L^\infty(\cA) = A$ and $\norm{T}_{L^\infty(\cA)}$ is the norm on $A$.
	\end{definition}
	
	\begin{theorem}
		Let $\cA = (A,\tau)$ as above.
		\begin{enumerate}[(1)]
			\item For $p \in [1,\infty]$, $\norm{\cdot}_p$ defines a norm on $L^p(\cA)$, and $L^p(\cA)$ is a complete with respect to this norm, hence it is a Banach space.
			\item $A$ is a dense subspace of $L^p(\cA)$ for $p \in [1,\infty)$.
			\item Let $p$, $p_1$, $p_2 \in [1,\infty]$ with $1/p = 1/p_1 + 1/p_2$.  If $T_1 \in L^{p_1}(\cA)$ and $T_2 \in L^{p_2}(\cA)$, then $T_1 T_2 \in L^p(\cA)$ and $\norm{T_1T_2}_p \leq \norm{T_1}_{p_1} \norm{T_2}_{p_2}$.
			\item For $p \in [1,\infty]$, if $T \in L^p(\cA)$, then $T^* \in L^p(\cA)$ with $\norm{T}_p = \norm{T^*}_p$.
			\item $\tau$ extends uniquely to a bounded map $L^1(\cA) \to \C$ (still denoted by $\tau$ or $\tau_{\cA}$) that satisfies $\tau(a^*) = \overline{\tau(a)}$
			\item Let $p \in [1,\infty)$ and let $1/p + 1/q = 1$.  Then $L^q(\cA)$ may be canonically identified with the dual of $L^p(\cA)$ through the pairing $(T_1,T_2) \mapsto \tau(T_1 T_2)$.  In particular, this yields an identification between $L^1(\cA)$ and $A_*$.
			\item If $A_*$ is separable, then $L^p(\cA)$ is separable for $p \in [1,\infty)$.
		\end{enumerate}
	\end{theorem}
	
	\begin{theorem}
		Let $\cA = (A,\tau)$, and $T \in \Aff(\cA)$.  Then $T \in L^2(\cA)$ if and only if $\widehat{1} \in \dom(T)$.  There is unitary isomorphism of Hilbert spaces $\phi: L^2(\cA) \to H_{\cA}$ given by $T \mapsto T\widehat{1}$.  Furthermore, for $T \in L^2(\cA)$ and $a \in A$, we have $\phi(aT) = a \phi(T)$.
	\end{theorem}
	
	The next lemma can be deduced from well-known facts about $\mathrm{C}^*$-algebras as well as the properties of affiliated operators and $L^p$ spaces in the previous section.
	
	\begin{lemma}
		Let $\iota: \cA \to \cB$ be a tracial $\mathrm{W}^*$-embedding.  Then $\norm{\iota(a)} = \norm{a}$ for $a \in A$.  Moreover, for $a \in \cA$ and $p \in [1,\infty)$, we have $\iota((a^*a)^{p/2}) = (\iota(a)^*\iota(a))^{p/2}$.  Hence, $\norm{\iota(a)}_{L^p(\cB)} = \norm{a}_{L^p(\cA)}$ for $a \in \cA$, and therefore, $\iota$ extends to an isometric linear map $L^p(\cA) \to L^p(\cB)$ for every $p \in [1,\infty)$.  In fact, $\iota$ extends to an injective $*$-homomorphism $\Aff(\cA) \to \Aff(\cB)$.
	\end{lemma}
	
	\begin{notation}
		If $\iota: \cA \to \cB$ is a tracial $\mathrm{W}^*$-embedding, we will denote the extended map $\Aff(\cA) \to \Aff(\cB)$ also by $\iota$.
	\end{notation}
	
	\begin{proposition}
		Let $\iota: \cA = (A,\tau) \to \cB = (B,\sigma)$ be a tracial $\mathrm{W}^*$-embedding.  Let $E: L^2(\cB) \to L^2(\cA)$ be the adjoint of the map $\iota: L^2(\cA) \to L^2(\cB)$.
		\begin{enumerate}[(1)]
			\item For $p \in [1,\infty]$, $E$ extends to a unique bounded linear map $L^p(\cB) \to L^p(\cA)$, and (denoting the extended map still by $E$) we have $\norm{E(b)}_{L^p(\cA)} \leq \norm{b}_{L^p(\cB)}$.
			\item For all $b \in L^1(\cB)$, we have $\tau(E(b)) = \sigma(b)$; in other words, $E$ is trace-preserving.
			\item For all $b \in L^1(\cB)$, we have $E(b^*) = E(b)^*$.
			\item If $b \in L^p(\cB)$ and $a \in L^q(\cA)$ with $1/p + 1/q = 1$, then $E[\iota(a) b] = a E[b]$ and $E[b \iota(a)] = E[b] a$.
		\end{enumerate}
	\end{proposition}
	
	\begin{proof}[Sketch of proof]
		(1) Fix $p, q \in [1,\infty]$ with $1/p + 1/q = 1$.  If $a \in L^\infty(\cA)$ and $b \in L^\infty(\cB)$, we have
		\[
		|\ip{a,E[b]}_{\cA}| = |\ip{\iota(a),b}_{\cB}| \leq \norm{\iota(a)}_{L^q(\cB)} \norm{b}_{L^p(\cB)} = \norm{a}_{L^q(\cA)} \norm{b}_{L^p(\cB)}.
		\]
		By density of $L^\infty(\cA)$ in $L^q(\cA)$ and the duality of $L^q(\cA)$ and $L^p(\cA)$, we obtain $\norm{E[b]}_{L^p(\cA)} \leq \norm{b}_{L^p(\cB)}$, and the extension follows from this.
		
		(2) Since $\iota(1) = 1$ and $E = \iota^*$, we obtain $\tau(E(b)) = \sigma(b)$ for $b \in L^2(\cB)$ and this extends to $L^1(\cB)$ by density.
		
		(3) (4) The claims are first checked for $b \in L^\infty(\cB)$ and $a \in L^\infty(\cA)$ using the properties of the trace and the fact that $E$ is the adjoint of $\iota$, by similar reasoning as in Lemma \ref{lem:Econvexitymotivation}.  Then we use density of $L^\infty$ in $L^p$ to conclude.
	\end{proof}
	
	\begin{notation}
		Let $\cA = (A,\tau)$ be a tracial $\mathrm{W}^*$-algebra.  Let $X = (X_1,\dots,X_m) \in \Aff(\cA)^m$.  We denote by $\mathrm{W}^*(X)$ the smallest $\mathrm{W}^*$-subalgebra of $A$ to which $X_1$, \dots, $X_m$ are affiliated operators.  Equivalently, letting $X_j = U_j |X_j|$ be a polar decomposition of $X_j$, $\mathrm{W}^*(X)$ is the weak-$*$ closure of the $*$-algebra generated by $U_j$ and $f(|X_j|)$ for bounded Borel functions $f: \R \to \C$ and $j = 1$, \dots, $m$.  We view $\mathrm{W}^*(X)$ as a tracial $\mathrm{W}^*$-algebra, where the trace is the simply the restriction of $\tau$.
	\end{notation}
	
	\begin{notation} \label{not:Lpm}
		Let $\cA = (A,\tau)$ and $\cB = (B,\sigma)$ be tracial $\mathrm{W}^*$-algebras.  For $p \in [1,\infty]$, we equip $L^p(\cA)^m$ with the norm
		\[
		\norm{(X_1,\dots,X_m)}_{L^p(\cA)^m} = \begin{cases} \left(\sum_{j=1}^m \tau((X_j^*X_j)^{p/2}) \right)^{1/p}, & p < \infty, \\ \max_{j=1,\dots,m} \norm{X_j}_A, & p = \infty \end{cases}
		\]
		Note that $L^p(\cA)_{\sa}^m$ is a real subspace of $L^p(\cA)^m$.  For $X = (X_1,\dots,X_m)$ and $Y = (Y_1,\dots,Y_m)$ in $L^2(\cA)^m$, we define
		\[
		\ip{X,Y}_{L^2(\cA)^m} = \sum_{j=1}^m \tau(X_j^*Y_j) = \sum_{j=1}^m \ip{X_j,Y_j}_{\cA}.
		\]
		Given a tracial $\mathrm{W}^*$-embedding $\iota: \cA \to \cB$ and the corresponding conditional expectation $E: \cB \to \cA$, if $X \in L^1(\cA)^m$ and $Y \in L^1(\cB)^m$, we write
		\[
		\iota(X) := (\iota(X_1),\dots,\iota(X_m)), \qquad E[Y] := (E[Y_1],\dots,E[Y_m]).
		\]
		Note that $\iota$ and $E$ both preserve the real subspaces of self-adjoint tuples.
	\end{notation}

	\subsection{Laws and the Wasserstein distance for $L^p$ variables}
	
	To extend the notion of non-commutative laws and couplings to $L^p$ variables, we use a fairly standard trick in operator algebras, namely transforming an unbounded operator into a bounded operator using functional calculus.  If $(X_1,\dots,X_m) \in L^p(\cA)_{\sa}^m$, then $\arctan(X) := (\arctan(X_1),\dots,\arctan(X_m))$, where $\arctan(X_j)$ is defined by functional calculus, is an $m$-tuple of bounded self-adjoint operators with has a non-commutative law $\lambda_{\arctan(X)} \in \Sigma_{m,\pi/2}$.  Rather than defining $\lambda_X$ directly, we will work with $\lambda_{\arctan(X)}$ instead.  The analogous procedure in classical probability theory would be to study a probability distribution $\mu$ on $\R^m$ through the compactly supported probabilitity distribution $\arctan_* \mu$ obtained by pushing forward $\mu$ by the function $(x_1,\dots,x_d) \mapsto (\arctan(x_1),\dots,\arctan(x_d))$.
	
	Given a law $\lambda \in \Sigma_{m,\pi/2}$, the following criterion decides whether $\lambda = \lambda_{\arctan(X)}$ for some $m$-tuple $X$ in a non-commutative $L^p$-space:
	
	\begin{lemma}
		Let $\lambda \in \Sigma_{m,\pi/2}$.  For each $j$ there is a measure $\lambda_j$ on $[-\pi/2,\pi/2]$ with $\int f(x)\,d\lambda_j(x) = \lambda(f(x_j))$ for all non-commutative polynomials $f$.  The following are equivalent:
		\begin{enumerate}[(1)]
			\item $\tan \in L^p(\lambda_j)$ for every $j$.
			\item There exists a tracial $\mathrm{W}^*$-algebra $\cA$ and $X \in L^p(\cA)_{\sa}^m$ such that $\lambda = \lambda_{\arctan(X)}$. 
		\end{enumerate}
	\end{lemma}
	
	\begin{proof}
		(1) $\implies$ (2).  Let $(\cA,Y)$ be the GNS realization of $\lambda$ given by Proposition \ref{prop:GNS}.  Then $\lambda_j$ is the spectral distribution of $Y_j$ with respect to $\tau$.  Because $\tan \in L^p(\lambda_j)$ for each $j$, we know that $\lambda_j$ has no mass at $\pm \pi/2$ and therefore $X_j = \tan(Y_j)$ is a well-defined self-adjoint operator affiliated to $\cA$.  Using Theorem \ref{thm:affiliated} (3), we have $\mu_{\tan(Y_j)} = \tan_* \mu_{Y_j} = \tan_*\lambda_j$.  Hence, $\tau(|X_j|^p) = \int |t|^p \,d\mu_{X_j}(t) = \int |\tan t|^p\,d\lambda_j(t) < \infty$.  Therefore, $X_j \in L^p(\cA)_{\sa}$ and $\lambda_{\arctan(X_j)} = \lambda$.
		
		(2) $\implies$ (1).  If $X$ is as in (2), then let $Y_j = \arctan(X_j)$.  The spectral distribution of $Y_j$ with respect to $\tau$ is thus $\mu_{Y_j} = \lambda$.
	\end{proof}
	
	\begin{definition}
		We define $\Sigma_m^{(p)}$ as the set of non-commutative laws $\lambda$ in $\Sigma_{m,\pi/2}$ such that $\tan \in L^2(\lambda_j)$ for every $j$.  We define the weak-$*$ topology on $\Sigma_m^{(p)}$ as the restriction of the weak-$*$ topology on $\Sigma_{m,\pi/2}$.
	\end{definition}
	
	Next, we define couplings of laws in $\Sigma_m^{(p)}$.  It will be useful to have two different points of view on couplings, one more measure-theoretic, and the other more probabilistic.
	
	\begin{definition}
		Given $\mu, \nu \in \Sigma_m^{(p)}$, a \emph{measure-theoretic coupling} of $\mu$ and $\nu$ is a law $\gamma \in \Sigma_{2m}^{(p)}$ such that $\gamma(f(x_1,\dots,x_m)) = \mu(f(x_1,\dots,x_m))$ and $\gamma(f(x_{m+1},\dots,x_{2m})) = \nu(f(x_1,\dots,x_m))$ for all $f \in \C\ip{x_1,\dots,x_m}$.  We denote by $\Gamma^{(p)}(\mu,\nu) \subseteq \Sigma_{2m}^{(p)}$ the space of measure-theoretic couplings.
	\end{definition}
	
	\begin{definition}
		A \emph{probabilistic coupling} of $\mu$ and $\nu$ is a tuple $(\cA,X,Y)$, where $\cA$ is a tracial $\mathrm{W}^*$-algebra and $X, Y \in L^p(\cA)_{\sa}^m$ with $\lambda_{\arctan(X)} = \mu$ and $\lambda_{\arctan(Y)} = \nu$.
	\end{definition}
	
	Of course, if $(\cA,\mu,\nu)$ is a probabilistic coupling, then $\gamma = \lambda_{\arctan(X),\arctan(Y)}$ is a measure-theoretic coupling.  Conversely, if $\gamma$ is a measure-theoretic coupling, then a probabilistic coupling can be obtained from the GNS construction of $\gamma$.
	
	\begin{definition}[Wasserstein distance] \label{def:NCoptimalcoupling}
		For a given $\mu$, $\nu \in \Sigma_m^{(p)}$, we define $d_W^{(p)}(\mu,\nu)$ to be the infimum $\norm{X - Y}_{L^p(\cA)_{\sa}^m}$ over all probabilistic couplings $(\cA,X,Y)$ with $\cA \in \mathbb{W}$.
	\end{definition}
	
	\begin{proposition}
		The Wasserstein distance $d_W^{(p)}$ defines a metric on the set $\Sigma_m^{(p)}$ which makes it into a complete metric space.
	\end{proposition}
	
	The argument to show that $d_W^{(p)}$ is a metric on $\Sigma_m^{(p)}$ is exactly the same as in \cite{BV2001}.  The hardest axiom to verify is the triangle inequality, but this follows because a coupling of $\mu_1$ and $\mu_2$ and a coupling of $\mu_2$ and $\mu_3$ can be joined by taking the amalgamated free product of the tracial $\mathrm{W}^*$-algebras corresponding to the two couplings over the subalgebra generated by the variables corresponding to $\mu_2$.  Hence, $\Sigma_m^{(p)}$ is a metric space.
	
	\subsection{Optimal couplings and Wasserstein geodesics}
	
	\begin{definition}
		A \emph{probabilistic coupling} $(\cA,X,Y)$ of $\mu$ and $\nu \in \Sigma_m^{(p)}$ is said to be \emph{optimal} if $\norm{X - Y}_{L^p(\cA)_{\sa}^m} = d_W^{(p)}(\mu,\nu)$; in this case we also say that the corresponding \emph{measure-theoretic coupling} is \emph{optimal}.  We denote the space of optimal measure-theoretic couplings by $\Gamma_{\opt}^{(p)}(\mu,\nu)$.
	\end{definition}
	
	To show the existence of optimal couplings, we use a certain type of continuity and compactness.
	
	\begin{lemma}
		Let $p \in [1,\infty)$.  Then $\Gamma^{(p)}(\mu,\nu)$ is compact in the weak-$*$ topology.  For $\gamma \in \Gamma^{(p)}(\mu,\nu)$, the quantity
		\[
		N(\gamma,p) := \norm{X - Y}_{L^p(\cA)_{\sa}^m} \text{ where } \gamma = \lambda_{\arctan(X),\arctan(Y)},
		\]
		where $(X,Y) \in L^2(\cA)_{\sa}^m$ is a $2m$-tuple with $\lambda_{\arctan(X),\arctan(X)}$, only depends on $\gamma$, and moreover $\gamma \mapsto N(\gamma,p)$ is continuous.
	\end{lemma}
	
	\begin{proof}
		First, $\Gamma^{(p)}(\mu,\nu)$ is compact because it is a closed subset of $\Sigma_{2m,\pi/2}$, which is compact.  To show well-definedness and continuity $\gamma \mapsto N(\gamma,p)$, first note that for any polynomial $\phi$, the map
		\[
		\lambda_{Z,Z'} \mapsto \norm{\phi(Z) - \phi(Z')}_{L^p(\cA)_{\sa}^m}
		\]
		is well-defined and continuous on $\Sigma_{2m,\pi/2}$.  Indeed, let $M$ be an upper bound on $|\phi|$ and let $g_k$ be a sequence of polynomials that converge uniformly on $[-2M,2M]$ to the function $| \cdot |^{p/2}$.  Then $g_k((\phi(Z) - \phi(Z'))^*(\phi(Z) - \phi(Z')))$ converges to $|\phi(Z) - \phi(Z')|^p$ and the rate of convergence is uniform for all $\cA$ and all $(Z,Z')$ with $\norm{(Z,Z')}_\infty \leq \pi$ because of the spectral mapping theorem.  The continuity of
		\[
		\lambda_{Z,Z'} \mapsto \sum_{j=1}^m g_k(\phi(Z) - \phi(Z')) = \lambda_{Z,Z'}(g_k(\phi(x_j) - \phi(x_{m+j})))
		\]
		is immediate by definition of the weak-$*$ topology.  Thus, the continuity of $\lambda_{Z,Z'} \mapsto \norm{\phi(Z) - \phi(Z')}_p$ follows from uniform convergence.  Similarly, using uniform convergence, we can generalize $\phi$ from a polynomial to an arbitrary continuous real-valued function on $[-\pi/2,\pi/2]$.
		
		Now let $\phi_k \in C([-\pi/2,\pi/2];\R)$ be a sequence such that $|\phi_k| \leq |\tan|$ and $\phi_k \to \tan$ pointwise.  Suppose that $\gamma \in \Gamma^{(p)}(\mu,\nu)$ and $\gamma = \lambda_{(Z,Z')}$.  Then $(\tan(Z),\tau(Z')) \in L^p(\cA)_{\sa}^m$.  Also,
		\begin{align*}
			\left| \norm{\tan(Z) - \tan(Z')}_{L^p(\cA)_{\sa}^m} - \norm{\phi_k(Z) - \phi_k(Z')}_{L^p(\cA)_{\sa}^m} \right| &\leq \sum_{j=1}^m \norm{\tan(Z_j) - \phi_k(Z_j)}_{L^p(\cA)} + \sum_{j=1}^m \norm{\tan(Z_j') - \phi_k(Z_j')}_{L^p(\cA)} \\
			&= \sum_{j=1}^m \left( \int |\tan - \phi_k|^p\,d\mu_j \right)^{1/p} + \sum_{j=1}^m \left( \int |\tan - \phi_k|^p\,d\nu_j \right)^{1/p},
		\end{align*}
		where $\mu_j$ and $\nu_j$ are the measures on $[-\pi/2,\pi/2]$ representing the $j$th marginals of $\mu$ and $\nu$ respectively.  The bound on the right-hand side only depends on $\mu$ and $\nu$ and thus is a uniform bound for all $\gamma \in \Gamma^{(p)}(\mu,\nu)$.  Furthermore, by the dominated convergence theorem $|\tan - \phi_k| \to 0$ in $L^p(\mu_j)$ and $L^p(\nu_j)$.  Therefore, the map sending $\gamma$ to $\norm{\tan(Z) - \tan(Z')}_{L^p(\cA)_{\sa}^m}$ is continuous as the uniform limit of continuous maps. 
	\end{proof}
	
	\begin{corollary}
		For each $\mu, \nu \in \Sigma_m^{(p)}$, the space of optimal couplings $\Gamma_{\opt}^{(p)}$ is nonempty and compact.
	\end{corollary}
	
	Given the existence of $L^2$ optimal couplings, all the theorems from \S \ref{sec:duality} and \S \ref{sec:displacement} can be generalized to $\Sigma_m^{(2)}$ with the appropriate changes to notation.  Almost no change is needed for the proofs since $E$-convex functions were defined for $L^2$ non-commutative random variables to begin with.
	
	Just as in the case of classical probability theory, the existence of optimal couplings and the ability to take convex combinations of non-commutative random variables immediately leads to the existence of geodesics between any two points in $\Sigma_m^{(p)}$.
	
	\begin{definition}
		Let $(\Omega,d)$ be a metric space.  A \emph{geodesic} in $(\Omega,d)$ is a continuous map $g: I \to \Omega$, where $I \subseteq \R$ is an interval (of positive length), such that for all $t_1 < t_2 < t_3$ in $I$, we have
		\[
		d(g(t_1),g(t_3)) = d(g(t_1),g(t_2)) + d(g(t_2),g(t_3)).
		\]
		The geodesic is said to be \emph{constant speed} if $d(g(t_1),g(t_2)) / (t_2 - t_1)$ is constant for all $t_1 < t_2$.
	\end{definition}
	
	\begin{proposition} \label{prop:geodesic}
		Let $p \in [1,\infty)$.  Let $\mu, \nu \in \Sigma_d^{(p)}$, and let $(\cA,X,Y)$ be a probabilistic optimal coupling.  For $t \in [0,1]$, let $X_t = (1 - t)X + tY$.  Let $\mu_t = \lambda_{\arctan(X_t)}$.  Then $t \mapsto \mu_t$ is a constant speed geodesic in $(\Sigma_m^{(p)},d_W^{(p)})$.  Moreover, for $s, t \in [0,1]$, $(\cA,X_s,X_t)$ is an optimal coupling of $\mu_s$ and $\mu_t$.
	\end{proposition}
	
	\begin{proof}
		Of course,
		\[
		d_W^{(p)}(\mu_s,\mu_t) \leq \norm{X_s - X_t}_{L^p(\cA)_{\sa}^m} = |s - t| \norm{X - Y}_{L^p(\cA)_{\sa}^m}.
		\]
		Thus, for $s < t$,
		\begin{align*}
			\norm{X - Y}_{L^p(\cA)_{\sa}^m} &= d_W^{(p)}(X,Y) \\
			&\leq d_W^{(p)}(X,X_s) + d_W^{(p)}(X_s,X_t) + d_W^{(p)}(X_t,Y) \\
			&\leq \norm{X - X_s}_{L^p(\cA)_{\sa}^m} + \norm{X_s - X_t}_{L^p(\cA)_{\sa}^m} + \norm{X_t - Y}_{L^p(\cA)_{\sa}^m} \\
			&= \norm{X - Y}_{L^p(\cA)_{\sa}^m}.
		\end{align*}
		Thus, all the inequalities are forced to be equalities.  Hence, $d_W^{(p)}(\mu_s,\mu_t) = (t - s) d_W^{(p)}(\mu,\nu)$.  Therefore, $t \mapsto \mu_t$ is a constant speed geodesic.  Also, $d_W^{(p)}(\mu_s,\mu_t) = \norm{X_s - Y_t}_{L^p(\cA)_{\sa}^m}$, so that $(\cA,X_s,X_t)$ is an optimal coupling.
	\end{proof}

	\bibliographystyle{plain}
	\bibliography{optimal_transport}

\begin{thebibliography}{10}

\bibitem{AD2006}
Claire Anantharaman-Delaroche.
\newblock On ergodic theorems for free group actions on noncommutative spaces.
\newblock {\em Probab. Theory Rel. Fields}, 135:520--546, 2006.

\bibitem{ADP}
Claire Anantharaman-Delaroche and Sorin Popa.
\newblock An introduction to $\mathrm{II}_1$ factors.
\newblock 2021.

\bibitem{AGZ2009}
Greg~W. Anderson, Alice Guionnet, and Ofer Zeitouni.
\newblock {\em An Introduction to Random Matrices}.
\newblock Cambridge Studies in Advanced Mathematics. Cambridge University
  Press, 2009.

\bibitem{AGKE2020}
Scott Atkinson, Isaac Goldbring, and Srivatsav {Kunnawalkam Ellayavalli}.
\newblock Factorial relative commutants and the generalized {J}ung property for
  $\mathrm{II}_1$ factors.
\newblock arXiv:2004.02293.

\bibitem{AKE2021}
Scott Atkinson and Srivatsav {Kunnawalkam Ellayavalli}.
\newblock On ultraproduct embeddings and amenability for tracial von {N}eumann
  algebras.
\newblock {\em International Mathematics Research Notices}, 2021(4):2882--2918,
  2021.

\bibitem{BdP1981}
V.~Barbu and G.~{da Prato}.
\newblock Global existence for the {H}amilton-{J}acobi equations in {H}ilbert
  space.
\newblock {\em Annali della Scuola Normale Superiore di Pisa Classe di Scienze
  4e s{\'e}rie}, 8(2):257--284, 1981.

\bibitem{BdP1985b}
V.~Barbu and G.~{da Prato}.
\newblock Hamilton-{J}acobi equations in {H}ilbert spaces; variational and
  semigroup approach.
\newblock {\em Annali di Matematica pura ed applicata}, 142:303--349, 1985.

\bibitem{BdP1985a}
V.~Barbu and G.~{da Prato}.
\newblock A note on a {H}amilton-{J}acobi equation in {H}ilbert space.
\newblock {\em Nonlinear Analysis: Theory, Methods \& Applications},
  9(12):1337--1345, 1985.

\bibitem{IBY2008}
Ita{\"i} {Ben Yaacov}.
\newblock Topometric spaces and perturbations of metric structures.
\newblock {\em Log. Anal.}, 1:235–272, 2008.

\bibitem{BDLL2021}
Charles Bertucci, M{\'e}rouane Debbah, Jean-Michel Lasry, and Pierre-Louis
  Lions.
\newblock A spectral dominance approach to random matrices.
\newblock Preprint arXiv:2105.08983, 2021.

\bibitem{BV2001}
Philippe Biane and Dan-Virgil Voiculescu.
\newblock A free probability analogue of the {W}asserstein metric on the
  trace-state space.
\newblock {\em Geometric and Functional Analysis}, 11:1125--1138, 2001.

\bibitem{BrownOzawa2008}
Nathaniel~P. Brown and Narutaka Ozawa.
\newblock {\em $\mathrm{C}^*$-algebras and finite-dimensional approximations},
  volume~88 of {\em Graduate Studies in Mathematics}.
\newblock American Mathematical Society, Providence, 2008.

\bibitem{Capraro2010}
Valerio Capraro.
\newblock A survey on {C}onnes' embedding conjecture.
\newblock arXiv preprint at arXiv:1003.2076, 2010.

\bibitem{CDLL2019}
Pierre Cardaliaguet, Fran{\c c}ois Delarue, Jean-Michel Lasry, and Pierre-Louis
  Lions.
\newblock {\em The Master Equation and the Convergence Problem in Mean Field
  Games}, volume~2 of {\em Annals of Mathematics Studies}.
\newblock Princeton University Press, 2019.

\bibitem{CCP2020}
Ren{\'e} Carmona, Mark Cerenzia, and Aaron~Zeff Palmer.
\newblock The {D}yson and {C}oulomb games.
\newblock {\em Ann. Henri Poincar{\'e}}, 21:2897--2949, 2020.

\bibitem{CN2019}
Ian Charlesworth and Brent Nelson.
\newblock Free {S}tein irregularity and dimension.
\newblock arXiv:1902.02379, 2019.

\bibitem{Connes1976}
Alain Connes.
\newblock Classification of injective factors. {C}ases {$II_{1},$} {$II_{\infty
  },$} {$III_{\lambda },$} {$\lambda \not=1$}.
\newblock {\em Ann. of Math. (2)}, 104(1):73--115, 1976.

\bibitem{CrLi1985}
Michael~G. Crandall and Pierre-Louis Lions.
\newblock Hamilton-{J}acobi equations in infinite dimensions {I}. uniqueness of
  viscosity solutions.
\newblock {\em J. Funct. Anal.}, 62:379--396, 1985.

\bibitem{CrLi1986a}
Michael~G. Crandall and Pierre-Louis Lions.
\newblock Hamilton-{J}acobi equations in infinite dimensions {II}. existence of
  viscosity solutions.
\newblock {\em J. Funct. Anal.}, 65:368--405, 1986.

\bibitem{daSilva2018}
Ricardo~Correa {da Silva}.
\newblock Lecture notes on non-commutative {$L_p$}-spaces.
\newblock arXiv:1803.02390, 2018.

\bibitem{Dabrowski2017}
Yoann Dabrowksi.
\newblock A {L}aplace principle for {H}ermitian {B}rownian motion and free
  entropy {I}: the convex functional case.
\newblock arXiv:1604.06420, 2017.

\bibitem{Dabrowski2010}
Yoann Dabrowski.
\newblock A non-commutative path space approach to stationary free stochastic
  differential equations.
\newblock arxiv:1006.4351, 2010.

\bibitem{DGS2016}
Yoann Dabrowski, Alice Guionnet, and Dimitri Shlyakhtenko.
\newblock Free transport for convex potentials.
\newblock arXiv:1701.00132, 2016.

\bibitem{Dixmier1953}
Jacques Dixmier.
\newblock Formes lin{\'e}aires sur un anneau d'op{\'e}rateurs.
\newblock {\em Bulletin de la Soci{\'e}t{\'e} Math{\'e}matique de France},
  81:9--39, 1953.

\bibitem{Effros1981}
Edward~G. Effros.
\newblock {\em Dimensions and $\mathrm{C}^*$-algebras}, volume~46 of {\em CBMS
  Regional Conference Series in Mathematics}.
\newblock American Mathematical Society, Providence, 1981.

\bibitem{GMS2021}
Wilfrid Gangbo, Sergio Mayorga, and Andrzej {\'S}wi{\d e}ch.
\newblock Finite dimensional approximations of {H}amilton-{J}acobi-{B}ellman
  equations in spaces of probability measures.
\newblock {\em SIAM J. Math. Anal.}, 53(2):1320--1356, 2021.

\bibitem{GMMZ2021}
Wilfrid Gangbo, Alp{\'a}r M{\'e}sz{\'a}ros, Chenchen Mou, and Jianfeng Zhang.
\newblock Mean field games master equations with non- separable {H}amiltonians
  and displacement monotonicity.
\newblock arXiv preprint arXiv:2101.12362, 2021.

\bibitem{GM2021}
Wilfrid Gangbo and Alp{\'a}r~R M{\'e}sz{\'a}ros.
\newblock Global well-posedness of master equations for deterministic
  displacement convex potential mean field games.
\newblock arXiv preprint arXiv:2004.01660, 2021.

\bibitem{GaTu2018}
Wilfrid Gangbo and Adrian Tudorascu.
\newblock On differentiability in the {W}asserstein space and well-posedness
  for {H}amilton-{J}acobi equations.
\newblock {\em J. Math Pures Appl.}, 125:119--174, 2018.

\bibitem{Gromov1987}
Mikhael Gromov.
\newblock Hyperbolic groups.
\newblock In {\em Essays in group theory}, volume~8 of {\em Math. Sci. Res.
  Inst. Publ.}, page 75–263. Springer, New York, 1987.

\bibitem{GS2014}
Alice Guionnet and Dimitri Shlyakhtenko.
\newblock Free monotone transport.
\newblock {\em Inventiones Mathematicae}, 197(3):613--661, 09 2014.

\bibitem{HaMu2011}
Uffe Haagerup and Magdalena Musat.
\newblock Factorization and dilation problems for completely positive maps on
  von neumann algebras.
\newblock {\em Comm. Math. Phys.}, 303(2):555--594, 2011.

\bibitem{HaMu2015}
Uffe Haagerup and Magdalena Musat.
\newblock An asymptotic property of factorizable completely positive maps and
  the connes embedding problem.
\newblock {\em Comm. Math. Phys.}, 338(2):721--752, 2015.

\bibitem{HaSh2018}
Don Hadwin and Tatiana Shulman.
\newblock Tracial stability for $\mathrm{C}^*$-algebras.
\newblock {\em Integral Equations Operator Theory}, 90(1), 2018.

\bibitem{HJNS2021}
Benjamin Hayes, David Jekel, Brent Nelson, and Thomas Sinclair.
\newblock A random matrix approach to absorption in free products.
\newblock {\em International Mathematics Research Notices}, 2021(3):1919--1979,
  2021.

\bibitem{HPU2004}
Fumio Hiai, D{\'e}nes Petz, and Yoshimichi Ueda.
\newblock Free transportation cost inequalities via random matrix
  approximation.
\newblock {\em Probab. Theory Related Fields}, 130(2):199--221, 2004.

\bibitem{HU2006}
Fumio Hiai and Yoshimichi Ueda.
\newblock Free transportation cost inequalities for noncommutative
  multi-variables.
\newblock {\em Infinite Dimensional Analysis, Quantum Probability, and Related
  Topics}, 9:391--412, 2006.

\bibitem{Jekel2018}
David Jekel.
\newblock An elementary approach to free entropy theory for convex potentials.
\newblock {\em arXiv:1805.08814}, 2018.
\newblock To appear in Analysis and PDE Journal.

\bibitem{JekelExpectation}
David Jekel.
\newblock Conditional expectation, entropy, and transport for convex gibbs laws
  in free probability.
\newblock {\em International Mathematics Research Notices IMRN}, 2020, 2020.

\bibitem{JekelThesis}
David Jekel.
\newblock {\em Evolution equations in non-commutative probability}.
\newblock PhD thesis, University of California, Los Angeles, 2020.

\bibitem{JLS2021}
David Jekel, Wuchen Li, and Dimitri Shlyakhtenko.
\newblock Tracial non-commutative smooth functions and the free {W}asserstein
  manifold.
\newblock arXiv:2101.06572, 2021.

\bibitem{JNVWY2020}
Zhengfeng Ji, Anand Natarajan, Thomas Vidick, John Wright, and Henry Yuen.
\newblock {MIP*=RE}.
\newblock arXiv:2001.04383, 2020.

\bibitem{Jing2015}
Naihuan Jing.
\newblock Unitary and orthogonal equivalence of sets of matrices.
\newblock {\em Linear Algebra and its Applications}, 481:235--242, 2015.

\bibitem{JKO1998}
Richard Jordan, David Kinderlehrer, and Felix Otto.
\newblock The variational formulation of the {F}okker--{P}lanck equation.
\newblock {\em SIAM Journal on Mathematical Analysis}, 29(1):1--17, 1998.

\bibitem{Jung2007}
Kenley Jung.
\newblock Amenability, tubularity, and embeddings into {$\mathcal{R}^\omega$}.
\newblock {\em Math. Ann.}, 338(1):241--248, 2007.

\bibitem{Lafferty1988}
John~D. Lafferty.
\newblock The density manifold and configuration space quantization.
\newblock {\em Transactions of the American Mathematical Society},
  305(2):699--741, 1988.

\bibitem{LL1986}
J.M. Lasry and P.L. Lions.
\newblock A remark on regularization in {H}ilbert spaces.
\newblock {\em Israel J. Math.}, 55:257--266, 1986.

\bibitem{McDuff1969}
Dusa McDuff.
\newblock Uncountably many $\mathrm{II}_1$ factors.
\newblock {\em Ann. of Math. (2)}, 90(2):372--377, 1969.

\bibitem{CrLi1986b}
Pierre-Louis~Lions Michael G.~Crandall.
\newblock Hamilton-{J}acobi equations in infinite dimensions {III}.
\newblock {\em J. Funct. Anal.}, 68:214--247, 1986.

\bibitem{MvN1}
F.~J. Murray and J.~{von Neumann}.
\newblock On rings of operators.
\newblock {\em Annals of Mathematics}, 37(1):116--229, 1936.

\bibitem{MvN2}
F.J. Murray and J.~{von Neumann}.
\newblock On rings of operators, {II}.
\newblock {\em Trans. Amer. Math.Soc.}, 41:208--248, 1937.

\bibitem{MuRo2020a}
Magdalena Musat and Mikael R{\o}rdam.
\newblock Non-closure of quantum correlation matrices and factorizable channels
  that require infinite dimensional ancilla (with an appendix by {N}arutaka
  {O}zawa).
\newblock {\em Comm. Math. Phys.}, 375:1761--1776, 2020.

\bibitem{Nelson2015a}
Brent Nelson.
\newblock Free monotone transport without a trace.
\newblock {\em Communications in Mathematical Physics}, 334(3):1245--1298,
  2015.

\bibitem{Nelson2015b}
Brent Nelson.
\newblock Free transport for finite depth subfactor planar algebras.
\newblock {\em Journal of Functional Analysis}, 268(9):2586--2620, 2015.

\bibitem{Olshanskii1993}
A.~Yu. Olshanskii.
\newblock On residualing homomorphisms and g-subgroups of hyperbolic groups.
\newblock {\em Internat. J. Algebra Comput.}, 3(4):365--409, 1993.

\bibitem{Otto2001}
Felix Otto.
\newblock The geometry of dissipative evolution equations the porous medium
  equation.
\newblock {\em Communications in Partial Differential Equations},
  26(1-2):101--174, 2001.

\bibitem{OV2000}
Felix Otto and C{\'e}dric Villani.
\newblock Generalization of an inequality by {T}alagrand and links with the
  logarithmic {S}obolev inequality.
\newblock {\em Journal of Functional Analysis}, 173(2):361--400, 2000.

\bibitem{Ozawa2004}
Narutaka Ozawa.
\newblock There is no separable universal $\mathrm{II}_1$ factor.
\newblock {\em Proc. Amer. Math. Soc.}, 132:487--90, 2004.

\bibitem{Ozawa2013}
Narutaka Ozawa.
\newblock About the connes embedding conjecture: Algebraic approaches.
\newblock {\em Japanese Journal of Mathematics}, 8:147--183, 2013.

\bibitem{Paulsen2003}
Vern Paulsen.
\newblock {\em Completely Bounded Maps and Operator Algebras}.
\newblock Cambridge Studies in Advanced Mathematics. Cambridge University
  Press, 2003.

\bibitem{Popa1986injective}
Sorin Popa.
\newblock A short proof of ``injectivity implies hyperfiniteness'' for finite
  von neumann algebras.
\newblock {\em J. Operator Theory}, 16:261--272, 1986.

\bibitem{Procesi1976}
Claudio Procesi.
\newblock The invariant theory of $n \times n$ matrices.
\newblock {\em Advances in Mathematics}, 19:306--381, 1976.

\bibitem{ReedSimon1972}
Michael Reed and Barry Simon.
\newblock {\em Methods of Modern Mathematical Physics I: Functional Analysis}.
\newblock Academic Press, New York, 1972.

\bibitem{Sakai1971}
Sh{\^o}ichir{\^o} Sakai.
\newblock {\em $\mathrm{C}^*$-algebras and $\mathrm{W}^*$-algebras}, volume~60
  of {\em Ergebnisse der {M}athematik und ihrer {G}renzgebiete}.
\newblock Springer-Verlag, Berlin Heidelberg, 1971.

\bibitem{Shapiro1991}
Helene Shapiro.
\newblock A survey of canonical forms and invariants for unitary similarity.
\newblock {\em Lin. Algebra Appl.}, 147:101--167, 1991.

\bibitem{Shlyakhtenko2003}
Dimitri Shlyakhtenko.
\newblock Free fisher information for non-tracial states.
\newblock {\em Pacific J. Math}, 211:375--390, 2003.

\bibitem{Specht1940}
W.~Specht.
\newblock Zur theorie der matrizen, ii.
\newblock {\em Jahresber. Deutsch. Math.-Verein.}, 50:19--23, 1940.

\bibitem{TakesakiI}
Masamichi Takesaki.
\newblock {\em Theory of Operator Algebras I}, volume 124 of {\em Encyclopaedia
  of Mathematical Sciences}.
\newblock Springer-Verlag, Berlin Heidelberg, 2002.

\bibitem{TakesakiII}
Masamichi Takesaki.
\newblock {\em Theory of Operator Algebras II}, volume 125 of {\em
  Encyclopaedia of Mathematical Sciences}.
\newblock Springer-Verlag, Berlin Heidelberg, 2003.

\bibitem{TakesakiIII}
Masamichi Takesaki.
\newblock {\em Theory of Operator Algebras III}, volume 127 of {\em
  Encyclopaedia of Mathematical Sciences}.
\newblock Springer-Verlag, Berlin Heidelberg, 2003.

\bibitem{Ueda2011}
Yoshimichi Ueda.
\newblock Factoriality, type classification and fullness for free product von
  {N}eumann algebras.
\newblock {\em Adv. Math.}, 228(5):2647--2671, 2011.

\bibitem{Villani2008}
C\'edric Villani.
\newblock {\em Optimal Transport: Old and New}, volume 338 of {\em Grundlehren
  Der {M}athematischen {W}issenschaften}.
\newblock Springer, Berlin, 2009.

\bibitem{Voiculescu1985}
Dan-Virgil Voiculescu.
\newblock Symmetries of some reduced free product ${C}^*$-algebras.
\newblock In Huzihiro Araki, Calvin~C. Moore, {\c{S}}erban-Valentin Stratila,
  and Dan Voiculescu, editors, {\em Operator Algebras and their Connections
  with Topology and Ergodic Theory}, pages 556--588. Springer, Berlin,
  Heidelberg, 1985.

\bibitem{Voiculescu1986}
Dan-Virgil Voiculescu.
\newblock Addition of certain non-commuting random variables.
\newblock {\em Journal of Functional Analysis}, 66(3):323--346, 1986.

\bibitem{Voiculescu1991}
Dan-Virgil Voiculescu.
\newblock Limit laws for random matrices and free products.
\newblock {\em Inventiones mathematicae}, 104(1):201--220, Dec 1991.

\bibitem{VoiculescuFE1}
Dan-Virgil Voiculescu.
\newblock The analogues of entropy and {F}isher's information measure in free
  probability theory, {I}.
\newblock {\em Communications in Mathematical Physics}, 155(1):71--92, 1993.

\bibitem{VoiculescuFE2}
Dan-Virgil Voiculescu.
\newblock The analogues of entropy and of {F}isher's information measure in
  free probability theory, {II}.
\newblock {\em Inventiones Mathematicae}, 118:411--440, 1994.

\bibitem{VoiculescuFE3}
Dan-Virgil Voiculescu.
\newblock The analogues of entropy and of {F}isher's information measure in
  free probability theory, {III}: Absence of {C}artan subalgebras.
\newblock {\em Geometric and Functional Analysis}, 6:172--199, 1996.

\bibitem{VoiculescuFE5}
Dan-Virgil Voiculescu.
\newblock The analogues of entropy and of {F}isher's information measure in
  free probability theory {V}.
\newblock {\em Inventiones Mathematicae}, 132:189--227, 1998.

\bibitem{Voiculescu1998}
Dan-Virgil Voiculescu.
\newblock A strengthened asymptotic freeness result for random matrices with
  applications to free entropy.
\newblock {\em International Mathematics Research Notices}, 1998(1):41--63,
  1998.

\bibitem{VDN1992}
Dan-Virgil Voiculescu, Kenneth~J. Dykema, and Alexandru Nica.
\newblock {\em Free Random Variables}, volume~1 of {\em CRM Monograph Series}.
\newblock American Mathematical Society, Providence, RI, 1992.

\bibitem{Wiegmann1961}
N.~Wiegmann.
\newblock Necessary and sufficient conditions for unitary similarity.
\newblock {\em J. Austral. Math. Soc.}, 2:122--126, 1961.

\bibitem{Wilde2013}
Mark~M. Wilde.
\newblock {\em Quantum Information Theory}.
\newblock Cambridge University Press, 2013.

\bibitem{Witten2020}
Edward Witten.
\newblock A mini-introduction to information theory.
\newblock {\em La Rivista del Nuovo Cimento}, 43:187--227, 2020.

\end{thebibliography}
	
\end{document}